\documentclass[11pt]{amsart}
\usepackage[utf8]{inputenc}

\usepackage[english]{babel}
\usepackage{enumitem}
\usepackage{tikz-cd}
\usepackage{mathtools}
\usepackage{mathrsfs}
\usepackage{amsthm}
\usepackage{amssymb}
\usepackage{amsmath}
\usepackage{wasysym}
\usepackage[hidelinks]{hyperref}
\usepackage{graphicx}
\graphicspath{{immagini/}}
\usepackage{float}
\usepackage{framed}
\usepackage{fancybox}
\usepackage[bottom=4cm]{geometry}
 \usepackage[all]{xy}
 \usepackage{cleveref}
\usepackage{comment}

\newcommand{\namedthm}[3]{\theoremstyle{plain}
   \newtheorem*{thm#1}{#1 Theorem}\begin{thm#1}[#2]#3\end{thm#1}}

    \usepackage{caption}
\usepackage{subcaption}

\usepackage{tikz}
\usetikzlibrary{
  knots,
  hobby,
  decorations.pathreplacing,
  shapes.geometric,
  calc
}

\tikzset{
  knot diagram/every strand/.append style={
    ultra thick,
    red
  },
  show curve controls/.style={
    postaction=decorate,
    decoration={show path construction,
      curveto code={
        \draw [blue, dashed]
        (\tikzinputsegmentfirst) -- (\tikzinputsegmentsupporta)
        node [at end, draw, solid, red, inner sep=2pt]{};
        \draw [blue, dashed]
        (\tikzinputsegmentsupportb) -- (\tikzinputsegmentlast)
        node [at start, draw, solid, red, inner sep=2pt]{}
        node [at end, fill, blue, ellipse, inner sep=2pt]{}
        ;
      }
    }
  },
  show curve endpoints/.style={
    postaction=decorate,
    decoration={show path construction,
      curveto code={
        \node [fill, blue, ellipse, inner sep=2pt] at (\tikzinputsegmentlast) {}
        ;
      }
    }
  }
}

\usetikzlibrary{patterns}

\newcommand{\N}{\mathbb{N}}
\newcommand{\R}{\mathbb{R}}

\DeclareMathOperator{\id}{id}

\mathchardef\hyp="2D

\newcommand{\F}{\mathcal{F}}

\newcommand{\E}{\mathcal{E}}
\newcommand{\Etil}{\widetilde{\mathcal{E}}}
\newcommand{\W}{\mathcal{W}}
\newcommand{\Wtil}{\widetilde{\mathcal{W}}}
\newcommand{\Rcal}{\mathcal{R}}
\newcommand{\Rtil}{\widetilde{\mathcal{R}}}

\newcommand{\Stil}{\widetilde{\mathcal{S}}}

\newcommand{\Ttil}{\widetilde{\mathcal{T}}}
\newcommand{\ax}{\mathbf{a}}
\newcommand{\bx}{\mathbf{b}}
\newcommand{\D}{\mathcal{D}}

\newcommand{\Ftil}{\mathcal{F}}
\DeclareMathOperator{\Fix}{Fix}

\newtheorem{thm}{Theorem}[section]
\newtheorem{cor}[thm]{Corollary}
\newtheorem{lem}[thm]{Lemma}
\newtheorem{prop}[thm]{Proposition}
\newtheorem*{teo}{Theorem}
\newtheorem*{cor*}{Corollary}

\theoremstyle{definition}
\newtheorem{dfn}[thm]{Definition}
\newtheorem{rmk}[thm]{Remark}
\newtheorem{nota}[thm]{Notation}

\newtheorem{conv}[thm]{Convention}

\newtheoremstyle{named}{}{}{\itshape}{}{\bfseries}{.}{ }{#1 \thmnote{#3}}
\theoremstyle{named}

\newtheoremstyle{namedtheo}{}{}{\itshape}{}{\bfseries}{.}{.5em}{#1 \thmnote{#3}}
\theoremstyle{namedtheo}

\title{The Alexander and Markov theorems for strongly involutive links}
\author{Alice Merz}
\address{Université de Lille, Lille 59000, France}
\email{alice.merz@univ-lille.fr}

\begin{document}
\begin{abstract} The Alexander theorem (1923) and the Markov theorem (1936) are two classical results in knot theory that show respectively that every link is the closure of a braid and that braids that have the same closure are related by a finite number of operations called Markov moves.
    This paper presents specialized versions of these two classical theorems for a class of links in $S^3$ preserved by an involution, that we call strongly involutive links. When connected, these links are known as strongly invertible knots, and have been extensively studied. We develop an equivariant closure map that, given two palindromic braids, produces a strongly involutive link. We demonstrate that this map is surjective up to equivalence of strongly involutive links. Furthermore, we establish that pairs of palindromic braids that have the same equivariant closure are related by an equivariant version of the original Markov moves. 
\end{abstract}
\maketitle

\section{Introduction}

An oriented link $L$ in $S^3$ is {invertible} if there exists an orientation preserving homeomorphism $\tau$ of $S^3$ that preserves $L$ setwise inverting its orientation. Notice that $\tau$ is allowed to exchange the components of the link. 
When $\tau$ is an involution of $S^3$ with fixed points we will call the pair $ (L,\tau)$ a \emph{strongly involutive link}. Furthermore, if $\tau$ preserves the components of the link, the pair is called a \emph{strongly invertible link} and we call a strongly invertible link with one component a \emph{strongly invertible knot} (see \Cref{fig: sikintro}).
\begin{figure}
    \centering
    \includegraphics[width=0.15\textwidth]{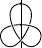}
    \caption{The knot in the picture is preserved setwise by the $\pi$-rotation around the vertical axis. Any orientation on the knot is reversed by this rotation.}
    \label{fig: sikintro}
\end{figure}
Observe that if $(L,\tau)$ is a strongly invertible link, then any possible orientation on $L$ is reversed by the involution $\tau$.

Strongly invertible knots have been a classical object of study (see for example \cite{Sakuma}), but in recent times there has been a strong new interest in the study of equivariant concordance for strongly invertible knots and their equivariant $4$-genus, for example in \cite{AlfieriBoyle}, \cite{BoyleIssa}, \cite{DaiMallickStoffregen}, \cite{DiPrisa}, \cite{LobbWatson}, \cite{MillerPowell}.

 In this paper we specialize two classical theorems in knot theory, namely the Alexander and Markov theorems, to strongly involutive links.
\newline

In this paper, unless otherwise stated, we will work in the PL category. In the non-equivariant setting it is a well-known fact that smooth and PL links are equivalent. In \Cref{app: equiv} we give a proof of the equivalence of these two settings in the equivariant case.
\newline

 Let $B_n$ denote the braid group in $n$ strands, and let $\mathcal{L}$ denote the set of all links up to isotopy.

In the classical case, the Alexander theorem \cite{Alexander} states that every link can be obtained as the closure of a braid or, in other words, that the closure map 
$$ \bigcup_{n \in \N} B_n \xrightarrow{\widehat{\cdot}} \mathcal{L}$$ is surjective.

This map is however far from being injective, as for any link $L$ there are infinitely many braids that have $L$ as their closure. A classical theorem of Markov \cite{Markov} defines two operations, called conjugation and stabilization, that generate an equivalence relation on the set of all braids so that the induced map on the quotient is a bijection, as shown in this commutative triangle:

$$ \begin{tikzcd}[row sep = large]
     \bigcup_{n \in \N} B_n \arrow[r] \arrow[d] & \mathcal{L} \\
     \left(\bigcup_{n \in \N} B_n\right)/ _\sim. \arrow[ru, "\simeq"'] &
\end{tikzcd}$$
\\
Before addressing the main results of this paper we state two different notions of equivalence on strongly involutive links. One, called Sakuma equivalence, is most commonly used in the literature about strongly involutive knots and links. The other one, called equivariant isotopy, is stronger than Sakuma equivalence, and will be frequently used throughout this paper. The positive resolution of the Smith conjecture together with a result by Lobb and Watson \cite{LobbWatson} will allow us to relate the two concepts.

\begin{dfn} \label{def: equivalencesintro}
We say that two PL strongly involutive links $(L_0,\tau_0)$ and $(L_1,\tau_1)$ are \emph{Sakuma equivalent} if there exists a PL homeomorphism $f:S^3 \to S^3$ such that $f(L_0)=L_1$ as oriented links and $\tau_0= f^{-1}\circ \tau_1 \circ f$. 
Two strongly involutive links $(L_0,\tau)$ and $(L_1,\tau)$ are \emph{equivariantly isotopic} if there exists a PL ambient isotopy of $S^3$ that sends $L_0$ to $L_1$ as oriented links and that commutes with $\tau$ at every time (notice that in this case the involution is the same). 
\end{dfn}
It is clear from the definition that two equivariantly isotopic strongly involutive links are also Sakuma equivalent. Viceversa, suppose that $(L_0,\tau_{0})$ and $(L_1,\tau_{1})$ are two strongly involutive links, Sakuma equivalent via the PL homeomorphism $f:S^{3} \to S^{3}$. By the positive resolution of the Smith conjecture (\cite{Waldhausen},\cite{MorganBass}), we know that any two orientation preserving involutions of $S^{3}$ with fixed points are conjugated and their fixed point set is an unknot. As a consequence we can suppose without loss of generality that $\tau_{0 }= \tau_{1}= \tau$ and we call the fixed point set $\ax$. Then $f$ commutes with $\tau$ and therefore fixes $\ax$ setwise. It follows from the proof of Proposition 2.4 in \cite{LobbWatson} that when $f$ preserves the orientation on $\ax$ the links are also equivariantly isotopic. Notice that in \cite{LobbWatson} actually work in the smooth settings but by the equivalence shown in \Cref{app: equiv} the same holds in the piecewise linear case.

In this paper we specialize the Alexander theorem to strongly involutive links (and therefore, to strongly invertible links), showing that every strongly involutive link is the closure of a braid with a particular symmetry.

Before stating the Equivariant Alexander Theorem, we want to introduce some terminology and explain what we mean by \emph{equivariant closure}.

Palindromic braids are, roughly speaking, those which are the same when read from left to right or from right to left when written in the Artin generators of the braid group
(see \Cref{def: palindromic} for a precise definition). Let $P_n \subset B_n$ denote the set of palindromic braids with $n$ strands. 

It is a consequence of a theorem of Deloup (\cite{Deloup}) that a palindromic braid $\alpha$ always admits a representative $a\subset [0,1] \times D^2$ in its isotopy class that is fixed by the involution $\rho$ that consists of a $\pi$-rotation around the axis $ \{(1/2,t,0)\in [0,1] \times D^2 \, |\, |t|\le 1$\} (see \Cref{fig:palindromicnaive}).

\begin{figure}
    \centering
    \includegraphics[width= 0.25 \textwidth]{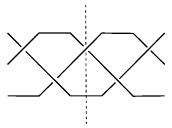}
    \caption{A palindromic braid admits a representative in its isotopy class that is fixed by the $\pi$-rotation around the vertical axis in the picture.}
    \label{fig:palindromicnaive}
\end{figure}
We show that any two such representatives of the same braid are equivalent, as stated in the following theorem:
\begin{teo}[see \Cref{thm: uniqueness}]
    Let $\alpha$ be a palindromic braid and let $a,a'\subset I \times D^2$ be two representatives such that $\rho(a)=a$ and $\rho(a')=a'$, where $I$ denotes the interval $[0,1]$. Then there exists a PL homeomorphism $H:I \times D^2 \to I \times D^2$, such that $H(a)=a'$, $H \circ \rho = \rho \circ H$ and $H_{|\partial (I \times D^2)}= \id_{\partial (I \times D^2)}$.
\end{teo}
Let $\mathcal{SIL}$ denote the set of strongly involutive links up to Sakuma equivalence.
We call \emph{equivariant closure map} the function
\begin{figure}
    \centering
    \includegraphics[width= 0.45 \textwidth]{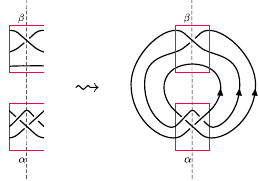}
    \caption{The equivariant closure map. The involution is given by a $\pi$-rotation around the dotted axis.}
    \label{fig:equivariantclosure}
\end{figure}
$$ \bigcup_{n \in \N} P_n \times P_n \to \mathcal{SIL}$$
that takes in input two palindromic braids $\alpha$ and $\beta$ and returns a strongly involutive link whose underlying link is $\widehat{\alpha\beta}$
and whose involution is depicted in \Cref{fig:equivariantclosure}. We refer to \Cref{sec: palindromi} for a precise definition of equivariant closure.
This map is well-defined as a consequence of \Cref{thm: uniqueness}.
We are finally ready to state the Equivariant Alexander Theorem.
\namedthm{Equivariant Alexander}{see \Cref{thm: Alexander}}
{The equivariant closure map is surjective.}

In other words, every strongly involutive link is Sakuma equivalent to the equivariant closure of two palindromic braids.

As in the classical case, the equivariant closure map is far from being injective and therefore we prove an equivariant version of Markov's theorem. We define a finite set of operations on $\bigcup_{n \in \N}P_{n} \times P_{n}$ that generate an equivalence relation so that the map induced by equivariant closure on the quotient is a bijection, as shown in the following commutative diagram:

$$ \begin{tikzcd}[row sep = large]
    \bigcup_{n \in \N} P_n \times P_{n} \arrow[r] \arrow[d] & \mathcal{SIL} \\
    \left(\bigcup_{n \in \N} P_n \times P_{n}\right) /_\sim. \arrow[ru, "\simeq"'] &
\end{tikzcd}$$
More precisely, we prove the following theorem:

\namedthm{Equivariant Markov}{see \Cref{thm: eqmarkov}}{
    Let $(\widehat{\alpha\beta},\tau)$ and $(\widehat{\gamma\delta},\tau)$ be two strongly involutive links. Then they are Sakuma equivalent if and only if the pairs $(\alpha, \beta)$ and $(\gamma, \delta)$ are related by a finite sequence of the following operations and their inverses:
    \begin{itemize}
        \item equivariant conjugation, i.e. changing $(\alpha,\beta)$ with $ (\varepsilon\alpha \overline{\varepsilon}, \overline{\varepsilon}^{-1}\beta {\varepsilon}^{-1})$ for any braid $\varepsilon$. Here $\overline{\varepsilon}$ indicates the braid obtained by ``reading from right to left" a word for $\varepsilon$ in the Artin generators of the braid group (see \Cref{def: palindromic} for details);
        \item stabilization on a fixed point;
        \item stabilization on a fixed point through $\infty$;
        \item double stabilization;
        \item exchanging the components of a pair, i.e. substituting the pair $(\alpha,\beta)$ with $(\beta,\alpha)$.
    \end{itemize}
} 

To see the effects of the operations introduced in the previous theorem, see Figures \ref{fig:eqconjugateintro}, \ref{fig:stabfixintro}, \ref{fig:stabfixinftyintro} and \ref{fig:stabdoubleintro}. For the precise definitions of said moves we refer to \Cref{sec: Markov}.

\begin{figure}
    \centering
    \includegraphics[width=0.2 \textwidth]{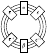}
    \caption{The effect of equivariant conjugation on a strongly involutive link. }
    \label{fig:eqconjugateintro}
\end{figure}
\begin{figure}
        \centering
        \includegraphics[width=0.4 \textwidth]{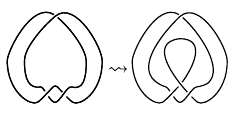}
        \caption{An example of stabilization on a fixed point.}
        \label{fig:stabfixintro}
    \end{figure}
    \begin{figure}
        \centering
        \includegraphics[width=0.4 \textwidth]{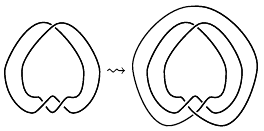}
        \caption{An example of stabilization on a fixed point through $\infty$.}
        \label{fig:stabfixinftyintro}
    \end{figure}

    \begin{figure}
        \centering
        \includegraphics[width=0.4 \textwidth]{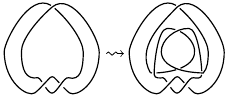}
        \caption{An example of double stabilization.}
        \label{fig:stabdoubleintro}
    \end{figure}

\textbf{Outline of the paper:}
This paper is organised as follows: 
in \Cref{sec: firstdef} we introduce some of the necessary notation and we establish some preliminary results.

In \Cref{sec: palindromi} we present palindromic braids. We show that $\rho$-invariant representatives of a palindromic braid are equivalent (\Cref{thm: uniqueness}) and, as a consequence, that the equivariant closure map is well-defined.

In \Cref{sec: Alex} we prove the Equivariant Alexander Theorem.

Finally, in \Cref{sec: Markov} we prove the Equivariant Markov Theorem. The proof is divided in four steps and is an adaptation to the equivariant case of the proof by Birman in \cite{Birman}.

In \Cref{app: equiv} we provide a proof of the equivalence of the piecewise linear and smooth setting.

In \Cref{app: examples} we provide some low crossing number strongly invertible knots in equivariant braid closure form.
\newline

 After completing this project, I encountered the papers \cite{CouturePerron} and \cite{Couture}, that come to very similar conclusions as the ones in this paper, although their results are set in a quite different framework, coming from the study of divides (introduced in \cite{Acampo}).
In fact in \cite{CouturePerron} the authors show that every link coming from an ordered Morse signed divide is the closure of a braid that is the composition of two palindromic braids, while \cite{Couture} shows that every strongly invertible link comes from an ordered Morse signed divide. Moreover in \cite{Couture} it is shown that two strongly invertible links are equivariantly equivalent if and only if the corresponding ordered Morse signed divides are related by a set of moves that resemble the equivariant Markov moves. However it is not obvious (although expected) that the involution of a strongly invertible link and the involution on the closure of the associated braid found in \cite{CouturePerron} give rise to equivalent strongly invertible links.

\textbf{Acknowledgements.} This paper is part of the author's PhD thesis. The author would like to thank her advisor Paolo Lisca for his support and the comments on the first version of this paper. 
Moreover, she would also like to thank the referees of her PhD thesis, Keegan Boyle and David Cimasoni, for their much appreciated feedback and corrections, which significantly contributed to enhancing the clarity and quality of this paper. The author is affiliated with INdAM-GNSAGA. The author acknowledges the MIUR Excellence Department Project awarded to the Department of Mathematics, University of Pisa, CUP I57G22000700001.

%\section{An equivariant version of Alexander and Markov theorems}
%\chaptermark{The equivariant Alexander and Markov theorems}

\section{Equivariant triangle moves and PL isotopy of strongly involutive links} \label{sec: firstdef}

Let $L$ be an oriented link in $S^3$. 
Recall that all links are supposed to be PL. In this section we introduce some operations on strongly involutive links called equivariant triangle moves. We show that two strongly involutive links are isotopic if and only if they are related by a finite number of admissible equivariant triangle moves. Finally we prove \Cref{lem: convex} that will help us establish when an equivariant triangle move is admissible.
\newline

We think of links as a set of vertices in $\R^3$ connected by straight edges, in particular we allow three or more consecutive vertices to be aligned.

%\begin{figure}
 %   \centering
 %   \includegraphics[width= 0.4 \textwidth]{siknots.pdf}
  %  \caption{Two knots preserved by the rotation of $\pi$ around the dotted axes.}
   % \label{fig:siknots}
%\end{figure}

We fix the following notation:

\begin{nota}
We think of $S^3$ as $\R^3 \cup \{\infty\}$. Given a finite set of points $p_1,\ldots,p_m \in \R^3$ we denote by $[p_1,\ldots,p_m]$ the convex hull in $\R^3$ of these points. If $m=2$, $[p_1,p_2]$ denotes a segment and $p_1p_2$ denotes the oriented segment from $p_1$ to $p_2$.
\end{nota}

\begin{figure}
    \centering
    \includegraphics[width= 0.8 \textwidth]{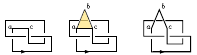}
    \caption{The move $\E_{a,c}^b$.}
    \label{fig:Emove}
\end{figure}

\begin{dfn} \label{def: admissible}
Let $a,c$ be consecutive vertices in $L$, and let $b$ be another point in $\R^3$. When $L'=(L \setminus [a,c]) \cup [a,b]\cup[b,c]$ is a link, we call it $\mathcal{E}^b_{a,c}(L)$ (see \Cref{fig:Emove}). \\
When $[a,b,c]\cap L= [a,c]$ the move $\E_{a,c}^b$ is said to be \emph{admissible}. 
\end{dfn}
Observe that when a move $ \E_{a,c}^b$ is admissible then the links $L$ and $\E_{a,c}^b(L)$ are PL isotopic. 
Viceversa it is a well-known fact that every two links are isotopic if and only if they are obtained from each other by a finite sequence of admissible $\mathcal{E}$ moves and their inverses.

Now we would like to state a similar result for strongly involutive links. Notice however that while in the classical setting isotopies of links can be always assumed to avoid the point at infinity, in the equivariant setting this is not possible. 

As a consequence we introduce as well the $\E$ moves that go through $\infty$:

\begin{conv}
    Let $a,c$ be consecutive vertices in $L$ and let $b$ be a third unaligned point in $\R^3\setminus L$. Denote by $\alpha$ the plane in $\R^3$ that contains the points $a,b,c$. When $\alpha \cap (L \setminus [a,c])$ is contained in the interior of $[a,b,c]$ we add to the set of admissible moves in \Cref{def: admissible} the move $\E_{a,c}^b$. We will refer to this move as an $\E$ move going through $\infty$. For an example in the equivariant setting, see \Cref{fig:Einfinity}.
\end{conv}

Let $(L,\tau)$ be a strongly involutive link. By the positive resolution of the Smith conjecture (\cite{Waldhausen},\cite{MorganBass}), up to conjugation we can and will suppose that $\tau|_{\R^3}$ is the $\pi$-rotation around a fixed axis, and $\tau(\infty)=\infty$. We will denote the fixed point set of $\tau$ by $\ax$.
 Notice that $\ax$ is an unknotted circle in $S^3$.

Notice that up to equivariant isotopy (see \Cref{def: equivalencesintro}) we can always suppose that a strongly involutive link $L$ is contained in $\R^3$.
We now define an equivariant version of the $\E$ move.

\begin{dfn}
    \label{def: admissibleEtil}
Let $a,c$ be consecutive vertices on a strongly involutive link $(L,\tau)$. Let $b \in \R^3$ be another point. We define the move $\Etil_{a,c}^b \coloneqq  \E_{\tau(a),\tau(c)}^{\tau(b)}\circ \E_{a,c}^b$ whenever the right-hand side is well-defined and the moves $\E_{\tau(a),\tau(c)}^{\tau(b)}$ and $ \E_{a,c}^b $ do not go through $\infty$. \\
When $\E_{a,c}^b$ is admissible on $L$, and $\E_{\tau(a),\tau(c)}^{\tau(b)}$ is admissible on $ \E_{a,c}^b(L)$ the move $\Etil_{a,c}^b$ is said to be \emph{admissible}. 
\end{dfn}

\begin{rmk} \label{rmk: asse}
Given a strongly involutive link $(L,\tau)$, remark that $(\Etil_{a,c}^b(L),\tau )$ is also strongly involutive. \\
   It is not hard to see that $\Etil_{a,c}^b $ is admissible if and only if $\E_{a,c}^b $ is admissible and $[a,b,c]\cap \tau([a,b,c]) = [a,c]\cap \tau([a,c])$. 
   The set on the right-hand side is either empty (\Cref{fig:Etilmove}) or is equal to one of the endpoints of $[a,c]$ when this point lies in $\ax$.
\end{rmk}

\begin{lem} \label{lem: equivalence}
Two strongly involutive links $(L,\tau)$, $(L',\tau)$ are equivariantly isotopic if and only if $L'$ is obtained from $L$ by a finite sequence of moves of the following types (and their inverses):
\begin{enumerate}
    \item \label{mossa uno}admissible $\widetilde{\mathcal{E}}_{a,c}^b$ moves (\Cref{fig:Etilmove});
    \item \label{mossa due}admissible $\E_{a,\tau(a)}^b$ moves with $b=\tau(b)$ (\Cref{fig:Enotinfinity}, \Cref{fig:Einfinity}).
\end{enumerate}
\end{lem}
We will sometimes refer to these moves as equivariant triangle moves.

\begin{figure}
    \centering
    \includegraphics[width = 0.6 \textwidth]{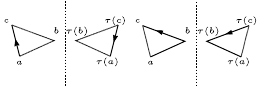}
    \caption{The move $\Etil_{a,c}^b$.}
    \label{fig:Etilmove}
\end{figure}

\begin{figure}
    \centering
    \includegraphics[width= 0.5 \textwidth]{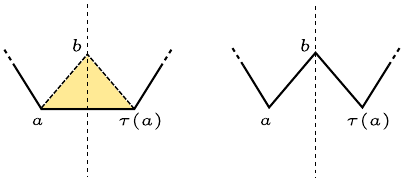}
    \caption{The move $\E_{a,\tau(a)}^b$. Notice that the yellow triangle is $\tau$-invariant.}
    \label{fig:Enotinfinity}
\end{figure}

\begin{figure}
    \centering
    \includegraphics[width= 0.5 \textwidth]{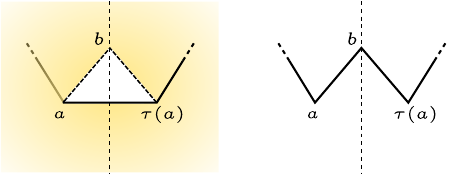}
    \caption{The move $\E_{a,\tau(a)}^b$ going through $\infty$.}
    \label{fig:Einfinity}
\end{figure}

The following proof is an adaptation of the proof of Proposition 1.10 in the book of Burde and Zieschang \cite{BurdeZieschang}.

\begin{proof}
It is easy to see that if $L$ and $L'$ are related by a finite sequence of equivariant triangle moves then there is an equivariant isotopy that sends $L$ to $L'$. 
 Viceversa, suppose that $H:S^3 \times I \to S^3$ is a (PL) $\tau$-equivariant ambient isotopy such that $H_0 = \id_{S^3}$ and $H_1(L)=L'.$
 Recall that $\ax$ is the fixed point set of $\tau$.
 First of all we prove the statement in the case there exists a point $p$ in $\ax \cap \R^3$ such that $H(p,[0,1]) \cap H(L,[0,1]) = \emptyset$. One can see that there exists an equivariant isotopy $G$ of $S^3$ sending $H_1(p)$ to $p$ which is supported in a neighbourhood of the arc in $\ax$ disjoint from $L$ connecting $p$ and $H_1(p)$. Moreover we can suppose that $ G_1\circ H_1$ fixes a neighbourhood $U$ of $p$. Let $T$ be a homothety of $\R^3$ centered on a point in $\ax \cap U$ such that $T(L)\subset U$. 
 There is an obvious isotopy between $\id_{\R^3}$ and $T$ that is a homothety at every time. The trace of $L$ along this isotopy can be easily tessellated by triangles in an equivariant way.
 As a consequence, $L$ and $T(L)$ are related by a finite sequence of equivariant triangle moves. Now we show that $T(L)$ and $L'$ are related by the same operations which will imply the statement. In fact, $G_1\circ H_1$ being a PL function, it maps the triangle moves between $L$ and $T(L)$ to (a composition of) triangle moves between $G_1\circ H_1(L)=L'$ and $G_1\circ H_1(T(L))=T(L) $, possibly going through $\infty$.

 In general, one can find $t_0,\ldots,t_k$ such that $0=t_0<\ldots <t_k=1$ and $p_i \in \ax \setminus \{\infty\}$ for $i =0,\ldots,k$ such that $H(p_i,[t_i,t_{i+1}]) \cap H(L,[t_i,t_{i+1}]) = \emptyset$. Therefore one can apply the previous reasoning to this sequence of isotopies to conclude the proof.
\end{proof}

From now on we suppose that all PL strongly involutive links under consideration have the points in $L \cap \ax$ as vertices. 
Notice that by definition a move $\E_{a,\tau(a)}^b$ can only be applied on consecutive vertices $ a$ and $\tau(a)$.
Since from now on all our links have all points in $L \cap \ax$ as vertices, these moves cannot be applied on our links. We thus introduce a substitute for the $\E$ moves, that we call $\D$ move.

\begin{dfn} \label{def: admissibleD}
Let $a,c$ be consecutive vertices in $L$, and let $b$ be another point in $\ax$. The quadrilateral $Q$ with vertices $a,b,c,\tau(a)$ lies in a plane $\alpha$. When $L'=\big(L \setminus ([a,c]\cup [c,\tau(a)])\big) \cup ([a,b]\cup[b,\tau(a)])$ is a link, we call it $\mathcal{D}^b_{a,c}(L)$. \\
When either $Q \cap L= [a,c]\cup [c,\tau(a)]$ or the closure in $\R^3$ of $(\alpha \setminus Q) \cap L$ is equal to $ [a,c]\cup [c,\tau(a)]$, $\D^b_{a,c}$ is said to be \emph{admissible} (\Cref{fig:Dmove}). In the second case we say that this is a $\D$ move going through $\infty$. 
\end{dfn}

\begin{figure}
    \centering
    \includegraphics[width=0.5 \textwidth]{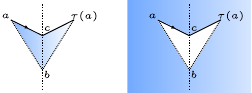}
    \caption{The quadrilateral $Q$ (left) and $\alpha \setminus Q$ (right) in the definition of the move $\D$.}
    \label{fig:Dmove}
\end{figure}

The following is a straightforward consequence of \Cref{lem: equivalence}.

\begin{lem}
    Two strongly involutive links $L,L'$ are equivariantly isotopic if and only if they are related by a finite sequence of $\Etil$ moves, $\D$ moves and their inverses.
\end{lem}

The following lemma will help us establish when a $\Etil$ move is admissible.

\begin{lem} \label{lem: convex}
Let $K$ be a compact convex subset of $\R^3 \subset S^3$. Then $K \cap \tau (K) \neq \emptyset$ if and only if $K \cap \mathbf{a}\neq \emptyset$.
\end{lem}

\begin{proof}

Of course when $K\cap \mathbf{a}\neq \emptyset$, since $\ax$ is the set of fixed points of $\tau$, then $K \cap \tau (K) \neq \emptyset $. Viceversa, if $K \cap \tau (K) \neq \emptyset$, let $x \in K \cap \tau (K)$. Then $\tau(x) \in K$ and since $K$ is convex the segment $[x,\tau(x)] \subset K$. But $[x,\tau(x)]\cap \ax\neq \emptyset$ and as a consequence $ K \cap \ax \neq \emptyset$.
\end{proof}

\begin{rmk} \label{rmk: admissibleEtil}
    As a consequence of \Cref{rmk: asse} and \Cref{lem: convex}, the move $\Etil_{a,c}^b$ on a strongly involutive link $L$ is admissible if and only if $[a,b,c]\cap L=[a,c]$ and $[a,b,c] \cap \ax$ is either empty or it is an endpoint of $[a,c]$ lying in $\ax$.
\end{rmk}

\section{Palindromic braids and the equivariant closure map} \label{sec: palindromi}
In view of the proof of the equivariant version of Alexander theorem, in this section we define the equivariant closure map. This will take as input two braids $\alpha$ and $\beta$ with a particular symmetry, called palindromic braids, and has a strongly involutive link as output, whose underlying link is $\widehat{\alpha\beta}$ and that is in equivariant braid position with respect to the unknotted circle $\bx=\{(0,0,z) \in \R^3\} \cup \{\infty\}$ contained in $\R^3 \cup \{\infty\} =S^3$.

Let us start by clarifying what we mean by equivariant braid position. 
Recall first that in the non-equivariant setting, an oriented link $L$ is said to be in braid position with respect to an unknot $A$ if the link $L$ can be obtained, up to ambient isotopy with support in the exterior of $A$, as the closure of a braid with respect to the axis $A$ as described for example in Chapter 2 of \cite{BurdeZieschang}.

\begin{dfn}
Let $(L,\tau)$ be a strongly involutive link that is in braid position with respect to some axis $A$. Suppose that $\tau$ preserves $A$ reversing its orientation. In this case we say that $(L,\tau)$ is in \emph{equivariant braid position} with respect to $A$.
\end{dfn}

We now define what are palindromic braids and show that they have some symmetries.

\begin{dfn} \label{def: palindromic}
Let $\alpha = \sigma_{i_1}^{a_{1}}\cdots \sigma_{i_m}^{a_{m}}$ be a $n$-braid, where $\sigma_i$ is the $i$-th Artin generator of the braid group. 
Then the \emph{reverse} 
of $\alpha$ is $\overline{\alpha}=\sigma_{i_m}^{a_{m}} \cdots \sigma_{i_1}^{a_{1}}.$ If $\alpha= \overline{\alpha}$ the braid is said to be \emph{palindromic}.
\end{dfn}

For every $n \in \N$, let $p_1,\ldots,p_n \in  D^2$ be $n$ points.
Recall that $n$-braids can be defined as $1$-submanifolds of $I \times D^2$ with no closed components, such that every connected component embeds strictly monotonically with respect to the $I$ component, and with boundary equal to $ \bigcup_{i=1}^n \{0,1\} \times\{p_i\}  $, up to ambient isotopy fixing the boundary of $I \times D^2$.

We will always assume that the $p_i$'s lie in $\{(t,0) \,| \, |t| <1\} \subset D^2$.
\begin{rmk} \label{rmk: palindromicity}
As a consequence of Proposition 2.9 in \cite{Deloup}, a palindromic braid $\alpha$ can be written as $\alpha= \gamma\Delta \overline{\gamma}$, where $\gamma$ is some braid, and $\Delta=\sigma_{j_1}^{\epsilon_1} \cdots \sigma_{j_k}^{\epsilon_k}$, with $\epsilon_i = \pm 1$ and $|j_i-j_l|>2$ for all $i,l$ (i.e. the $\sigma_{j_i}$ commute). As a consequence, every palindromic braid has a representative $a$ contained in $I \times D^2$ such that $a= \rho(a)$, where $$\rho: I \times D^2 \to I \times D^2$$
sends $(x,(y,z)) \in I \times D^2$ to $(1-x, (y , -z))$. Roughly speaking, $\rho$ is a $\pi$-rotation around the segment $\{ (1/2,(t,0))\, |\, |t|<1 \}\subset I \times D^2 $. We will call such a representative for $a$ a \emph{symmetric representative}.

Viceversa, if a braid $\alpha$ has a symmetric representative then it is easy to see that $\alpha= \gamma \Delta \overline{\gamma}$ for some $\gamma$ and $\Delta$ as above and is therefore palindromic (see \Cref{fig:palindromicbraid}).
\end{rmk}

Given two palindromic $n$-braids $\alpha$ and $\beta$, let $a$ and $b$ be symmetric representatives of these braids as in \Cref{rmk: palindromicity}. Fix $\tau:S^3\to S^3$ to be the $\pi$-rotation around an unknotted circle $\ax$ as in \Cref{sec: firstdef}. We will always suppose from now on that $\ax \setminus \{\infty\}$ is equal to the $y$-axis in $\R^3$. Then the equivariant closure of $a$ and $b$ is the strongly involutive link $(\widehat{ab},\tau)$ constructed as in \Cref{fig:eqbraidclosure}. We now state a precise definition.

  Consider $B_1=\{(x,y,z) \in \R^3 \,|\, |x|<1/2, \|(y+2,z)\|<1 \}$ and $B_2=\{(x,y,z) \in \R^3 \,|\, |x|<1/2, \|(y-2,z)\|<1 \}$. There is an obvious identification of $B_i$ with $I \times D^2$ for $i=1,2$, since the two sets are translates along $(-1/2,\pm 2, 0)$ of $I \times D^2 \subset \R^3$. Use these identifications to embed $a$ inside $B_1$ and $b$ inside $B_2$.

\begin{dfn} \label{dfn: eqclosure}
    With the notation stated above and \Cref{fig:eqbraidclosure} in mind, the \emph{equivariant closure} of $a$ and $b$ is the strongly involutive link $(L,\tau)$, where $L$ is obtained by embedding $a$ and $b$ in $B_1$ and $B_2$ respectively and then connecting their endpoints as prescribed: for $j=1,\ldots, n$, connect $\{-1/2\} \times\{p_i+(2,0)\}$ to $\{-1/2\} \times\{p_i-(2,0)\} $  with arcs $\gamma_i$ inside the $(x,y)$-plane that project monotonically to the $y$-axis. Connect $\{1/2\} \times\{p_i+(2,0)\}$ to $\{1/2\} \times\{p_i-(2,0)\} $  with the arc $\tau(\gamma_i)$. 
    We orient $L$ so that every component rotates counterclockwise with respect to $0$ in the projection onto the $(x,y)$-plane. 
\end{dfn}

Notice that in the previous definition, the underlying link of the equivariant closure of $a$ and $b$ would be the closure of ${a\overline{b}}$. However, since $a$ and $b$ are symmetric representatives of palindromic braids, this is equal to the closure of ${a{b}}$.

\begin{rmk}
    Observe that the resulting strongly involutive link is indeed in equivariant braid position with respect to $\bx = \{(0,0,z) \in \R^3\} \cup \{\infty\}$.
\end{rmk}

\begin{figure}
    \centering
    \includegraphics[width= 0.3 \textwidth]{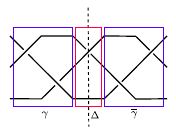}
    \caption{Palindromic braids are exactly those braids that are preserved by the $\pi$-rotation around the dotted axis.}
    \label{fig:palindromicbraid}
\end{figure}

\begin{figure}
    \centering
    \includegraphics[width= 0.45 \textwidth]{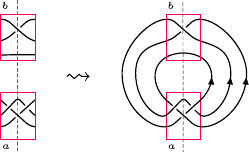}
    \caption{The equivariant closure of symmetric representatives of two palindromic braids. The involution $\tau$ is the $\pi$-rotation around the dotted vertical axis.}
    \label{fig:eqbraidclosure}
\end{figure}

The rest of this section is devoted to showing that the equivariant closure map does not depend, up to equivariant isotopy, on the choice of symmetric representative for a palindromic braid.

The following theorem shows that all possible symmetric representatives of a palindromic braid are actually equivariantly isotopic relative to the boundary.

\begin{thm} \label{thm: uniqueness}
Let $\alpha$ be a palindromic braid and let $a,a'\subset I \times D^2$ be two representatives such that $\rho(a)=a$ and $\rho(a')=a'$. Then there exists a PL homemorphism $H:I \times D^2 \to I \times D^2$, such that $H(a)=a'$, $H \circ \rho = \rho \circ H$ and $H_{|\partial (I \times D^2)}= \id_{\partial (I \times D^2)}$.
\end{thm}

To prove \Cref{thm: uniqueness} we need to establish an intermediate result.

\begin{prop} \label{prop: handlebody}
    Let $H_g$ be a handlebody of genus $g$ and let $\tau:H_g \to H_g$ be an orientation preserving involution of $H_g$ with fixed points and such that the fixed point set has no closed components. Suppose that $\psi:H_g\to H_g$ is a PL homeomorphism that is $\tau$-equivariant on the boundary, i.e. such that $\tau \circ \psi_{|\partial H_g} = \psi \circ \tau_{|\partial H_g}$. Then there exists a $\tau$-equivariant PL homeomorphism $\widetilde{\psi}:H_g \to H_g$, 
    such that $\widetilde{\psi}_{|\partial H_g}=\psi_{|\partial H_g}$.
\end{prop}
\begin{proof}
As a consequence of results on the involutions of handlebodies obtained in \cite{McCulloughMillerZimmermann} (see also \cite{PantaleoniPiergallini}) we can suppose that $g= g_1+2 g_2$ and $\tau$ is the involution described in \Cref{fig:handlebody1}.

Call $a_1,\ldots,a_{g_1+1}$ the arcs fixed by $\tau$ in $H_g$. Let $\gamma_1,\ldots,\gamma_{g_1+1}$ be oriented arcs in the boundary of $H_g$ obtained by pushing the arc $a_i$ relative to the boundary and such that $c_i=\gamma_i \cup \tau(-\gamma_i)$ is the boundary of a compressing disk $D_i$ for $H_g$ (\Cref{fig:handlebody1}).

Notice that since $\psi$ is $\tau$-equivariant on the boundary, the set of fixed points of $\tau$ on the boundary is preserved by $\psi$. 

Observe as well that the endpoints of an arc $a_i$ (and therefore the endpoints of $\gamma_i$) are sent by $\psi$ to the endpoints of an arc $a_j$ for some $j$. In fact if that was not the case, the endpoints of $a_i$ would be sent to the endpoints of two different segments, $a_j$ and $a_k$ with $j \neq k$. On the one hand we know that the algebraic intersection number of $\psi(c_i)$ and $D_j$ must be $0$ since $\psi(c_i)$ bounds a disk in $H_g$. On the other hand, this intersection number would be equal to $n_1 + n_2 \pm 1$, where $n_1$ is the intersection number of the interior of $\psi(\gamma_i)$ with $c_j$ and $n_2$ is the intersection number of the interior of $\psi(\tau(-\gamma_i))=\tau(\psi(-\gamma_i))$ with $c_j$. We need to add $\pm 1$ coming from the intersection with $c_j$ on the boundary of $\psi(\gamma_i)$. One just has to notice that $n_1 = n_2$ since $\tau(c_j)=-c_j$, and therefore the endpoints of $\psi(a_i)$ must be the endpoints of some $a_{j_i}$.

Observe that there is a set of generators $x_1,\ldots,x_g$ of $\pi_1(H_g)$ such that $\tau_*(x_i)=x_i^{-1}$ for $i\le g_1$, $\tau_*(x_i)=x_{i+g_2}$ for $g_1 <i\le g_1+g_2 $ and $ \tau_*(x_i)=x_{i-g_2}$ for $ g_1 +g_2<i\le g$ (see \Cref{fig:handlebody2}). 

Notice that $\tau_*$ is an automorphism of a finitely generated free group that preserves the length of the words in the generators listed above. 

Then, since no generator is fixed by $\tau_*$ it follows easily that $\tau_*$ has no nontrivial fixed point. 

\begin{figure}
    \centering
    \includegraphics[width=0.5 \textwidth]{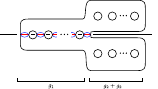}
    \caption{The handlebody $H_g$. The involution $\tau$ is given by rotation of $\pi$ around the horizontal axis. The curves $c_i$ are blue, the arcs $a_i$ are in red.}
    \label{fig:handlebody1}
\end{figure}

\begin{figure}
    \centering
    \includegraphics[width= 0.5 \textwidth]{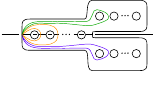}
    \caption{The generators $x_i$ of $\pi_1(H_g)$. The orange curves are the $x_i$ for $i \le g_1 $, the green and purple curves are the $x_i$ for $g_1<i\le g_1+g_2$ and $g_1+g_2<i\le g$ respectively.}
    \label{fig:handlebody2}
\end{figure}

\begin{figure}
    \centering
    \includegraphics[width= 0.45 \textwidth]{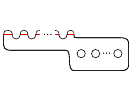}
    \caption{The quotient handlebody $\overline{H}_g$.}
    \label{fig:handlebody3}
\end{figure}

Suppose that the endpoints of $a_i$ are sent to the endpoints of $a_{j_i}$. 
We want to show that $\psi(\gamma_i) \cup a_{j_i}$ also bounds a disk in $H_g$. In fact $\psi(c_i)$ is null-homotopic and is also homotopic to the concatenation of $\psi(\gamma_i) \ast (\pm\gamma_{j_i})$ and the inverse of $\tau_*( \psi(\gamma_i) \ast (\pm \gamma_{j_i}))=(\tau(\psi(\gamma_i)))\ast \tau(\pm\gamma_{j_i})$, since $\gamma_{j_i}$ and $\tau(\gamma_{j_i})$ are homotopic arcs relative to their boundary.
Here the orientation that we have to consider on $ \gamma_{j_i}$ depends on how $\psi$ maps the endpoints of $\gamma_i $ to the endpoints of $ \gamma_{j_i}$.

Hence we write $\psi(c_i)$ as a concatenation of two loops, and therefore one should be the inverse of the other, that is to say $\tau_*( \psi(\gamma_i) \ast (\pm\gamma_{j_i}))= \psi(\gamma_i) \ast (\pm\gamma_{j_i})$, but by the above remark this is possible if and only if $\psi(\gamma_i) \ast (\pm\gamma_{j_i})$ is null-homotopic. Since $\gamma_{j_i}$ is homotopic relative to its boundary to $a_{j_i}$ then we have that $\psi(\gamma_i) \cup a_{j_i} $ bounds a disk in $H_g$ as desired. 

Let $p:H_g \to \overline{H}_g \coloneqq H_g {/\tau}$ the projection to the quotient. Notice that $\overline{H}_g$ is a handlebody of genus $g_2$ (\Cref{fig:handlebody3}).

Observe that the restriction of $\psi$ to the boundary induces a well-defined PL homeomorphism $\overline{\psi}$ on the boundary of $\overline{H}_g$. 
We would like to extend $\overline{\psi} $ to a PL homeomorphism of $\overline{H}_g$ that fixes $p(\Fix(\tau))$ setwise.

To do that, we start by remarking that, since $\gamma_i \cup a_i$ and $\psi(\gamma_i) \cup a_{j_i}$ bound a disk in $H_g$, their image via $p$ in $\overline{H_g}$ both bound embedded disks, that we call $\overline{D}_i$ and $\overline{D}'_i$. 

By a standard cut and paste argument we can suppose that for $i=1,\ldots,g_1+1$ the disks $\overline{D}_i$ are pairwise disjoint and the same holds for the disks $ \overline{D}_i'$.

We extend $\overline{\psi}$ on $\overline{D}_i$ so that $\overline{\psi}(\overline{D}_i)=\overline{D}'_i$ and $\overline{\psi} $ restricted to $\overline{D}_i $ is a PL homeomorphism onto the image. We still call the extended map $\overline{\psi}$. 
Let $\overline{F_j}$ for $j=1,\ldots, g_2$ be compressing disks for $\overline{H}_g$ that avoid the disks $\overline{D}_i$'s. Since these disks are far from the image of the fixed point set, $p^{-1} (\overline{F_j})$ consists of two disjoint disks $F_j, F'_j$ that are interchanged by $\tau$. Since $\overline{\psi}(\partial\overline{F_j})= p (\psi(\partial F_j)) $, and $\psi(\partial F_j) $ bounds the disk $\psi(F_j) \subset H_g$, $\overline{\psi}(\partial\overline{F_j})$ bounds a compressing disk. Hence we can extend $\overline{\psi}$ to be defined on the compressing disks $ \overline{F_j}$. Now, since $\overline{H}_g \setminus \bigcup_{i=1}^{g_1+1} \overline{D}_i \cup \bigcup_{j=1}^{g_2} F_j$ is an open ball, we can extend $\overline{\psi}$ on the whole $\overline{H}_g$ as well. We still call the resulting map $\overline{\psi}$.

We now want to show that $\overline{\psi}$ can be lifted to $\psi:H_g \to H_g$.
By standard covering theory, since $p(\Fix(\tau))$ is fixed setwise by $\overline{\psi}$, $\overline{\psi}$ can be lifted if and only if the image of
$\overline{\psi}_*\circ p_*: \pi_1(H_g \setminus \Fix(\tau)) \to \pi_1(\overline{H}_g \setminus p(\Fix(\tau))$
is contained in the image of 
$p_*: \pi_1(H_g \setminus \Fix(\tau)) \to \pi_1(\overline{H}_g \setminus p(\Fix(\tau)) $. 

Notice that the inclusion induces surjective homomorphisms $ \pi_1(\partial H_g \setminus \Fix (\tau)) \to \pi_1(H_g \setminus \Fix (\tau)) $ and $ \pi_1(\partial \overline{H}_g \setminus p(\Fix (\tau))) \to \pi_1(\overline{H}_g \setminus p(\Fix (\tau)))$. Since $\overline{\psi}_{|\partial \overline{H}_g}$ has $\psi$ as a lift on $\partial H_g$, this implies that $ \overline{\psi}_* \circ p_* (  \pi_1(\partial H_g \setminus \Fix (\tau)) ) \subset  p_*(\pi_1(\partial H_g \setminus \Fix (\tau)))$. By the naturality of the inclusion induced maps it follows that $\overline{\psi}_* \circ p_* (  \pi_1( H_g \setminus \Fix (\tau)) ) \subset  p_*(\pi_1( H_g \setminus \Fix (\tau))) $.

Hence $\overline{\psi}$ has a unique lift $\widetilde{\psi}$ which coincides with $\psi$ on the boundary. Since $\tau$ is an involution, this lift is by construction $\tau$-equivariant.
\end{proof}

We can finally prove \Cref{thm: uniqueness}.

\begin{proof}[Proof of \Cref{thm: uniqueness}]
Given a point $(t,p) \in I \times D^2$ and $\varepsilon>0$, let us denote by $D_\varepsilon(t,p)$ the closed ball of radius $\varepsilon$ contained in $\{t\}\times D^2$ and centered at $p$. 
Let $\varepsilon$ be small enough so that $N_a \coloneqq \bigcup_{(t,p)\in a} D_\varepsilon(t,p) $ is a closed tubular neighbourhood of $a$ and $N_{a'} \coloneqq \bigcup_{(t,p')\in a'} D_\varepsilon(t,p') $ is a closed tubular neighbourhood of $a'$. 
Notice that both $N_a$ and $N_{a'}$ are fixed by $\rho$ as sets.

Let $a_1,\ldots,a_n$ and $ a_1',\ldots,a_n'$ be the connected components of $a$ and $a'$ respectively. We can suppose without loss of generality that we ordered them so that the endpoints of $a_i$ coincide with the ones of $a_i'$ for $i=1,\ldots,n$. Let $p_i, p_i': I \to D^2$ be such that $ a_i= \bigcup _{t \in I} (t,p_i(t))$ and $ a'_i = \bigcup_{t \in I} (t,p_i'(t))$ for $i= 1,\ldots ,n$.

Let $\phi: N_a \to N_{a'}$ be the map that sends $ (t,p_i(t) + v)\in N_a$, where $v \in \R^2$, $\|v\| \le \varepsilon$, to $(t,p'_i(t)+v) $ in $N_{a'}$. 

One can check by direct computation that $\phi$ is $\rho$-equivariant.

We would like to show that $\phi$ can be extended to a map $I \times D^2 \to I \times D^2$ so that it is the identity on the boundary.

Define $X_a$  to be the handlebody $(I\times D^2) \setminus \operatorname{Int}(N_a)$ and $X_{a'}$ to be the handlebody $(I\times D^2) \setminus \operatorname{Int}(N_{a'})$.
For $i=1,\ldots, n$, define $ \widetilde{a}_i \coloneqq \bigcup_{t \in I} (t, p_i(t) + (\varepsilon,0))$, the pushoff of $a_i$ on the boundary of $N_a $ along $(\varepsilon,0)$. 
Similarly, let $\widetilde{a}'_i$ be the pushoff of $a'_i$ on the boundary of $N_{a'} $ along $(\varepsilon,0)$. 
Let $\gamma_i$ be an arc in $\partial(I \times D^2)  \cap \partial X_a$, with the same endpoints as $\widetilde{a}_i$ and such that the loop $\gamma_i \cup \widetilde{a}_i$ bounds a disk inside $X_a$. We also ask the $\gamma_i$'s to be pairwise disjoint.

We claim that also $\gamma_i \cup \widetilde{a}'_i $  bounds a disk inside $X_{a'}$. In fact, since $a$ and $a'$ represent the same braid, there is a diffeomorphism of $I \times D^2$ fixing the boundary that maps $X_a$ to $X_{a'}$. This diffeomorphism will send $\widetilde{a}_i$ to a certain arc $\delta_i \subset \partial X_{a'} \cap \partial N_{a'}$. In particular, $a_i \cup \widetilde{a}_i $ and $a'_i \cup \delta_i $
 are isotopic relative to the boundary inside $I \times D^2$. Notice that there is only one arc, up to isotopy of $X_{a'} $ fixing $\partial (I \times D^2) \cap \partial X_{a'}$ and contained in $ \partial X_{a'}  \cap \partial N_{a'}$, with this property. 
 Observe that $ a_i \cup \widetilde{a}_i$ and $a'_i \cup \widetilde{a}'_i$ are clearly isotopic relative to the boundary inside $I \times D^2$, hence $\widetilde{a}'_i $ is isotopic to $\delta_i$ in $ X_{a'}$, relative to  $\partial (I \times D^2) \cap \partial X_{a'}$. As a consequence $ \gamma_i \cup \widetilde{a}'_i$ bounds a disk in $X_{a'}$.

One has to notice that $\phi$ can be extended as a map $I \times D^2 \to I \times D^2$ so that $\phi_{\partial(I \times D^2)}= \id_{\partial(I \times D^2)}$. In fact, the nonseparating meridians $ \widetilde{a_i} \cup \gamma_i$ of $X_a$ are sent by the map $\phi$, extended as the identity on $\partial (I \times D^2)$, to the nonseparating meridians $ \widetilde{a}'_i \cup \gamma_i$ of $X_{a'}$. We will denote this extension as $\phi$ as well.

Let $\rho'=\phi^{-1}\circ \rho \circ \phi.$ Both $\rho$ and $\rho'$ are involutions on $X_a$, which is a genus $n$ handlebody. 
Notice that $\rho$ and $\rho'$ agree on $\partial X_a$.
Hence, as a consequence of results in \cite{PantaleoniPiergallini}, (see also \cite{McCulloughMillerZimmermann}), there exists a PL homeomorphism $f:X_a \to X_a$ such that $\rho= f^{-1} \circ \rho' \circ f$.    

By \Cref{prop: handlebody}, there exists $\overline{f}: X_a \to X_a$ that agrees with $f$ on $\partial X_a$ and that is $\rho'$-equivariant. Hence $\rho = f^{-1} \circ \rho' \circ f = (\overline{f}^{-1} \circ f)^{-1} \circ \rho' \circ (\overline{f}^{-1} \circ f) $
and $ \psi \coloneqq \overline{f}^{-1} \circ f$ is the identity on the boundary and can therefore be extended to the rest of $I \times D^2$, i.e. on $N_a$, as the identity. We still call this extension $\psi$. Then $ H \coloneqq\phi \circ \psi $ sends $a$ to $a'$ and $H \circ \rho = \rho \circ H$ and is therefore the required PL homeomorphism. 
\end{proof}

We are ready to show that the equivariant closure map only depends on the palindromic braids and not on the choice of representatives.
\begin{prop}
Let $\alpha$ and $\beta$ be palindromic $n$-braids. Then for any choice of symmetric representatives for $\alpha$ and $\beta$ the resulting strongly involutive links are equivariantly isotopic.
\end{prop}

\begin{proof}
Suppose that $a$ and $a'$ are two different representatives of $\alpha$ (the case where we pick two different representatives for $\beta$ is analogous), and let $b$ be a symmetric representative of $\beta$. Call $(L,\tau)$ the strongly involutive link obtained by equivariant closure of $a$ and $b$, and $(L',\tau)$ the one obtained by equivariant closure of $a'$ and $b$. Then let $H$ be the PL homeomorphism prescribed by \Cref{thm: uniqueness}. One can extend $H$ to be the identity outside of $I \times D^2$ and therefore $H$ is a $\tau$-equivariant PL homeomorphism on $S^3$ that sends $L$ to $L'$. 
Since $H$ preserves the orientation of the fixed point set $\ax$, by the proof of Proposition 2.4 in \cite{LobbWatson}, $(L,\tau)$ and $(L',\tau)$ are equivariantly isotopic. 
\end{proof}

\section{An equivariant Alexander theorem} \label{sec: Alex}
After defining the equivariant closure map in \Cref{sec: palindromi}, in this section we would like to show that every strongly involutive link is Sakuma equivalent to the equivariant closure of two palindromic braids. 
What we will actually prove is that every strongly involutive link $(L,\tau)$ is equivariantly isotopic to the equivariant closure of two palindromic braids. Recall the definition of Sakuma equivalence and equivariant isotopy from \Cref{def: equivalencesintro}. Here $\tau:S^3 \to S^3$ is the $\pi$-rotation around the axis $\ax$ (as defined in \Cref{sec: firstdef} and used in \Cref{sec: palindromi}). Recall that we supposed that $\ax \setminus \{\infty\}$ is the $y$-axis in $\R^3$.

Since, as observed in \Cref{sec: firstdef}, every strongly involutive link is Sakuma equivalent to a strongly involutive link where the involution is $\tau$, this will imply the desired result.

From now on, if we do not indicate the involution of a strongly involutive link, we will assume that it is $\tau$.

Let $L$ be a strongly involutive link. 
Up to a small equivariant perturbation of the link $L$, there exists a plane $\kappa$ in $\R^3$ containing the $y$-axis such that the orthogonal projection on $\kappa$ gives a diagram for the strongly involutive link $L$. One can suppose without loss of generality that the plane $\kappa$ corresponds to the $(x,y)$-plane in $\R^3$ and we will sometimes refer to it as $\R^2$. 
Notice that $\tau$ preserves $\kappa=\R^2$ and acts on it as the reflection along the $y$-axis.

Let $\bx$ be the union of the $z$-axis in $\R^3$ and $\{\infty\}$. We can suppose up to small equivariant perturbations that no projection of an edge of $L$ lies on a line passing through $0\in \R^2$. Hence every edge of $L$ inherits a sign: positive when the projection is oriented so that it goes counterclockwise with respect to $0$ and negative otherwise. 

From now on we will always suppose that our links project to the $(x,y)$-plane as stated above. As a consequence it will make sense to consider the number of negative edges as associated to the link (as subset of $\R^3$) instead of to the diagram.
Notice that since $\tau$ reverses the orientation of $L$, if $e$ is a positive (respectively negative) edge of $L$, then $\tau(e)$ is a positive (respectively negative) edge of $L$.

Notice that if $L$ has no negative edges, then $L$ is in equivariant braid position with respect to $\bx $. 
We will find it useful to define the following complexity function: 
\begin{dfn}
Given a diagram of a strongly involutive link $L$, defined as above, we define $h(L)$ to be the number of negative edges of $L$.
\end{dfn}

\begin{rmk} \label{rmk: braidclosure}
Observe that, up to equivariant planar isotopy that fixes $0 \in \R^2$, 
we can suppose that all the crossings in the diagram of a strongly involutive link $L$ with $h(L)=0$ lie inside the projections of the balls $B_1$ and $B_2$ in \Cref{dfn: eqclosure}. As a consequence, all strongly involutive links with $h=0$ are equivariantly isotopic in the exterior of $\bx$ to the equivariant closure of two palindromic braids.
Our strategy to prove the equivariant Alexander theorem will be to show that up to equivariant isotopy we can reduce the number of negative edges. 
\end{rmk}

We now introduce a move on strongly involutive links that decreases the value of the function $h$.
\begin{dfn}
Let $e=p_0p_m$ be a negative edge of $L$ and let $p_1,\ldots,p_{m-1} \in \operatorname{Int}(e)$ such that $p_i p_{i+1}$ is negative for all $i =0,\ldots, m-1$. Let $q_1,\ldots,q_m \notin e$, such that the $p_iq_{i+1}$ and $q_{i}p_i$ are positive.
Then we define the move $\Stil_{p_0,\ldots,p_m}^{q_1,\ldots,q_m}\coloneqq\Etil_{p_{m-1},p_m}^{q_m} \circ \cdots \circ \Etil_{p_{0},p_1}^{q_1},$
whenever the right-hand side is well-defined. We call this move a \emph{sawtooth move} (see \Cref{fig:stil} for an example).
\end{dfn}
Notice that applying the moves on the right-hand hand side in any order gives as a result the same link.

\begin{dfn}
 A sawtooth move $\Stil_{p_0,\ldots,p_m}^{q_1,\ldots,q_m}\coloneqq\Etil_{p_{m-1},p_m}^{q_m} \circ \cdots \circ \Etil_{p_{0},p_1}^{q_1} $ is \emph{admissible} if $\Etil_{p_i,p_{i+1}}^{q_{i+1}}$ is admissible on the link it is applied to and if for every $i\neq j$ with $i,j = 0,\ldots,m-1$ the set $ [p_i,p_{i+1},q_{i+1}] \cap \tau ([p_j,p_{j+1},q_{j+1}])$ is empty. 

\begin{figure}
    \centering
\includegraphics[width= \textwidth]{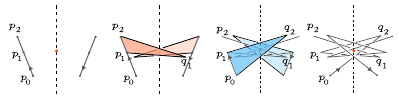}
    \caption{A sawtooth move.}
    \label{fig:stil}
\end{figure}

\end{dfn}
Notice that if $\Stil_{p_0,\ldots,p_m}^{q_1,\ldots,q_m} $ is admissible then $L$ and $\Stil_{p_0,\ldots,p_m}^{q_1,\ldots,q_m}(L) $ are equivariantly isotopic.
\begin{rmk}
Observe that there was some abuse of notation in the previous definition, since the points $p_i$'s, for $i=1, \dots, m-1$, are not vertices of $L$, while the $\Etil$ moves can only be applied to consecutive vertices. We are implicitly assuming that performing a sawtooth move comprises the operation of adding all the $p_i$'s to the vertices of $L$.
\end{rmk}

Notice that a sawtooth move replaces two negative edges with many positive ones, and therefore the value of the  function $h$ decreases by $2$. \\ The following lemma shows that if a strongly involutive link has a negative edge, then one can find an admissible sawtooth move on $e$. As a result, this shows that it is always possible to put a strongly involutive link in equivariant braid position.

\begin{lem}\label{lem: sawtooth}
Let $L$ be a strongly involutive link.
Given a negative edge $e=ab$ of $L$, there is an admissible sawtooth move $\Stil_{p_0,\ldots,p_n}^{q_1,\ldots,q_n}$ so that $p_0=a$ and $p_n=b$. Moreover, if $\Omega\subset \R^3$ is a compact convex set such that $\Omega \cap e = \emptyset$ and $\tau(\Omega)=\Omega$, we can pick the points $p_i$ and $q_i$ so that the triangles $[p_i,p_{i+1},q_{i+1}]$ and $\tau([p_i,p_{i+1},q_{i+1}])$ do not intersect $\Omega$. 
\end{lem}

\begin{proof}

\textbf{Case 1:} suppose first that there are no points in 
$e$ that project to a double point in the diagram.

\begin{figure}
    \centering
    \includegraphics[width=0.32 \textwidth]{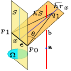}
    \caption{Case 1 of the proof of \Cref{lem: sawtooth}.}
    \label{fig: Stilnocross}
\end{figure}
For this part of the proof we refer to \Cref{fig: Stilnocross}.
 Consider the set $S$ defined as the union over all points $y $ in $e$ of lines $r_y$ parallel to $\bx$ and containing $y$.
 This set does not intersect $L\setminus e$, since there are no points of $e$ that project to a double point. Moreover, if $S$ intersects $\ax$, then the intersection must be an endpoint of $e$. In fact, if $S \cap \ax$ is non-empty, then $S \cap \tau(e)$ is also non-empty. By the previous observation we must have that $S \cap \tau(e) $ lies in $e$. This is possible if and only if $e$ and $\tau(e)$ are consecutive edges and their common endpoint is a fixed point of $\tau$.  
 
 The set $S$ is split by $e$ in two components. Since $\Omega$ is convex and $e \cap \Omega = \emptyset$, at least one of the two components of $S \setminus e$ does not intersect $\Omega$. Let $hS$ be one such component. Notice that this is a union of half-lines $hr_y$, such that $hr_y \subset r_y$ and $hr_y$ has $y$ as endpoint. 
 Let $x \in \operatorname{Int}(e)$, and let $\alpha_x$ be the plane containing $x$ and $\bx$. For all $y \in e$, let $\alpha_y$ be the plane parallel to $\alpha_x$ and containing $y$.
 Take a rotation of a small angle $\theta$ centered in $y$ of each $hr_y$ inside $\alpha_y$, and such that the image of $hr_x$ via this rotation, that we call $\widetilde{hr}_x$, has non-empty intersection with  $\bx$. 
 Call the image of $hS$ via this rotation $\widetilde{hS}$.
 If $\theta$ is small enough, since $L$ and $\Omega$ are compact, $ \widetilde{hS}$ does not intersect $L \setminus e$ nor $\Omega$.
 The axis $\ax$ is transverse to the planes $\alpha_y$ for every $y \in e$, and the intersection $\ax \cap \alpha_y$ consists of one point. This means that for $\theta$ small enough, $\widetilde{hS} \cap \ax$ is either empty or equal to one endpoint of $e$ (when this is a fixed point of $\tau$). Let $p_0=a$ and $p_1=b$. Now take $q_1$ on $\widetilde{hr}_x$ such that $[p_0,p_1,q_1]$ intersects $\bx$ in a point. Observe that $ p_0q_1$ and $q_1p_1$ are positive. Then $\Stil_{p_0,p_1}^{q_1}$ is admissible and the triangles $[p_0,p_1,q_1]$ and $\tau([p_0,p_1,q_1])$ do not intersect $\Omega$, as required.

\textbf{Case 2:} suppose now that there exists a point $x \in e$ that projects to a crossing with some edge $f$ of $L$. Let $n>0$ be the number of such points. Suppose first that any such $f$ is not $\tau(e)$. For this part of the proof we refer to \Cref{fig: Stilcrostot}.

The plane $\alpha$ in $\R^3$ containing $x$ and $\mathbf{b}$ intersects the link in a finite number of points and the axis $\mathbf{a}$ only in $0$. The line $r$ in $\alpha$ containing $x$ and parallel to $\mathbf{b}$ meets $L$ exactly in one point, which lies in $f$. 

Let $hS$ be as in the previous case and let $hr$ be the half-line $r \cap hS$.
\begin{figure}[ht]
    \centering
    \begin{subfigure}[t]{0.4\textwidth}
    \centering
\includegraphics[width=0.8\textwidth]{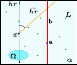}
\caption{The half-line $\widetilde{hr}$}\label{fig: Stilcross}
    \end{subfigure} 
 \hfill
\begin{subfigure}[t]{0.4\textwidth}
        \centering
\includegraphics[width=0.8\textwidth]{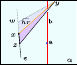}
\caption{The triangle $[y,z,w]$.}
\label{fig: Stilcross2}
    \end{subfigure} 

    \caption{The second case in the proof of \Cref{lem: sawtooth}, when there exists $x \in e$ that projects to a double point in the diagram.}
    \label{fig: Stilcrostot}
\end{figure}

Take a small rotation $\widetilde{hr}$ of $hr$ inside $\alpha$ centered in $x$, so that the new half-line $\widetilde{hr}$ intersects $\bx$. If the rotation is small enough, since $\Omega$ is compact, this half-line will still be disjoint from $\Omega$ and will generically avoid $L$ and $\ax$. See \Cref{fig: Stilcross} for a picture.

Let $y $ be a point in $ \widetilde{hr}$ such that the segment $[x,y]$ has nonempty intersection with $\bx$. Then fix $z,w \in e$ such that $x \in [z,w]$ and $zw$ is negative. If we pick $z,w$ close to $x$, then the triangle $[y,z,w]$ does not intersect the link $L$, the convex set $\Omega$, nor the axis $\ax$. One can easily see that $zy$ and $yw$ are positive segments. 

This triangle will be a ``tooth" of the sawtooth move. More precisely, if there are no other points in $e=ab$ that project to a crossing, call $a=p_0$, $z=p_1$, $w=p_2$, $b=p_3$ and $y=q_2$. We can apply the same procedure described in Case 1 successively on the segments $az$ and $wb$, finding points that we call respectively $q_1$ and $q_3$. We pick $q_1$ and $q_3$ in the set obtained by a rotation of a very small angle of $hS \cap \bigcup_{x\in[a,z]}r_x$ and $hS \cap \bigcup_{x\in[w,b]}r_x $ respectively.  
%Observe that one has to apply this procedure not on $L$, but on the link $\Etil_{z,w}^y(L)$ to find $q_1$, and on the link $\Etil_{a,z}^{q_1} \circ \Etil_{z,w}^y(L)$ to find $q_3$.
Since $e$ and $\tau(e)$ do not project to a crossing, $hS \cap \tau(hS) \subset e \cap \tau(e)$, and this implies that if we pick the rotations to be small enough, $[p_i,p_{i+1},q_{i+1}] \cap \tau([p_j,p_{j+1},q_{j+1}])$ is empty if $i \neq j$, and contained in $e \cap \tau(e)$ if $i=j$. 
Hence the move $\Stil_{p_0,p_1,p_2,p_3}^{q_1,q_2,q_3}$ is admissible on $L$. Otherwise, if $n>1$, the segments $az$ and $wb$ have at most $n-1$ points that project to a double point, and we can therefore iterate this procedure. 

 \textbf{Case 3:} we are left to find an admissible sawtooth move when there exists $x$ in $e$ that projects to a double point with $\tau(e)$ in the diagram. Notice that in this case the plane $\alpha$ containing $x$ and $\bx$ contains $\ax$ as well (see \Cref{{fig: Stilsym}}). 
\begin{figure}
    \centering
    \includegraphics[width=0.35\textwidth]{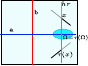}
    \caption{The final case of the proof of \Cref{lem: sawtooth}, when $x$ and $\tau(x)$ project to the same point in the diagram.}
    \label{fig: Stilsym}
\end{figure}

Let $S$ be as in the first case, let $r$ be the line parallel to $\bx$ containing $x$ and let $hr$ be the half-line contained in $ r$ with endpoint $x$ that does not intersect $\ax$. Let $hS$ be the closure of the component of $S\setminus e$ containing $hr$.

Suppose first that $hS \cap \Omega= \emptyset$. 
Hence $hr \cap \Omega$ is empty and we can proceed to find points $z,w,y$ and to conclude the proof in this case as in Case 2. 

If $hS \cap \Omega \neq \emptyset$, notice that $hr \cap \Omega = \emptyset$. In fact if that was not the case, since $\Omega = \tau(\Omega)$ then $\tau(hr) \cap \Omega \neq \emptyset$. Since $\Omega$ is convex and $ \tau (hr) \subset r \setminus hr$, that would mean that $ x \in \Omega$. Let $y,z,w$ be as in Case 2.

Call $H$ the half-plane in $\R^2$ containing $0$ and delimited by the line containing $\pi(\tau(e))$.
Notice that one and only one of $\pi([a,z])$ and $\pi([w,b])$ are contained in $H$. Suppose that $\pi([w,b]) \subset H$, the other case being analogous.

Let $S_z= \bigcup_{t \in [a,z]} r_t$ and $S_w= \bigcup_{t \in [w,b]} r_t$. 
Up to changing $z$ with another point closer to $x $, and $y$ with another point in $[y,z,w]$ closer to $\bx$ (that we will still call $z$ and $y$) we can suppose that $S_w \cap \tau([y,z,w]) = \emptyset$ (see \Cref{fig: riducitriangolo}). 

\begin{figure}
    \centering
    \includegraphics[width=0.25 \textwidth]{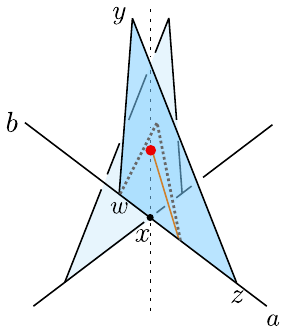}
    \caption{Up to changing $z$ with another point closer to $x $, and $y$ with another point in $[y,z,w]$ closer to $\bx$ we can suppose that $S_w \cap \tau([y,z,w]) = \emptyset$. One just needs to pick the new point close to $x$ such that the orange segment in the picture connecting it to $\bx$  does not intersect $\tau([b,w])$. Then if one picks $y$ close to $\bx$ we have the desired property.}
    \label{fig: riducitriangolo}
\end{figure}

Notice that $S_z \cap \tau([y,z,w]) = \emptyset$.  In fact $ \pi(S_z)=\pi([a,z]) \subset \R^2 \setminus H$, while $ \pi(\tau([y,z,w])) $ is contained in $ H$ since $ \tau([z,w]) \subset \tau(e)$ and $\pi(y)$, and therefore $\pi(\tau(y))$ are very close to $0 \in H$, as shown in \Cref{fig:separati}. 

\begin{figure}
    \centering
    \includegraphics[width=0.3 \textwidth]{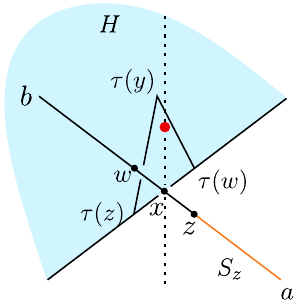}
    \caption{The sets $S_z$ and $\tau([y,z,w])$ are disjoint.}
    \label{fig:separati}
\end{figure}
Hence we can proceed as in the previous cases to construct the other teeth of the sawtooth on $[a,z]$ and $[w,b]$ on sets obtained by rotations of small angles of $S_z$ and $S_w$. More precisely, we pick $p_0,\ldots,p_k \in [a,z]$ and $p_{k+1},\ldots,p_m \in [w,b]$
such that $p_0p_k=az$ and $p_{k+1}p_m=wb$, $q_1,\ldots,q_m$, such that $q_{k+1}=y$ and such that $ p_iq_{i+1}$ and $q_{i+1}p_{i+1}$ are positive for $i=0,\ldots,m-1$, and such that, for $i \neq k$ the point $q_{i+1}$, lies in a small rotation (in the sense of the previous cases) of either $S_z$ or $S_w$.

Moreover we construct them so that 
$ \Stil_{p_0,\ldots,p_k}^{q_1,\ldots,q_k}$ and $ \Stil_{p_{k+1},\ldots,p_m}^{q_{k+1},\ldots,q_n}$ are admissible on $L$.
If the rotation angles are small enough, $[p_i,p_{i+1},q_{i+1}]\cap\tau([y,z,w])$ is empty for $i=0,\ldots,m-1$.
Hence $\Stil_{p_0,\ldots,p_m}^{q_1,\ldots,q_m} $ is an admissible sawtooth move on $e$.
\end{proof}

\begin{thm}[Equivariant Alexander theorem] \label{thm: Alexander}
The strongly involutive link $(L,\tau)$ is equivariantly isotopic to the equivariant closure of two palindromic braids $\alpha$ and $\beta$. 
\end{thm}

\begin{proof}
 If $h(L)=0$, we observed in \Cref{rmk: braidclosure} that the link $L$ is equivariantly isotopic to the equivariant closure of two palindromic braids $\alpha$ and $\beta$. Since $L=\tau(L)$, we must have that $\alpha=\tau(\alpha)$ and $\beta=\tau(\beta)$. Since $\tau(\alpha)= \overline{\alpha}$ and $\tau(\beta)= \overline{\beta}$, we get the conclusion in this case.
If $h(L)>0$, let $e$ be a negative edge of $L$. By \Cref{lem: sawtooth} we can find an admissible sawtooth move on $e$ and decrease $h(L)$. We conclude the proof by induction on $h(L)$.
\end{proof}

%\begin{rmk}
%A natural question that can arise in the perspective of \Cref{part1} and \Cref{part2} of this thesis, is what can be said of the homomorphism defect of strongly invertible knots invariants, such as, for example, the equivariant signature discussed in \cite{AlfieriBoyle}.
%To address this question, we would need a group structure on the set $P_n \times P_n$. However, to the best of our knowledge, there is not an obvious group structure on this set.
%\end{rmk}
%We provide a table of the representation as equivariant closure of palindromic braids of a few low-crossing strongly invertible knots. These coincide with the first few knots in the table in Appendix A of \cite{BoyleIssa}.
% \begin{figure}[H]
%     \centering
%     \includegraphics[width= 0.25 \textwidth]{braidclosure.pdf}
%     \caption{The closure of the composition of two palindromic braids is strongly involutive.}
%     \label{fig:braidclosure}
% \end{figure}

\section{An equivariant Markov theorem} \label{sec: Markov}

The goal of this section is to give a proof of an equivariant Markov theorem for strongly involutive links. 

We will first show that, if two strongly involutive links $L$ and $L'$ are equivariantly isotopic, and $h(L)=h(L')=0$, then $L'$ can be obtained from $L$ via a finite set of operations (see \Cref{thm: markovlink}). 

When $L=\widehat{\alpha\beta}$ and $L'=\widehat{\gamma\delta}$ we will interpret the operations described in \Cref{thm: markovlink} in terms of operations on palindromic braids. As a consequence we will obtain \Cref{thm: eqmarkov}, that gives necessary and sufficient conditions for two pairs of palindromic braids to have equivariant closures that are equivariantly isotopic.

The steps that we will take for the proof are the following:

 \textbf{Step 1:} We will start by defining some new operations on the set of strongly involutive links. These operations will be of four types: $\Rcal,\Rtil, \W, \Wtil$. Moves of type $\Rcal$ and $\Rtil$ are just moves of type $\D$ and $\Etil$ respectively, that substitute positive edges with other positive edges, hence preserving $h$. Moves of type $\Wtil, \W $ are as depicted in \Cref{fig:Wtil} and \Cref{fig:Wmove}, and they also replace positive edges with other positive edges. 
    Hence all these operations preserve the value of $h$.

 \textbf{Step 2:} We will show that the operations introduced in the previous step can be decomposed in other moves of the same types, with some special properties on their support. 

 Roughly speaking, we will show: 
    \begin{enumerate}[label=(\roman*)]
        \item the diameter of the support of the $\Rcal$ and $\Rtil$ moves can be supposed to be arbitrarily small. This will be done in \Cref{lem: reduceRtil} and \Cref{lem: reduceR}.
        \item that the support of the $\W$ and $\Wtil$ moves can be supposed to be disjoint from the link. In this case we say that the move is empty. 
        Moreover we show that the projection on the plane of the support of an empty $\W$ or $\Wtil$ move is disjoint from the projection of the link. See \Cref{prop: empty} and \Cref{prop: really empty}. 
\end{enumerate}
 
 This will help us in the following steps, since we will be able to assume that these moves have the above properties. Moreover we will also be able to identify these moves as actual operations on palindromic braids (see Step 4).
 
    \textbf{Step 3:} In this step we will prove \Cref{thm: markovlink}. Let $L$ and $L'$ be strongly involutive, and suppose that $L'$ can be obtained by $L$ after applying a finite set of admissible moves of type $\D, \Etil, \Rcal, \Rtil, \W,\Wtil, \Stil$ or their inverses. We want to show that we can find another set of the above moves so that $L=L_0 \to L_1\to \cdots \to L_{n-1} \to L_n=L'$, (here every arrow symbolizes the application of one of these moves), and the maximum value of $h$ is attained in one of the extremes of this sequence. When $h(L)=h(L')=0$ this implies that $h(L_i)=0$ for $i=1,\ldots,n$, thus proving \Cref{thm: markovlink}.
    This will be proved in two substeps:
    \begin{enumerate}[label=(\roman*)]
        \item We show in \Cref{prop: isolatemax} that we can modify the sequence $L=L_0 \to L_1\to \cdots \to L_n = L' $ so that if $h(L_i)>0$, then $h(L_{i+1})\neq h(L_i)$. Moreover we can suppose to do that without increasing the maximum value of $h$ attained by this sequence. As a consequence we can suppose that the $i$'s where $h$ is maximal for this sequence are isolated, unless the maximum is equal to $0$. 
        \item We show that we can reduce any local maxima for $h$ in the sequence. More precisely, when one has a piece of the sequence $L_{i-1}\to L_i \to L_{i+1} $, with $h(L_i)>h(L_{i-1})$ and $h(L_i)>h(L_{i+1} )$, the sequence can be modified to $L_{i-1} \to L'_1\to \cdots \to L'_k \to L_{i+1}$ so that $h(L'_j)<h(L_i)$ for $j=1,\ldots,k$. By the previous substep we can assume that every maximum for $h$ in the sequence is isolated, hence we can apply the described procedure to all the $0<i<n$ such that $h(L_i)$ is maximal. We thus find a new sequence such that the maximum attained value of $h$ is either strictly lower than before, or is achieved by $h(L)$ or $h(L')$. This will be done in \Cref{prop: reducemin}.
    \end{enumerate}

    To conclude, one just needs to realize that we can apply iteratively \Cref{prop: isolatemax} and \Cref{prop: reducemin} to reduce the maximum value of $h$ along the sequence. We therefore obtain a sequence such that the maximum value of $h$ is realized by either $h(L)$ or $h(L')$. 

\textbf{Step 4:} The goal of this final step is to prove \Cref{thm: eqmarkov}. 
Recall that by \Cref{rmk: braidclosure}, if $h(L)=0$, then, up to an equivariant isotopy in the exterior of $\bx$ that modifies the link diagram only up to planar isotopy, $L$ is the equivariant closure of two palindromic braids $ \alpha$ and $\beta$. These braids are not unique but depend on the choice of the equivariant isotopy. We first start by stating how all such possible pairs of palindromic braids are related in \Cref{rmk: nonuniquebraids}. We then proceed to show in \Cref{prop: Rconj} how $\Rcal$ and $\Rtil$ moves affect the underlying palindromic braids.
Finally we just have to observe in \Cref{dfn: Wstab} what effect $\W$ and $\Wtil$ moves have on the palindromic braids to conclude the proof of \Cref{thm: eqmarkov}.

\subsection{Step 1: defining some new operations}
Recall that in \Cref{def: admissibleEtil} we defined the moves of type $\Etil$ and that in \Cref{def: admissibleD} we defined the moves of type $\D$.
\begin{dfn}
If $\widetilde{\mathcal{E}}_{a,c}^b$ (resp. $\D^b_{a,c}$) is admissible and $ab$, $ac$ and $cb$ (resp. $ac$ and $ab$) are positive we call this move $\widetilde{\mathcal{R}}^{b}_{a,c}$ (resp. $\mathcal{R}^{b}_{a,c}$). 
\end{dfn}

\begin{nota}
    From now on we will call $\Etil$ (respectively $\D$) moves only those moves that are not $\Rtil$ (respectively $\Rcal$) moves, i.e. those that contain at least a negative edge. 
\end{nota}

\begin{nota}
    For technical reasons, we will not consider $\Rcal$ moves through $\infty$ as valid moves in what follows.
    This will not affect the proofs since an $\Rcal$ move through $\infty$ can be factorized as a composition of two $\D$ moves (one of them going through $\infty$). Hence from now on all $\Rcal$ moves are not through $\infty$.
\end{nota}

\begin{dfn}
We define the move $\widetilde{\mathcal{W}}_{a,d}^{b,c}=\Etil_{b,d}^c\circ \Etil_{a,d}^b$ whenever $ab,bc,cd,ad,db$ are positive (see \Cref{fig:Wtil}). We say that $ \Wtil_{a,d}^{b,c}$ is admissible if the moves on the right-hand side are admissible.

If $c=\tau(c)$ and $d=\tau(d)$, we define the move ${\W}_{a,d}^{b,c}={\D}_{b,d}^{c}\circ \Etil_{a,d}^b$ whenever $ab,db, ca, bc$ are positive (see \Cref{fig:Wmove}). 
We say that $ \W_{a,d}^{b,c}$ is admissible if the moves on the right hand side are admissible.
If ${\W}_{a,d}^{b,c}$ is admissible and $ {\D}_{b,d}^{c}$ goes through $\infty$, we say that ${\W}_{a,d}^{b,c}$ goes through $\infty$.
\end{dfn} 

\begin{figure}
    \centering
    \includegraphics[width= 0.9 \textwidth]{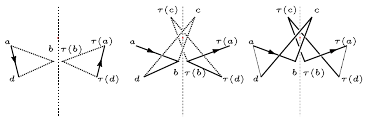}
    \caption{A move of type $\Wtil$.}
    \label{fig:Wtil}
\end{figure}

\begin{figure}
    \centering
    \includegraphics[width= 0.8 \textwidth]{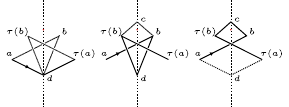}
    \caption{A move of type $\W$.}
    \label{fig:Wmove}
\end{figure}

\begin{rmk}
    Notice that the moves we introduced in this step substitute positive edges with other positive edges. Hence these operations preserve the value of the complexity function $h$.
\end{rmk}

\subsection{Step 2: decomposing moves of type $\Rcal,\Rtil,\W$ and $\Wtil$}
We now proceed to show that the moves introduced in the previous step can be decomposed into other moves of the same types, with some extra properties regarding their support. 

Given a move $\Rtil_{a,c}^b$, we call the triangle $[a,b,c]$ the \emph{support} of this move. Similarly, given a move $\Rcal_{a,c}^b$, we call the quadrilateral $[a,b,c, \tau(a)]$ the \emph{support} of the move.
We begin by showing that the support of $\Rtil$ and $\Rcal$ moves can be supposed to be arbitrarily small.

\begin{lem} \label{lem: reduceRtil}
For every $\varepsilon>0$, an admissible $\Rtil$ move factorizes into admissible $\Rtil$ moves such that the diameter of the interior of their support is smaller than $\varepsilon$ (and their inverses).
\end{lem}

\begin{figure}
    \centering
    \includegraphics[width=0.9 \textwidth]{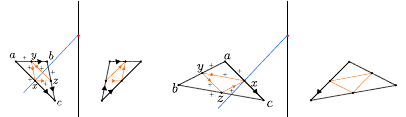}
    \caption{The two cases in the proof of \Cref{lem: reduceRtil}.}
    \label{fig:Rtilfactors}
\end{figure}

\begin{figure}
    \centering
    \includegraphics[width= 0.95 \textwidth]{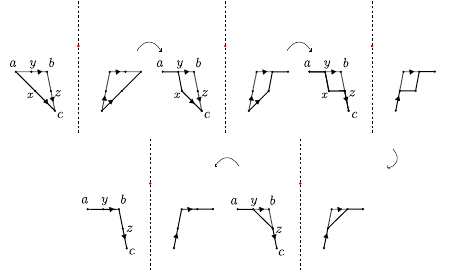}
    \caption{Every move of type $\Rtil$ factorizes into arbitrarily small $\Rtil$ moves. This is the case where $y$ and $z$ lie on opposite sides with respect to the blue line, like the left picture in \Cref{fig:Rtilfactors}.}
    \label{fig:Rtilfactorization}
\end{figure}

\begin{proof}
Fix $\Rtil_{a,c}^b$. Let $x \in [a,c],$ $y\in [a,b],$ $z \in [b,c]$. 
Notice that $yz$ is always positive for every choice of $y,z$. In fact $az$ is positive since it has the same starting point as $ab$ and $ac$, and lies in-between these edges, which are both positive. Now $yz$ has the same endpoint as $bz$ and $az$ and lies in-between them, therefore, since the last two are both positive, also $yz$ is. 
Consider the line in $\R^2$ connecting $0$ and the projection of $x$. If the projections of $y$ and $z$ lie on opposite sides of this line, then one can see that also $yx$ and $xz$ must be positive (see \Cref{fig:Rtilfactors}, left). 
In this case (see \Cref{fig:Rtilfactorization})
$$\Rtil_{a,c}^b=(\Rtil_{b,c}^{z})^{-1} \circ (\Rtil_{a,b}^{y})^{-1} \circ \Rtil_{y,z}^b \circ (\Rtil_{y,z}^x)^{-1}\circ \Rtil_{x,c}^z \circ \Rtil_{a,x}^y \circ \Rtil_{a,c}^x.$$

Otherwise, if the line does not separate $y$ and $z$, only one between $yx$ and $xz$ is negative (\Cref{fig:Rtilfactors}, right). Suppose $xz$ (the other case is analogous).
Then
$$\Rtil_{a,c}^b=(\Rtil_{b,c}^{z})^{-1} \circ (\Rtil_{a,b}^{y})^{-1} \circ \Rtil_{y,z}^b \circ (\Rtil_{z,c}^x)^{-1}\circ \Rtil_{y,x}^z \circ \Rtil_{a,x}^y \circ \Rtil_{a,c}^x.$$
If for example we pick $x,y,z$ that split in half the respective edge, all the moves involved are either applied on degenerate triangles, or have edges of length one half of the edges of the original moves.
Applying this procedure multiple times, we obtain moves that have arbitrarily small support.
\end{proof}

\begin{lem} \label{lem: reduceR}
For every $\varepsilon>0$, an admissible $\Rcal$ move factorizes into admissible $\Rtil$ moves, and admissible $\Rcal$ moves such that the diameter of the interior of their support is smaller than $\varepsilon$ (and their inverses).
\end{lem}

\begin{figure}
    \centering
    \includegraphics[width= 0.8 \textwidth]{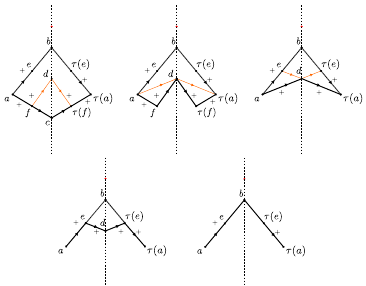}
    \caption{Every move of type $\Rcal$ factorizes into arbitrarily small $\Rcal$ and $\Rtil$ moves.}
    \label{fig:Rmovefactors}
\end{figure}

\begin{proof}
Fix $\Rcal_{a,c}^b$.
    Let $d \in[c,b]$, $e \in [a,b]$ and $ f \in [a,c]$. One notices (\Cref{fig:Rmovefactors}) that
    $$\Rcal_{a,c}^b= (\Rtil_{a,b}^e)^{-1}\circ \Rcal_{e,d}^{b}\circ \Rtil_{a,d}^e\circ (\Rtil_{a,d}^{f})^{-1}\circ \Rcal_{f,c}^d \circ \Rtil_{a,c}^{f}. $$
    
    Iterating this process at every step we obtain a smaller $\Rcal$ move. Since $d,e,f$ can be chosen arbitrarily in $[a,c],[c,b],[a,d]$ respectively, after a finite number of iterations one can make the $\Rcal$ moves arbitrarily small. By \Cref{lem: reduceRtil} also the $\Rtil$ moves can be factorized so that their support is arbitrarily small, hence the conclusion.
\end{proof}

Given a move $\Wtil_{a,d}^{b,c}$, we call $[a,b,c,d]$ the \emph{support} of this move. Similarly, given a move $\W_{a,d}^{b,c}$ not going through $\infty$, we call $[a,b,c, \tau(a), \tau(b)]$ the \emph{support} of the move.
When $\W_{a,d}^{b,c} $ goes through $\infty$, we call $\text{cl}(\R^3 \setminus [a,c,\tau(a),b, \tau(b),d])$ the \emph{support} of the move (here cl indicates the closure).
We now proceed to show that the supports of the $\W$ and $\Wtil$ moves can be assumed to be disjoint from the link. We first need a definition:

\begin{dfn}
  A move of type $\widetilde{\mathcal{W}}_{a,d}^{b,c}$ on a link $L$ is said to be \emph{empty} if it is admissible and its support $S$ intersects the link $L$ only in $[a,d]$, and $S \cap \ax $ is either empty or an endpoint of $[a,d]$ lying in $\ax$.

 A move of type $\W_{a,c}^{b,d}$ is said to be \emph{empty} if it is admissible and 
     its support $S$ intersects the link $L$ only in $[a,c]\cup[c,\tau(a)]$.
\end{dfn}

\begin{prop} \label{prop: empty}
An admissible move of type $\Wtil$ factorizes into admissible $\Rtil$ moves and empty $\Wtil$ moves.
An admissible move of type $\W$ factorizes into admissible $\Rtil$ moves, admissible $\Rcal$ moves and empty $\W$ moves.
\end{prop}

\begin{proof}
Let us deal first with the case $\Wtil_{a,d}^{b,c}$: let $x\in [a,d]$ be close to $d$. Then one can see (as in \Cref{fig:emptyWtil}) that
$$ \Wtil_{a,d}^{b,c}=(\Rtil_{a,b}^x)^{-1}\circ \Wtil_{a,x}^{b,c}\circ \Rtil_{a,d}^x.$$
If $x$ is close enough to $d$ then $[x,b,c,d]$ lies in a small neighbourhood of $[b,c,d]$, hence $[x,b,c,d]\cap L=[x,d]$ and since $\ax \cap [d,b,c]= \emptyset$ also $[x,b,c,d]\cap \ax = \emptyset. $

Now let us consider the case of a move $\W_{a,d}^{b,c}$ that does not go through $\infty$: let $d'$ be very close to $d$ such that $ad',cd' $ are positive. If $d'$ is close enough to $d$, then there exists $c'$ close enough to $c$ such that $c'c$ and $d'c'$ are positive. Then $$\W_{a,d}^{b,c}= (\Rtil_{a,b}^{d'})^{-1} \circ (\Rtil_{b,c}^{c'})^{-1}\circ \Rtil_{d'c'}^b \circ \W_{d',d}^{c',c} \circ \Rtil_{a,d}^{d'}$$ 
as shown in \Cref{fig:emptyW}, and one can see that, if $ c'$ and $d'$ are close enough to $c$ and $d$ respectively, all the moves on the right-hand side are well-defined and admissible. In fact the support of each of these moves is either contained or very close to $[a,b,d]$ and to the quadrilateral with vertices $b,c,\tau(c),d$, that do not intersect $L$ since $\W_{a,d}^{b,c}$ is admissible.
Moreover, if $d'$ is close enough to $d$ and $c'$ to $c$, then $[d',d,\tau(d'),c',\tau(c'),c]$ is close to the quadrilateral $Q$ with vertices $b,\tau(b),c,d$, thus the intersection $[d',d,\tau(d'),c',\tau(c'),c] \cap L $ is equal to $ Q \cap L = \{d\} $. Hence $ \W_{d',d}^{c',c}$ is empty.

\begin{figure}
    \centering
    \includegraphics[width= 0.6 \textwidth]{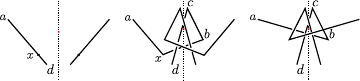}
    \caption{A move of type $\Wtil$ factorizes into $\Rtil$ moves and empty $\Wtil$ moves.}
    \label{fig:emptyWtil}
\end{figure}
 
\begin{figure}
    \centering
    \includegraphics[width= 0.9 \textwidth]{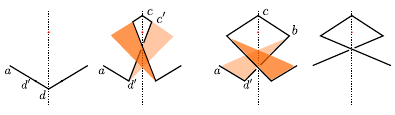}
    \caption{A move of type $\W$ factorizes into $\Rtil$ moves and empty $\W$ moves.}
    \label{fig:emptyW}
\end{figure}

Finally, let us deal with the case of  move $\W_{a,d}^{b,c} $ that goes through $\infty$. For $x\in [a,d]$ consider the line $r_x$ containing $x$ and $0$. Notice that, for $x$ close enough to $d$, the half-line $hr_x \subset r_x$ with endpoint $x$ that does not contain $0$, does not intersect $L$ outside of $x$ (because it is close to $\alpha \setminus Q$, $\alpha$ being the plane containing the quadrilateral $Q$ of vertices $b,\tau(b),c,d$). 

Let $d'\in [a,d]$ be close enough to $d$ so that it satisfies the above property. Let $y\in[d',d] $ and let $p$ lie on the half-line $hr_y$. Notice that if $d'$ is close enough to $d$, then for every choice of $p$ as above, one has that $[d',p,Kb]\cap L=\{d'\}$ for every $K>1$. 
In fact, for $z \in [d',d]$, let $U_z$ be the convex region inside the plane containing $0$, $b$, $z$ that is delimited by $hr_z$, $[z,b]$ and the half-line $\{Kb\,|\, K\ge 1\}$: up to taking $d'$ even closer to $d$, the set
$U_z \cap L$ is equal to $\{z\}$ for every $z \in [d,d']$, since $U_z$ is close to $\alpha \setminus [b,\tau(b),c,d]$ (see \Cref{fig:Ud}). 

\begin{figure}
    \centering
    \includegraphics[width= 0.35 \textwidth]{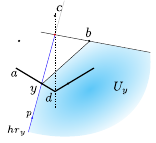}
    \caption{The set $U_{y}$.} 
    \label{fig:Ud}
\end{figure}
Notice that $[d',p,Kb] \subset \bigcup_{z \in [d,d']} U_z $, and as a consequence $[d',p,Kb]\cap L$ is contained in $ \bigcup_{z \in [d,d']} U_z \cap L= [d,d']$ for every $K>1$, for every $p \in hr_y$. Since $[d',p,Kb]\cap [d,d'] = \{d'\} $ then $[d',p,Kb]\cap L=\{d'\} $.

Moreover, if $p$ is sufficiently far from $y$, then for $K$ sufficiently large $L$ is contained in $[p,\tau(p),Kd,Kb,K\tau(b),Kc]$. We fix $K$ sufficiently large.
Then let $q$ be close to $Kb$ so that $bq$ is positive (see \Cref{fig:reallyemptyWinfinity}). 

\Cref{fig:reallyemptyWinfinity} also shows that
$$ \W_{a,d}^{b,c}=\Rcal_{b,c}^{Kc}\circ (\Rtil_{b,Kc}^{q})^{-1} \circ (\Rtil_{a,b}^{d'})^{-1}\circ \Rtil_{d',q}^b \circ (\Rtil_{d',q}^p)^{-1}\circ \W_{p,Kd}^{q,Kc} \circ \Rcal_{p,d}^{Kd}\circ\Rtil_{d',d}^p \circ \Rtil_{a,d}^{d'}.$$
Observe that all the moves on the right hand side are well-defined and admissible and that $\W_{p,Kd}^{q,Kc} $ is empty. This concludes the proof. 
\end{proof}
\begin{figure}
    \centering
    \includegraphics[width=0.75 \textwidth]{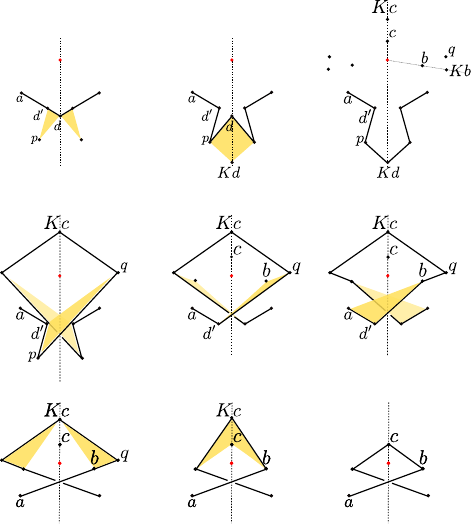}
    \caption{A move of type $\W$ through $\infty$ factorizes into moves of type $\Rcal$, $\Rtil$ and empty $\W$ moves through $\infty$.}
    \label{fig:reallyemptyWinfinity}
\end{figure}

We now show that we can actually suppose that the projections of the supports of the $\W$ and $\Wtil$ moves are disjoint from the projection of the link.

\begin{dfn}
Let $\pi:\R^3 \to \R^2$ the projection of the link diagram. Then: 
\begin{itemize}
    \item 
a move of type $\widetilde{\mathcal{W}}_{a,d}^{b,c}$ is said to be \emph{really empty} if it is empty, $\pi([a,b,c,d]) \cap \pi(L)=\pi([a,d])$ and $[a,b,c,d] \cap \ax$ is either empty or an endpoint of $[a,d]$ lying in $\ax$;

\item a move of type $\W_{a,d}^{b,c}$ that does not go through $\infty$ is said to be \emph{really empty} if it is empty and $$\pi([a,c,\tau(a),b, \tau(b),d]) \cap \pi(L)=\pi([a,d])\cup\pi([d,\tau(a)]);$$ 
\item a move of type $\W_{a,d}^{b,c}$ that goes through $\infty$ is said to be \emph{really empty} if it is empty, $\pi(L\setminus ([a,d] \cup \tau([a,d])))\cap \pi([a,b,d])=\emptyset,$ and, if we call $Q$ the quadrilateral with vertices $b,c,d,\tau(b)$, then $\pi(L)\subset \pi(Q)$.
\end{itemize}
\end{dfn}

\begin{dfn}
    The move $\Wtil_{a,d}^{b,c}$ is said to be \emph{of crossing type} if the segments $[a,b]$ and $[d,c]$ form a crossing in the projection (see \Cref{fig:crossingWtil}). Otherwise is said to be \emph{non-crossing}.
\end{dfn}

\begin{dfn}
    A move $\W_{a,d}^{b,c}$ such that $c \notin [a,b,\tau(a),\tau(b)]$ is said to be 
    \begin{itemize}
        \item \emph{of type A} if the projections of the points $b$ and $c$ in $\R^2$ lie on the same side with respect to the projection of the line containing $ a$ and $\tau(a)$ (see for example the second picture of \Cref{fig: typeAw});
        \item \emph{of type B} if the projections of the points $b$ and $c$ in $\R^2$ lie on opposite sides with respect to the projection of the line containing $ a$ and $\tau(a)$ (see for example the second picture of \Cref{fig:reallyemptywinftybb1});
    \end{itemize}
\end{dfn}

\begin{prop} \label{prop: really empty}
\begin{enumerate}
    \item An admissible move of type $\Wtil$ factorizes into admissible $\Rtil$ and really empty  $\Wtil$ moves of crossing type.
    \item An admissible move of type $\W$ that does not go through $\infty$ factorizes into admissible $\Rcal,\Rtil$ and a really empty $\W$ move of type A. 
    \item An admissible move of type $\W$ that goes through $\infty$ factorizes into admissible $\Rcal,\Rtil$ and a really empty $\W$ move of type B.

\end{enumerate}
\end{prop}

\begin{proof} Recall that we denote the diagram projection by $\pi$.
\begin{enumerate}[wide, labelwidth=0pt, labelindent=0pt] 
    \item 
Let us deal with the move $\Wtil_{a,d}^{b,c}$. By \Cref{prop: empty} we can assume $\Wtil_{a,d}^{b,c}$ is empty. If it is not of crossing type, let $x$ be a point very close to $d$, such that: 
\begin{itemize}
\item the projection of $x$ lies in the half-plane in $\R^2$ delimited by the line containing the projection of $[c,d]$ that does not contain $\pi(b)$; 
\item the projection of $[x,b] $ and $[c,d]$ intersect in a crossing point; 
\item $xd$ is positive. 
\end{itemize}
One can see that such an $x$ always exists (see \Cref{fig:xregion}), and that we can pick it as close to $d$ as we want.
\begin{figure}
    \centering
\includegraphics{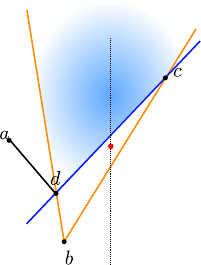}
    \caption{Any point $x$ projecting to the blue region has the properties that $xd$ is positive and that $xb$ and $cd$ intersect in the projection. Notice that the blue region is always nonempty, hence $x$ can be chosen to be any point close enough to $d$ that projects to that blue region.}
    \label{fig:xregion}
\end{figure}

Now we can see that $\Wtil_{a,d}^{b,c}=(\Rtil_{a,b}^x)^{-1} \circ \Wtil_{x,d}^{b,c} \circ \Rtil_{a,d}^x$ (see \Cref{fig:crossingWtil}).
\begin{figure}
    \centering
    \includegraphics[width= 0.9 \textwidth]{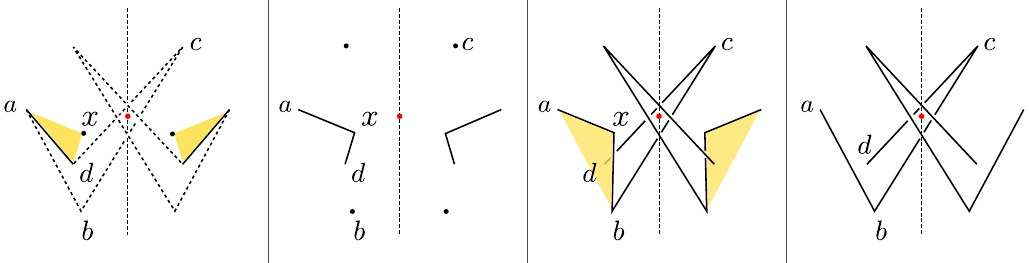}
    \caption{An empty move of type $\Wtil$ factorizes into moves of type $\Rtil$ and empty $\Wtil$ moves of crossing type.}
    \label{fig:crossingWtil}
\end{figure}
One can notice that, if $x$ is close enough to $d$, all the edges in the moves on the right-hand side have the required sign, and they are admissible because the starting move was empty. It is clear that $\Wtil_{x,d}^{b,c} $ is now of crossing type, and that it is still empty. 

Hence we can suppose that $\Wtil_{a,d}^{b,c}$ is an empty move of crossing type. Since the move is empty, $[a,b,c,d] \cap \ax $ is either empty or one of the endpoints of $[a,d]$ when it lies in $\ax$. Hence by \Cref{lem: convex} $[a,b,c,d] \cap \tau([a,b,c,d])= \emptyset$ or it is equal to an endpoint of $[a,d]$ when it lies in $\ax$. 

Notice that $[a,b,c,d] \cap \bx\neq \emptyset$. Since the link $L$ does not intersect the $\bx$ axis, we can find a ball $B$ inside $[a,b,c,d]$ centered on a point $p$ of the axis $\bx$ such that the projection of this ball is disjoint from $L$. Let $\theta$ be a homothety centered at $p$ so that $\theta([a,b,c,d])\subset B$. Let $a'$ be close to $\theta(a)$, $b'$ be close to $\theta(b)$, $c'$ be close to $\theta(c)$ and $d'$ be close to $\theta(d)$ so that $aa',b'b ,cc',d'd$ are all positive edges. 

Notice that (see \Cref{fig: reallyemptywtil}),  
$$ \Wtil_{a,d}^{b,c}= (\Rtil_{b,c}^{c'})^{-1} \circ(\Rtil_{b,c'}^{b'})^{-1}\circ\Rtil_{c',d}^c\circ\Rtil_{a,b'}^b\circ(\Rtil_{c',d}^{d'})^{-1}\circ(\Rtil_{a,b'}^{a'})^{-1}\circ\Wtil_{a',d'}^{b',c'}\circ\Rtil_{a',d}^{d'}\circ\Rtil_{a,d}^{a'}.$$
\begin{figure}
    \centering
    \includegraphics[width = 0.3 \textwidth]{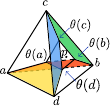}
    \caption{A generic picture of $[a,b,c,d]$ and of $\theta[a,b,c,d]$.}
    \label{fig:tetraedro}
\end{figure}
The moves are all admissible because the starting move was empty and because they are performed very close to the coloured quadrilaterals in \Cref{fig:tetraedro}, and the interiors of these quadrilaterals are pairwise disjoint. It is not hard to check that all the edges have the required sign. 
\begin{figure}
    \centering
    \includegraphics[width= 0.95 \textwidth]{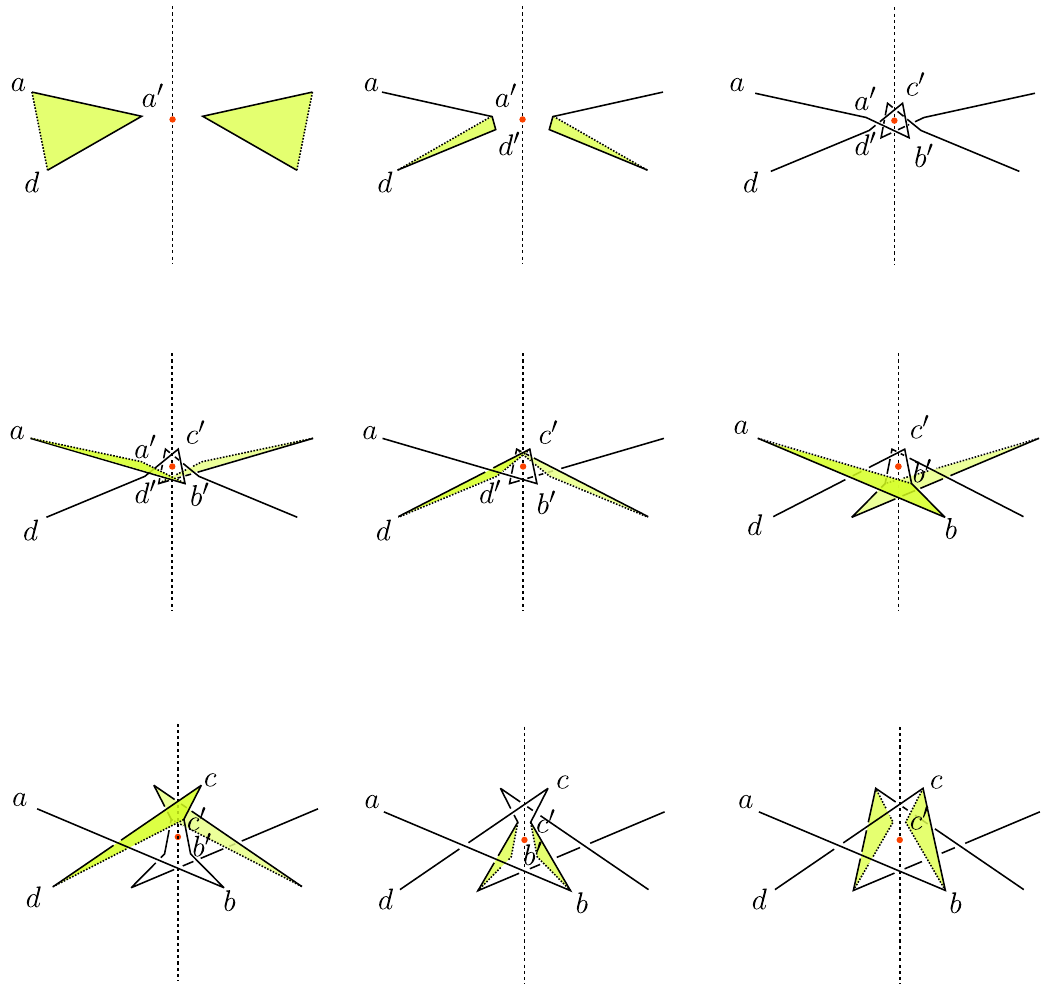}
    \caption{The factorization of an empty $\Wtil$ move of crossing type into admissible $\Rtil$ and a $\Wtil$ move of really empty type.}
    \label{fig: reallyemptywtil}
\end{figure}

 \item By \Cref{prop: empty} we will suppose that $\W_{a,d}^{b,c}$ is empty. 
 Let us show that we can furthermore assume that $\W_{a,d}^{b,c} $ is of type $A$. If it is of type $B$, let $a' \in [a,d]$ be close to $d$. Let $b' \in [a,b,d]$ be close to $d$ such that $ a'b'$ and $d b'$ are positive and $\pi(b')$ lies on the same side of $\pi(c)$ with respect to the projection of the line containing $a'$ and $\tau(a')$.
Notice that (see \Cref{fig: typeAw}), $$ \W_{a,d}^{b,c}= (\Rtil_{a,b}^{a'})^{-1}\circ (\Rtil_{a',b}^{b'})^{-1}\circ \Rtil_{b',c}^b\circ \W_{a',d}^{b',c} \circ \Rtil_{a,d}^{a'},$$
all the moves on the right-hand side are well-defined and admissible for a generic choice of $a'$ and $b'$. 
\begin{figure}
    \centering
    \includegraphics[width= 0.7 \textwidth]{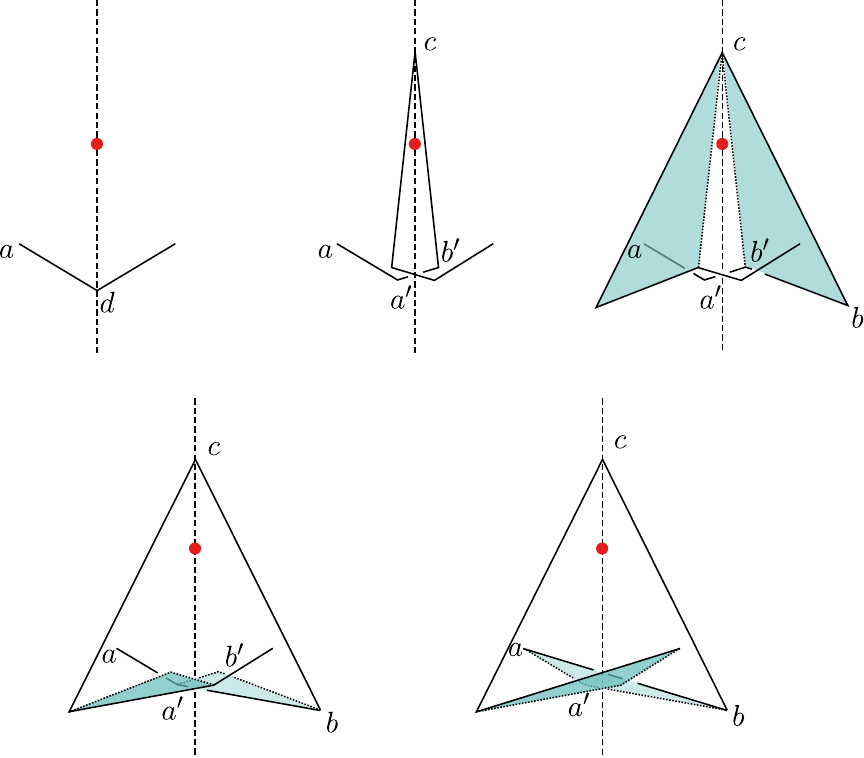}
    \caption{The factorization of an empty $\W$ move of type B into admissible $\Rtil$ and an empty $\W$ move of type A.}
    \label{fig: typeAw}
\end{figure}
The move $\W_{a',d}^{b',c} $ is still empty, $c \notin [a',b',\tau(a'),\tau(b')]$, and is of type A by construction.
 
We now assume that $\W_{a,d}^{b,c}$ is empty and of type A. As before, we can find a ball $B$ around $0$ inside $[a,d,\tau(a),b,c,\tau(b)]$ that does not intersect $L$ in the projection. There is a homothety $\theta$ centered at $0$ such that $\theta([a,d,\tau(a),b,c,\tau(b)]) \subset B $. Let $d'=\theta(d),c'=\theta(c)$ and let $a',b'$ be very close to $\theta(a)$ and $\theta(b)$ respectively so that $aa',b'b$ are positive. Then (see \Cref{fig:reallyemptyw})
$$ \W_{a,d}^{b,c}= (\Rtil_{a,b}^{b'})^{-1}\circ \Rtil_{b',c}^b \circ \Rcal_{b',c'}^c\circ(\Rtil_{a,b'}^{a'})^{-1}\circ\W_{a',d'}^{b',c'}\circ \Rtil_{a,d'}^{a'}\circ \Rcal_{a,d}^{d'}.$$ 

\begin{figure}
    \centering
    \includegraphics[width=0.8 \textwidth]{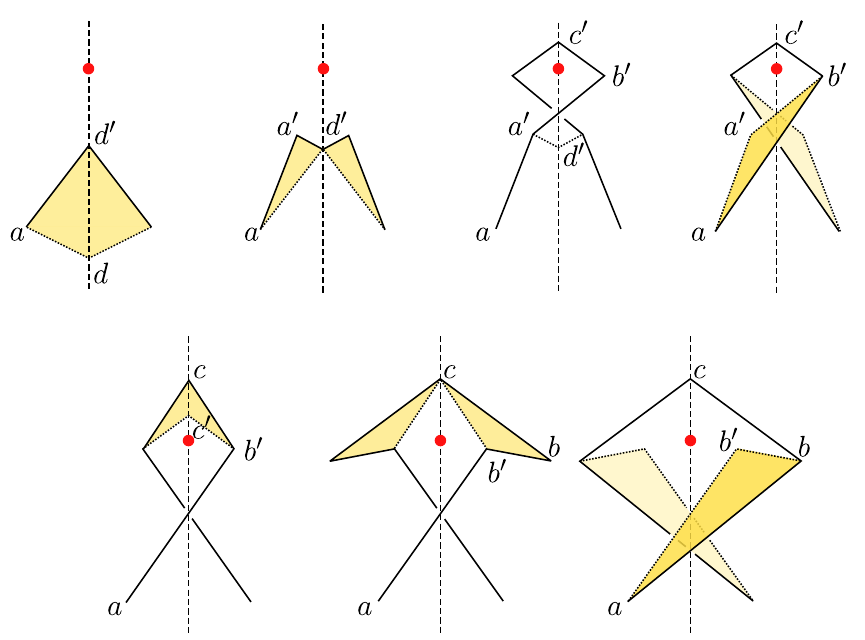}
    \caption{The factorization of an empty $\W$ move of type A into admissible $\Rcal$ and $\Rtil$ moves and a really empty $\W$ move of type A.}
    \label{fig:reallyemptyw}
\end{figure}
It is easy to check that all the moves are well-defined, and they are all admissible since the starting move was empty.
\item  %one can see that, in the proof of \Cref{prop: empty}, if we pick $K$ big enough and $p$ accordingly, the move $\W_{p,Kd}^{q,Kc}$ is really empty.
By \Cref{prop: empty} we will suppose that $\W_{a,d}^{b,c}$ is empty.

We first show that we can assume that $ \W_{a,d}^{b,c}$ is of type B: let $K,H>1$ and $c'=Kc $ and $d' =Hd$ so that $c'$ and $d'$ do not belong to $[a,b,c,\tau(b),\tau(a)]$. 

Fix $b'$ in the plane containing $\ax$ and $b$, so that 
\begin{itemize}
    \item the projection of the line containing $a$ and $\tau(a)$ separates $\pi(c')$ and $\pi(b')$; 
    \item $ab'$ and $db'$ are positive;
    \item and $b, c$ belong to  $[b',c',d']$. 
\end{itemize}
It is not hard to see that such $b'$ always exists. 

Up to increasing $H$ we can fix $b'$ such that $\pi([a,d',b']) \cap \pi([a,b,c,\tau(b), \tau(a), d])$ is equal to $\{\pi(a)\}$ (see \Cref{fig:reallyemptywinftyb}).  
\begin{figure}
    \centering
    \includegraphics[width=0.4\textwidth]{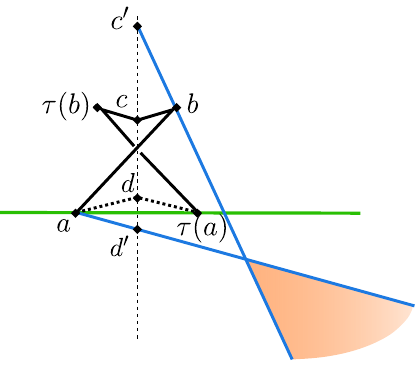}
    \caption{If one picks $b'$ such that it projects to a point in the orange region, then $b'$ satisfies all the requests.}
    \label{fig:reallyemptywinftyb}
\end{figure}
Then if $bb'$ is positive (see \Cref{fig:reallyemptywinftybb1}), $$ \W_{a,d}^{b,c}=\Rcal_{b,c'}^c\circ(\Rtil_{b,c'}^{b'})^{-1} \circ \Rtil_{a,b'}^b \circ\W_{a,d'}^{b',c'} \circ \Rcal_{a,d}^{d'}.$$

If $bb'$ is negative (see \Cref{fig:reallyemptywinftyb1b}) 
$$ \W_{a,d}^{b,c}=\Rcal_{b,c'}^c\circ(\Rtil_{a,b}^{b'})^{-1} \circ \Rtil_{b',c'}^b \circ\W_{a,d'}^{b',c'} \circ \Rcal_{a,d}^{d'}. $$

\begin{figure}
    \centering
    \includegraphics[width=0.8 \textwidth]{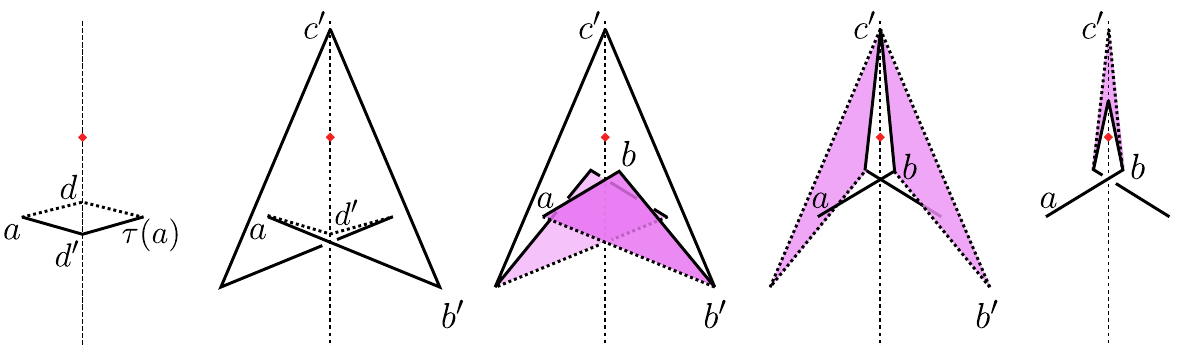}
    \caption{The factorization of an empty $\W$ move through $\infty$ into $\Rcal$ and $\Rtil$ moves and a $\W$ move of type B. This is the case of $bb'$ positive.}
    \label{fig:reallyemptywinftybb1}
\end{figure}

\begin{figure}
    \centering
    \includegraphics[width=0.8 \textwidth]{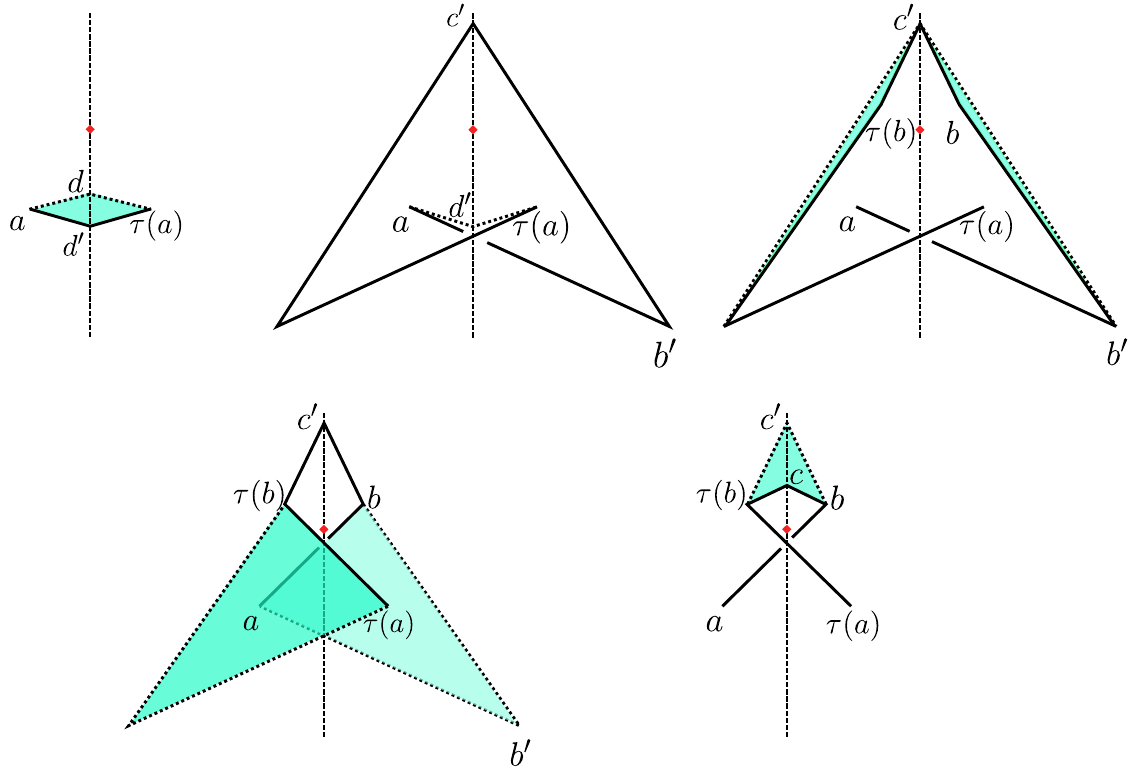}
    \caption{The factorization of an empty $\W$ move through $\infty$ into $\R$ and $\Rtil$ moves and a $\W$ move of type B. This is the case of $bb'$ negative.}
    \label{fig:reallyemptywinftyb1b}
\end{figure}

In both cases it is easy to check that the moves on the right-hand side are well-defined. We show that they are admissible. The only non obvious ones are $ (\Rtil_{b,c'}^{b'})^{-1} $, $\Rtil_{a,b'}^b$ in the first case and $(\Rtil_{a,b}^{b'})^{-1} $, $\Rtil_{b',c'}^b $ in the second case. 

Notice that $[a,b,b'] \cap \ax = \emptyset$. In fact $b$ and $b'$ lie on a plane containing $\ax$, and $a$ cannot lie on the same plane (otherwise $\W_{a,d}^{b,c}$ would not be well-defined in the first place). Since $[b,b'] \cap \ax = \emptyset$ (in fact $bc'$ and $b'c'$ are both positive) we conclude that indeed $ [a,b,b'] \cap \ax = \emptyset$.

Observe as well that $ [b,b',c']\cap \ax = \{c'\}$. In fact $bc' $ and $b'c'$ are both positive.

We now want to show that $ [a,b,b'] \cap [a,b,c',\tau(b),\tau(a),d'] $ coincides with $[a,b]$ and $ [b,b',c']\cap [a,b,c',\tau(b),\tau(a),d']$ coincides with $[b,c']$. Since $L$ is contained in $[a,b,c,\tau(b),\tau(a),d] $, it is contained in $ [a,b,c',\tau(b'),\tau(a),d']$, and this will imply that the above moves are admissible. 

First notice that, since $c'$ and $d'$ do not belong to $ [a,b,\tau(a),\tau(b)]$, one has that the convex set $[a,b,c',\tau(b'),\tau(a),d']$ is equal to $ C \cup   \tau(C)$
with $C=[a,b,c',d',\tau(b)] $. Notice that $C$ is the cone of the quadrilateral $Q=[b,c',d',\tau(b)] $ with vertex $a$. Since $ [b,b'] \cap Q = \{b\}$, we have that $[a,b,b'] \cap C= [a,b]$ and $[a,b,b'] \cap \tau(C)= \{b\} $, hence the conclusion in this case.
Now notice as well that $[b,b',c']\cap Q = [b,c']$ since $b \in [b',c',d']$, hence $[b,b',c'] \cap C \cup \tau(C)= [b,c']$. 

We can finally assume that the move $\W_{a,d}^{b,c}$ is empty and of type B. To conclude we just have to notice that by construction, $\pi(\Rcal_{a,d}^{d'}(L)\setminus ([a,d'] \cup \tau([a,d'])))\cap \pi([a,b',d'])=\emptyset,$ and,  that we can pick $H,K$ big enough, and $b'$ accordingly so that $\pi(L)\subset \pi(Q)$.
\end{enumerate}
\end{proof}

\subsection{Step 3: proof of \Cref{thm: markovlink}} The main result in this step is \Cref{thm: markovlink}. In order to prove it, we want to show that if $L$ and $L'$ are equivariantly isotopic and $h(L)=h(L')=0$ then we can obtain $L'$ from $L$ applying only admissible $\Rcal,\Rtil,\W,\Wtil$ moves. The proof of this fact is divided in two parts:
\begin{enumerate}[label=(\roman*)]
    \item We show in \Cref{prop: isolatemax} that we can obtain $ L'$ from $L$ with possibly another sequence of moves of type $\D, \Etil,\Stil,\Rcal,\Rtil, \W, \Wtil$ such that every move of this sequence is either applied to a link with $h=0$, or it modifies the link by adding a different number of negative edges than the ones that are removed, hence changing the value of $h$. The construction will guarantee that this new sequence of moves will not increase the maximum value of $h$ realized by this sequence.
    
    \item We show in \Cref{prop: reducemin} that we can modify the sequence of the previous part reducing isolated maxima.

\end{enumerate}
Applying iteratively \Cref{prop: isolatemax} and \Cref{prop: reducemin} one obtains the desired result. 
As a consequence, this will prove \Cref{thm: markovlink}.
\subsubsection{Proof of part (i)} 
We now proceed with the proof of Part (i).

\begin{nota}
    If the link $L'$ can be obtained by $L$ with one admissible move of type $\D,$ $\Etil,$ $\W,$ $\Wtil,$ $\Rcal,$ $\Rtil, $ $\Stil$ or the inverse of one of these moves, we write that $L \to L'$.
\end{nota}

Before proving \Cref{prop: isolatemax} we will need the following result (as explained in \Cref{rmk: concatenate}):

\begin{prop} \label{prop: sawtooth}
    If $\Stil_{p_0,\ldots,p_m}^{q_1,\ldots,q_m}$ and $\Stil_{r_0,\ldots,r_n}^{s_1,\ldots,s_n}$ are two admissible sawtooth moves such that $p_0p_m=r_0r_n=e$, then $\Stil_{r_0,\ldots,r_n}^{s_1,\ldots,s_n} \circ (\Stil_{p_0,\ldots,p_m}^{q_1,\ldots,q_m})^{-1}$ factorizes into admissible $\Rtil$ and $\Wtil$ moves and their inverses.
 \end{prop}

To prove this, let us first define an auxiliary move, called $\Ttil$, and let us first prove a simplified version of \Cref{prop: sawtooth}, that is stated in \Cref{lem: Ttil}.

\begin{dfn}
We define the move $\Ttil_{a,b,d}^c\coloneqq \Stil_{a,d}^c\circ (\Stil_{a,d}^b)^{-1}$ when the moves on the right-hand side are well-defined. We say that the $\Ttil$ move is admissible when the moves on the right-hand side are admissible and $ [a,b,d] \cap [\tau(a),\tau(d),\tau(c)]\subset [a,d]\cap \tau([a,d])$.
\end{dfn}

\begin{lem} \label{lem: Ttil}
An admissible $\Ttil$ move on a strongly involutive link $L$ factorizes into admissible $\Rtil$ and $\Wtil$ moves and their inverses.
\end{lem}

\begin{proof}
Fix the move $\Ttil_{a,b,d}^c$. We suppose for simplicity that $[a,d]\cap \ax = \emptyset$, the case where one of the endpoints of $[a,d]$ lies in $\ax$ being analogous.
% \begin{figure}
%     \centering
%     \includegraphics[width= 0.5 \textwidth]{Tmove.pdf}
%     \caption{The two possible cases in the proof of \Cref{prop: Ttil}. Notice that in this case the red dot is the $\ax$ axis and the vertical line is the $\bx$ axis.}
%     \label{fig:Tmove}
% \end{figure}
\begin{figure}
        \centering
        \includegraphics[width= 0.63 \textwidth]{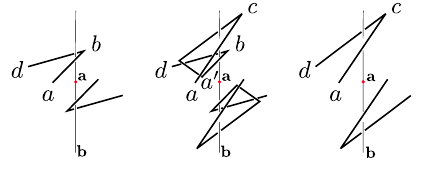}
        \caption{The proof of \Cref{lem: Ttil}. Notice that the red dot is the $\ax$ axis and the vertical line is the $\bx$ axis.}
        \label{fig:Ttilcaso1}
    \end{figure}
    Pick $a' $ close to $a$ and $[a,b]$ so that $aa', a'b$ are positive, $[a,b,a'] \cap L=[a,b]$ and $[a,b,a']\cap \ax= \emptyset$. Let $d'$ be close to $d$ and $[c,d]$ so that $cd',d'd$ are positive, $[c,d,d'] \cap L=[c,d]$ and $[c,d,d'] \cap \ax= \emptyset$. Notice that $\Rtil_{a,b}^{a'}$ and $\Rtil_{c,d}^{d'}$ are admissible moves. Moreover, for a generic choice of $ a'$ and $d'$, the moves $\Wtil_{a,a'}^{c,d'}$ and $\Wtil_{d',d}^{a',b}$ are well-defined.
    One can see that in this case (see \Cref{fig:Ttilcaso1})
    $$\Ttil_{a,b,d}^c=(\Rtil_{c,d}^{d'})^{-1}\circ (\Wtil_{d,d'}^{a',b})^{-1} \circ \Wtil_{a,a'}^{c,d'}\circ \Rtil_{a,b}^{a'}.$$

\vspace{-27pt}
\end{proof}

We are now ready to prove \Cref{prop: sawtooth}.

 \begin{proof}[Proof of \Cref{prop: sawtooth}]
Let $\Stil_{p_0,\ldots,p_m}^{q_1,\ldots,q_m}$ and $\Stil_{r_0,\ldots,r_n}^{s_1,\ldots,s_n}$ be two admissible sawtooth moves such that $p_0p_m=r_0r_n=e$. We want to show that $\Stil_{r_0,\ldots,r_n}^{s_1,\ldots,s_n} \circ (\Stil_{p_0,\ldots,p_m}^{q_1,\ldots,q_m})^{-1}$ factorizes into admissible $\Rtil$ and $\Wtil$ moves. Let $L$ be the link obtained after applying the move $ (\Stil_{p_0,\ldots,p_m}^{q_1,\ldots,q_m})^{-1}$ (hence $e \subset L$).
\textbf{ Case 1:} suppose that $([p_i,p_{i+1},q_{i+1}]\setminus[p_i,p_{i+1}]) \cap \tau ([r_k,r_{k+1},s_{k+1}] \setminus[r_k,r_{k+1}] ) = \emptyset$ for all $i,k$.
 
     Suppose that $r_j \notin \{p_0,\ldots,p_m\}$, and suppose it is in the interior part of $[p_i,p_{i+1}]$. 
     \begin{figure}
         \centering
         \includegraphics[width= 0.43 \textwidth]{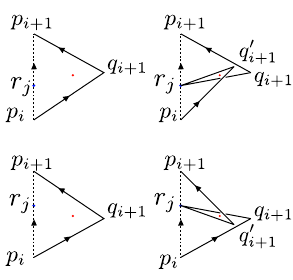}
         \caption{Adding a basepoint to the sawtooth in Case 1. The case where $q_{i+1}r_j$ is negative (above) and the case where $q_{i+1}r_j$ is positive (below).}
         \label{fig:corollario}
     \end{figure}
     Up to changing $q_{i+1}$ with another point in a neighbourhood of $q_{i+1}$ with an admissible $\Ttil$ move, by \Cref{lem: Ttil} move we can suppose that $q_{i+1}r_j $ is either positive or negative. \\
     If $q_{i+1}r_j$ is positive, then pick $q'_{i+1}$ very close to the triangle $[p_i,p_{i+1},q_{i+1}]$ so that $r_jq_{i+1}'$ and $q_{i+1}'p_{i+1}$ are positive. It is easy to check that the edge $q_{i+1}'q_{i+1} $ is positive. 

     In fact $\pi(p_i)$ and $\pi(r_j)$ must lie on the same side with respect to the line containing $0$ and $\pi(q'_{i+1})$. Since $ 0 \in \pi([p_i,q_{i+1},r_j])$ and $ r_jq'_{i+1}$ are positive, this implies that $q_{i+1}'q_{i+1} $ is positive.
     
     Hence $$ (\Stil_{p_0,\ldots,p_m}^{q_1,\ldots,q_m})^{-1}     = (\Stil_{p_0,\ldots,p_i,r_j,p_{i+1},\ldots, p_m}^{q_1,\ldots,q_{i+1},q'_{i+1},q_{i+2},\ldots,q_m} )^{-1}\circ \Wtil_{q_{i+1},p_{i+1}}^{r_{j}, q'_{i+1}}$$
     (see bottom of \Cref{fig:corollario}).

     If $q_{i+1}r_j$ is negative, then pick $q'_{i+1}$ very close to the triangle $[p_i,p_{i+1},q_{i+1}]$ so that $p_{i}q_{i+1}'$, and $q_{i+1}'r_j$ are positive. One can check as before that the edge $q_{i+1}q_{i+1}' $ is positive. Hence
     $$\left(\Stil_{p_0,\ldots,p_m}^{q_1,\ldots,q_m}\right)^{-1}     = \left(\Stil_{p_0,\ldots,p_i,r_j,p_{i+1},\ldots, p_m}^{q_1,\ldots,q'_{i+1},q_{i+1},q_{i+2},\ldots,q_m} \right)^{-1}\circ \Wtil_{p_{i},q_{i+1}}^{q'_{i+1},r_{j}}$$
     (see top of \Cref{fig:corollario}).
     Proceeding in this way for all $r_j \notin \{p_0,\ldots,p_m\}$ one can suppose that $\{r_0,\ldots,r_n\} \subset \{p_0,\ldots,p_m\}$. Notice that if we pick $q'_i$ close enough to $ [p_i,p_{i+1},q_{i+1}]$ then the initial requirement on the mutual intersection of the triangles is still valid.
     
     If one applies a symmetrical argument to the $p_i \notin \{r_0,\ldots,r_n\}$ one can actually suppose that $\{r_0,\ldots,r_n\} = \{p_0,\ldots,p_m\}$.

     If $m=n$ and $p_i=r_i$ for all $i$, then $\Stil_{r_0,\ldots,r_n}^{s_1,\ldots,s_n} \circ (\Stil_{p_0,\ldots,p_m}^{q_1,\ldots,q_m})^{-1}$ is the composition of $\Ttil_{p_i,q_{i+1},p_{i+1}}^{s_{i+1}}$ for $i=0,\ldots,n-1$, that are well-defined and admissible because of our initial hypothesis on the triangles still holds. By \Cref{lem: Ttil} these moves all factorize into $\Rtil$ and $\Wtil$ moves.

     \textbf{Case 2:} suppose that there exist $i_0,k_0$ such that $$([p_{i_0},p_{i_0+1},q_{i_0+1}]\setminus[p_{i_0},p_{i_0+1}]) \cap \tau ([r_{k_0},r_{k_0+1},s_{k_0+1}] \setminus[r_{k_0},r_{k_0+1}] ) $$ is not empty and $e$ and $\tau(e)$ do not project to a double point in the diagram.  
     
     For all $i=0,\ldots,m-1$ let $a_i$ be the intersection point of $T_i=[p_i,p_{i+1},q_{i+1}]$ and $ \bx $, and for all $k=0,\ldots,n-1$ let $b_k$ be the intersection point of $U_j=[r_j,r_{j+1},s_{j+1}] $ and $ \bx $. 
     Let $q_{i+1}' \in T_i$ be close to $a_i$ such that $p_iq'_{i+1} $ and $q'_{i+1} p_{i+1}$ are positive and let $s'_{j+1} \in u_j$ be very close to $b_j$ such that $r_js'_{j+1} $ and $s'_{j+1} p_{j+1}$ are positive. 
     Notice that $\Stil_{p_0,\ldots,p_m}^{q'_1,\ldots,q'_m} \circ (\Stil_{p_0,\ldots,p_m}^{q_1,\ldots,q_m})^{-1}$ and $\Stil_{r_0,\ldots,r_n}^{s_1,\ldots,s_n} \circ (\Stil_{r_0,\ldots,r_n}^{s'_1,\ldots,s'_n})^{-1} $ fall into the previous case, hence we would like to prove the theorem for $ (\Stil_{r_0,\ldots,r_n}^{s'_1,\ldots,s'_n})^{-1}\circ \Stil_{p_0,\ldots,p_m}^{q'_1,\ldots,q'_m}.$

     Call $T'_i= [p_i,p_{i+1},q'_{i+1}]$ and $ U'_j= [r_j,r_{j+1},s'_{j+1}]$.
     
     The $a_i$'s and the $b_j$'s are a finite number of points in $\bx$, and there exists a limited closed segment $l$ in $\bx \setminus \{\infty\}$ that contains all of them.
     We would like to construct an admissible sawtooth move on $e$ such that the triangles of the move do not intersect in their interior the set $C\cup \tau(C)$, where $C$ is the convex hull of $e$ and $l$.
     For every $x \in e$ let $r_x$ be the line parallel to $\bx\setminus \{\infty\}$ and containing $x$. Let $S= \bigcup_{x \in e}r_x$. Notice that $S \cap \tau(C) \subset \partial e$. In fact $\pi(\tau(C))$ is contained in one of half-planes delimited by $\pi(\ax)$ and $\pi(S)=\pi(e)$ is contained in the other one.

     Moreover $C \cap S = e$. In fact $0 = \pi(l)$ is contained in the interior of one of the two half-planes delimited by the line through $\pi(e)$. This implies that $l $ is contained in the interior of one of the two half-spaces delimited by the plane containing $S$. Hence $C$ is contained in this half-space and $C$ intersects the delimiting plane only in $e$, thus $C   \cap S=e$.

     Notice that if one rotates the set $S$ of a small angle $\theta$, fixing $e$ (as formalized in the proof of \Cref{lem: sawtooth}), the resulting set $\widetilde{S}$ has the same intersection with $C$ and $\tau(C)$ as $S$.

     With the same methods of the proof of \Cref{lem: sawtooth} we can find $t_0,\ldots,t_k \in e$ and $u_1,\ldots, u_k$ such that $t_0t_k=e$ and $\Stil_{t_0,\ldots,t_h}^{u_1,\ldots,u_h}$ is admissible on $L $.
     If we call $V_l=[t_l,t_{l+1},u_{l+1}]$, then $ V_l\cap C= [t_l,t_{l+1}]$ and $V_l\cap \tau(C)\subset \partial e$ for $l=0,\ldots,k-1$. 
     
       Up to taking the $q'_i $'s and $s'_j$'s very close to $\bx$, the triangles $T'_i$ and $U'_j$ are also disjoint from the $V_j$'s and $\tau(V_j)$'s by construction. Hence 
      $$\Stil_{r_0,\ldots,r_n}^{s'_1,\ldots,s'_n}\circ (\Stil_{p_0,\ldots,p_m}^{q'_1,\ldots,q'_m})^{-1} = (\Stil_{r_0,\ldots,r_n}^{s'_1,\ldots,s'_n}\circ (\Stil_{t_0,\ldots,t_k}^{u_1,\ldots,u_k})^{-1}) \circ (\Stil_{t_0,\ldots,t_k}^{u_1,\ldots,u_k}\circ (\Stil_{p_0,\ldots,p_m}^{q'_1,\ldots,q'_m})^{-1}) $$
      and by construction the right-hand side is the composition of two operations that fall into Case 1. We therefore get the conclusion in this case.
      
     \textbf{Case 3:} suppose that there exist $i_0,k_0$ such that $$([p_{i_0},p_{i_o+1},q_{i_0+1}]\setminus[p_{i_0},p_{i_0+1}]) \cap \tau ([r_{k_0},r_{k_0+1},s_{k_0+1}] \setminus[r_{k_0},r_{k_0+1}] ) $$ is not empty and $e$ and $\tau(e)$ project to a double point in the diagram. 
     
     Let $ T_i=[p_i,p_{i+1},q_{i+1}]$ and $ U_j= [r_j,r_{j+1},s_{j+1}]$. For $x \in e$, let $r_x$ be the line containing $x$ and parallel to $\bx \setminus \{\infty\}$. Let $S= \bigcup_{x \in e} r_x$. Notice that $e$ separates $S$ in two components, let $hS$ be the closure of the component that does not intersect $\ax$. By the proof of \Cref{lem: sawtooth} we can find $ t_0,\ldots, t_k \in e$, $ u_1,\ldots, u_k$ such that $t_0t_k=e$ and $\Stil_{t_0,\ldots,t_h}^{u_1,\ldots,u_h}$ is admissible on $L $ and $V_l= [t_l,t_{l+1},u_{l+1}]$ is contained in a rotation of $hS$ by a small angle $\theta_l$. If the $\theta_l$'s are small enough we can suppose that there exist $l$ such that $\tau(T_i) \cap V_l \neq \emptyset$ if and only if $ \tau(T_i) \cap hS\neq \emptyset$, and that $\tau(U_j) \cap V_l \neq \emptyset$ if and only if $ \tau(U_j) \cap hS\neq \emptyset$. 
     Notice that $$\Stil_{r_0,\ldots,r_n}^{s_1,\ldots,s_n} \circ (\Stil_{p_0,\ldots,p_m}^{q_1,\ldots,q_m})^{-1} = (\Stil_{r_0,\ldots,r_n}^{s_1,\ldots,s_n}\circ (\Stil_{t_0,\ldots,t_h}^{u_1,\ldots,u_h})^{-1}) \circ (\Stil_{t_0,\ldots,t_h}^{u_1,\ldots,u_h}\circ (\Stil_{p_0,\ldots,p_m}^{q_1,\ldots,q_m})^{-1}). $$
     We now deal with $\F=\Stil_{t_0,\ldots,t_h}^{u_1,\ldots,u_h}\circ (\Stil_{p_0,\ldots,p_m}^{q_1,\ldots,q_m})^{-1}$. Notice that $ \Stil_{r_0,\ldots,r_n}^{s_1,\ldots,s_n}\circ (\Stil_{t_0,\ldots,t_h}^{u_1,\ldots,u_h})^{-1}$ is the inverse of an operation with the same characteristics as $\F$, hence if we are able to factorize $\F$ into admissible $\Rtil$ and $\Wtil$ moves we are done.

     Let $s\subset \R^3$ be the line containing $e$ and let $H_0\subset \R^2$ be the half-plane delimited by $\pi(s)$ and containing $0$, and $H_1$ be the half-plane delimited by $\pi(\tau(s))$ and containing $0$, i.e. the half-plane obtained reflecting $H_0$ along $\pi(\ax)$.
     We pick as in the previous case $ q'_{i+1}\in T_i$ and $ u'_{l+1}\in V_l$ close to $\bx$ for every $i,l$. We want them close enough to $\bx$ so that $ \pi(q'_i)$ and $\pi(u'_l)$ lie in $ H_0 \cap H_1$. Let $T'_i= [p_i,p_{i+1},q'_{i+1}]$ and $V'_l= [t_l,t_{l+1},u'_{l+1}]$.

     Notice that $\Stil_{p_0,\ldots,p_m}^{q'_1,\ldots,q'_m} \circ (\Stil_{p_0,\ldots,p_m}^{q_1,\ldots,q_m})^{-1}$ and $\Stil_{t_0,\ldots,t_k}^{u_1,\ldots,u_k} \circ (\Stil_{t_0,\ldots,t_k}^{u'_1,\ldots,u'_k})^{-1} $ fall into Case 1, hence we would like to prove the theorem for $ (\Stil_{t_0,\ldots,t_k}^{u'_1,\ldots,u'_k})^{-1}\circ \Stil_{p_0,\ldots,p_m}^{q'_1,\ldots,q'_m}.$

     Let $I \subset \{0,\ldots,m-1\}$ be the set of $i$'s such that $\tau(T'_i) \cap hS \neq \emptyset$. 
     Let $p \in e$ be the point that projects via $\pi$ to a double point with $\tau(p) \in \tau(e)$. Let $\overline{i} \in \{0,\ldots,m-1\}$ be such that $ p \in \operatorname{Int}([p_{\overline{i}},p_{\overline{i}}])$. Such $\overline{i}$ always exist since the double points in the diagram cannot correspond to vertices. Notice that $\overline{i} \notin I$. In fact $\tau(p) \in S \setminus hS$ and if $\tau(T'_{\overline{i}}) \cap hS \neq \emptyset$ then by convexity $\tau(T'_{\overline{i}}) \cap e \neq \emptyset$, which is impossible. 

     If $\pi([p_i,p_{i+1}]) \subset H_1$ then $i \notin I$. In fact this implies that $ \pi(\tau(p_i)), \pi(\tau(p_{i+1}))$ and $\pi(\tau(q'_{i+1}))$ lie in the interior of $H_0$, hence $ \pi(\tau(T'_i)) \cap \pi(hS)= \emptyset$, and thus $ T'_i \cap hS= \emptyset$. 

     Now let $i_1 \in I$. We know that $ \pi([p_{i_1},p_{i_1+1}]) \cap H_1 = \emptyset. $
     Notice that, for every $x \in [p_{i_1},p_{i_1+1}]$, $\tau(r_x) \cap T'_{i}=\tau(r_x) \cap V'_l= \emptyset$. In fact $ \pi(T'_i)$ and $\pi(V'_l)$ are contained in $H_0$ for every $i$, and $ \pi(r_x)= \pi(x) \notin H_1$ implies that $ \pi(\tau(r_x)) \notin H_0$.
     
     Suppose for now that for every $i \in I$ there are no double points for $\pi$ in $[p_{i},p_{i+1}]$. 
     Hence (as in the proof of \Cref{lem: sawtooth}) we can find a point $q''_{i_1+1}$ such that the triangle $T''_{i_1}=[p_{i_1},p_{i_1+1},q''_{i_1+1}]$ is contained in a rotation of a very small angle $\theta$ of $ \bigcup_{x\in [p_{i_1},p_{i_1+1}]} r_x$, and $p_{i_1}q_{i_1+1}$ and $q_{i_1+1}p_{i_1+1} $ are positive, $ T''_{i_1}\cap \ax = \emptyset$ and $ T''_{i_1} \cap L = [p_{i_1},p_{i_1+1}]. $
     
     Moreover if the angle $\theta$ is small enough,  $\tau(T''_{i_1})\cap T'_i= \tau(T''_{i_1})\cap V'_l = \emptyset $.
     Apply this procedure to every $i \in I$ to find $q''_{i+1}$, and for every $i \notin I$, let $q''_{i+1} = q'_{i+1}$.

     We have that $$ \Stil_{t_0,\ldots,t_k}^{u'_1,\ldots,u'_k}\circ (\Stil_{p_0,\ldots,p_m}^{q'_1,\ldots,q'_m})^{-1}= (\Stil_{t_0,\ldots,t_k}^{u'_1,\ldots,u'_k} \circ (\Stil_{p_0,\ldots,p_m}^{q''_1,\ldots,q''_m})^{-1}) \circ ( \Stil_{p_0,\ldots,p_m}^{q''_1,\ldots,q''_m} \circ (\Stil_{p_0,\ldots,p_m}^{q'_1,\ldots,q'_m})^{-1}) $$
     and by construction the right-hand side is the composition of two operations that fall into Case 1. 
     
     In the case where there exist $i \in I$ and $y \in [p_i,p_{i+1}]$ double point for $\pi$, we can subdivide $[p_i,p_{i+1}] $ as in the proof of \Cref{lem: sawtooth} in smaller segments. More precisely, we find $p_i^0, \ldots, p_i^{h_i} \in [p_{i},p_{i+1}]$ such that $p_{i}p_{i+1}=p_i^{1} p_i^{h_i} $, and $ {q}_{i+1}^1,\ldots , {q}_{i+1}^{h_i}$ such that $[p_i^{j},p_i^{j+1},{q}_{i+1}^{j+1}]$ lies in a rotation of a very small angle $\theta_j$ of $\bigcup_{x \in [p_i,p_{i+1}]} r_x$, the segments
     $ p_i^{j}{q}_{i+1}^{j+1}$ and ${q}_{i+1}^{j+1} p_i^{j+1}$ are positive, 
     $[p_i^{j},p_i^{j+1},{q}_{i+1}^{j+1}]\cap \ax = \emptyset$ and $[p_i^{j},p_i^{j+1},{q}_{i+1}^{j+1}]\cap L = [p_i^{j},p_i^{j+1}]$. Moreover if $\theta_l$ is small enough $\tau([p_i^{j},p_i^{j+1},{q}_{i+1}^{j+1}]) \cap T'_i = \tau([p_i^{j},p_i^{j+1},{q}_{i+1}^{j+1}]) \cap V'_l = \emptyset $. 
     Hence up to adding to $\{p_0,\ldots,p_m\}$ the points $p_i^j$ and to $ \{{q''}_1,\ldots, {q''}_m\}$ the points ${q}_{i+1}^j$ for every $i \in I$ as above, for $j=i,\ldots ,h_i$.
This concludes the proof.
 \end{proof}

\begin{rmk}\label{rmk: concatenate}
    We will need \Cref{prop: sawtooth} in the proof of \Cref{prop: isolatemax} for the following reason: if we have $L \xrightarrow{\F} L'$ with $h(L)=h(L')>0$, sometimes we will need to factorize $\F$ into a composition of other moves $ \F= \F_1 \circ \ldots \circ \F_n $, with $\F_i$ preserving $h$ for every $i$. We will then find for every $i$ two $\Stil$ moves $\Stil_1^i,\Stil_2^i$ such that 
    $$ \begin{tikzcd}\left(\F_{i+1}\circ \cdots \circ \F_n (L)\right) \arrow[r,"\Stil_1^i"] & L_1 \arrow[r] &\cdots \arrow[r] & L_k \arrow[r,"(\Stil_2^i)^{-1}"] & \left(\F_i\circ \cdots \circ \F_n(L) \right)\end{tikzcd}$$ 
    with $h(L_i) < h(L). $ Moreover, for every $i=1,\ldots n-1$, the moves $\Stil_2^i$ and $\Stil_1^{i+1}$ will be applied to the same edges. Thus, by \Cref{prop: sawtooth}, $\Stil_1^{i+1}\circ(\Stil_2^i)^{-1} $ factorizes into admissible $\Rtil$ and $\Wtil$ moves, and recall that these moves preserve $h$. Hence we will be able to concatenate the sequences obtained for every $\F_i$, obtaining a sequence
    $$ L \to L_1' \to \ldots \to L'_m \to L'$$ such that $h(L'_j) < h(L)$ for $j= 1,\ldots, m.$
\end{rmk}

We are now ready to prove the main result of part (i) of Step 3.
\begin{prop}\label{prop: isolatemax}
If $L\to L'$ and $h(L)=h(L')>0$ then there exist $L_1,\ldots, L_n$ such that  $L\to L_1\to \ldots \to L_n\to L'$
and $h(L_i) < h(L)$ for $i=1,\ldots,n$.
\end{prop}

\begin{proof}
Let us call $\Ftil$ the move that transforms $L$ into $L'$. The only moves that preserve $h$ are $\Rcal,\Rtil,\W,\Wtil$, $ \Etil$ when it replaces a pair of negative edges with two positive and two negative edges, and $\D$ when it substitutes two negative edges with other two negative edges (and the inverses of these moves).

For what concerns $\Rcal,\Rtil,\W,\Wtil$, since they are applied on positive edges, there must be an edge $e$ not involved in the move that is negative. 
\begin{itemize}[wide, labelwidth=0pt, labelindent=0pt]
    \item In the case of $\Ftil=\Rtil_{a,c}^b$, for any $x \in e$ let $r_x$ be the line parallel to $\bx$ and containing $x$. 
Suppose for now that no $x \in e$ is such that $ \pi(x)=\pi(\tau(x))$, i.e. $\pi(x) \notin \pi(\ax \setminus \{\infty\})$. Then the projection of $e$ has positive distance $d$ from $\pi(\ax \setminus \{\infty\})$. 

By \Cref{lem: reduceRtil} and \Cref{rmk: concatenate} we can suppose that the triangle $[a,b,c]$ has diameter strictly less than $ d$. 
Hence the sheaf of lines $ \bigcup_{x \in e} r_x$ intersects at most one of the two triangles $[a,b,c]$ and $\tau([a,b,c])$, let us suppose without loss of generality that it does not intersect $\tau([a,b,c])$.

Now one proceeds as in the proof of \Cref{lem: sawtooth} to find $p_0,\ldots,p_{n} \in e$ and $q_1,\ldots, q_n\notin L$ such that $p_0p_n=e$ and $\Stil_{p_0,\ldots,p_n}^{q_1,\ldots,q_n} $ is an admissible sawtooth move such that for $i=0,\ldots,n-1$ $ [p_i,p_{i+1},q_{i+1}]$ does not intersect the convex set $ [a,b,c]$. One has to realize that in that construction we can pick the triangles $ [p_i,p_{i+1},q_{i+1}]$ to lie in a sheaf of lines obtained by rotating $\bigcup_{x \in e} r_x$ of a very small angle $\theta$. If $\theta$ is small enough then this rotated sheaf of lines will not intersect $\tau([a,b,c]).$ This implies that $\Rtil_{a,c}^b$ is admissible on $L_1=\Stil_{p_0,\ldots,p_n}^{q_1,\ldots,q_n}(L)$, and $(\Stil_{p_0,\ldots,p_n}^{q_1,\ldots,q_n})^{-1}$ is admissible on $ L_2= \Rtil_{a,c}^b(L_1).$ Hence 
$L \xrightarrow{\Stil_{p_0,\ldots,p_n}^{q_1,\ldots,q_n}} L_1 \xrightarrow{\Rtil_{a,c}^b} L_2 \xrightarrow{(\Stil_{p_0,\ldots,p_n}^{q_1,\ldots,q_n})^{-1}} L' $ and $h(L_1)=h(L_2) < h(L)$.

Suppose now that there exists $x \in e$ such that $\pi(x)=\pi(\tau(x)).$ Then, either $x = \tau(x) \in \ax$, or $ x$ and $\tau(x)$ project to a double point in the diagram and therefore $x$ lies in the interior of $e$. In any case, $x\notin [a,c]$ and therefore $x$ has positive distance $d$ from $[a,b,c]$. 
Let $hr_x \subset r_x$ be an half-line with endpoint $x$ that does not intersect $\ax$ in its interior. If $hr_x \cap [a,b,c]$ and $ hr_x \cap \tau([a,b,c])$ are both empty, we can proceed as in the proof of \Cref{lem: sawtooth} to construct a ``tooth" $[y,z,w]$ (i.e. a triangle) of a sawtooth move such that $x \in [z,w]$ and $[y,z,w]$ does not intersect $[a,b,c]$ nor $\tau([a,b,c])$. %Moreover we can ask, as in the proof of \Cref{lem: sawtooth}, that $ \bigcup_{t\in e \setminus{[z,w]}} r_t$ does not intersect $\tau([y,z,w])$.

Otherwise, we can suppose by \Cref{lem: reduceRtil} and \Cref{rmk: concatenate} that the diameter of the interior of $[a,b,c]$ is less than $d$. Then there exists at least a point $p \in \bx\setminus \{\infty\}$ such that $[x,p]$ does not intersect $[a,b,c]$ nor $\tau([a,b,c])$ (see \Cref{fig: avoidRtil}). Moreover we can pick $p$ so that $[x,p]$ does not intersect $\ax$ outside of $x$. 

\begin{figure}
    \centering
    \includegraphics[width=0.3\textwidth]{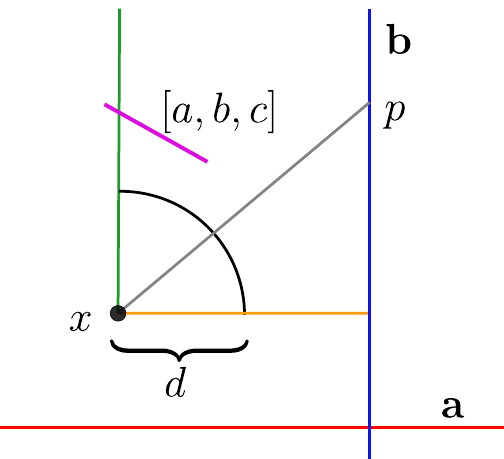}
    \caption{If the diameter of $[a,b,c]$ is less than $d$, there exists $p\in \bx\setminus\{\infty\}$ such that $[x,p] \cap [a,b,c]= \emptyset$.}
    \label{fig: avoidRtil}
\end{figure}

In this case, up to substituting $p$ with another point in a neighbourhood such that $[x,p]$ has the same properties as before and moreover $[x,p] \cap L= \{x\}$, we can construct the triangle $[y,z,w]$ of the sawtooth move in a neighbourhood of $[x,p]$, similarly to the proof of \Cref{lem: sawtooth}. We can now construct the triangles of an admissible sawtooth move on the remaining parts of the edge $e$ as in the previous case. We obtain points $p_0,\ldots,p_n \in e$, $q_1,\ldots,q_n \notin L$, such that $p_0p_n=e$, $ \Stil_{p_0,\ldots,p_n}^{q_1,\ldots,q_n}$ is admissible and for $i=0,\ldots n-1$ the triangle $ [p_i,p_{i+1},q_{i+1}]$ does not intersect $ [a,b,c]$ nor $\tau([a,b,c])$. Hence as before, $L \xrightarrow{\Stil_{p_0,\ldots,p_n}^{q_1,\ldots,q_n}} L_1 \xrightarrow{\Rtil_{a,c}^b} L_2 \xrightarrow{(\Stil_{p_0,\ldots,p_n}^{q_1,\ldots,q_n})^{-1}} L' $ and $h(L_1)=h(L_2) < h(L)$.

    \item In the case $\F=\Rcal_{a,c}^b$, by \Cref{lem: reduceR}, \Cref{rmk: concatenate} and the previous step of the proof we can suppose that the diameter of $[a,b,c, \tau(b)]$ is small enough so that it is disjoint from $e$. 
    
    By \Cref{lem: sawtooth}, we can find an admissible sawtooth move $\Stil_{p_0,\ldots,p_n}^{q_1,\ldots,q_n}$, such that $e=p_0p_n$ and such that $[a,b,c, \tau(b)] \cap [p_i,p_{i+1},q_{i+1}]$ is empty. Hence $L \xrightarrow{\Stil} L_1 \xrightarrow{\Ftil} L_2 \xrightarrow{\Stil^{-1}} L'$ and $h(L_1)$, $h(L_2) <h(L)$.

 \item If $\F=\W_{a,d}^{b,c}$ and this move does not go through $\infty$, by \Cref{prop: really empty}, \Cref{rmk: concatenate} and the previous steps of the proof, we can suppose that it is a really empty move. Thus, by \Cref{lem: sawtooth} we can find an admissible sawtooth move on $e$ that avoids the support of $\W_{a,d}^{b,c}$. Hence $L \xrightarrow{\Stil} L_1 \xrightarrow{\Ftil} L_2 \xrightarrow{\Stil^{-1}} L'$ and $h(L_1)$, $h(L_2) <h(L)$.

\item If $\F=\W_{a,d}^{b,c}$ that goes through $\infty$, we can suppose by \Cref{prop: really empty} and \Cref{rmk: concatenate} that it is a really empty move. Then
the sheaf $S$ of lines parallel to $\bx$ and intersecting $e$ does not intersect $[a,b,d] \cup \tau([a,b,d])$. Call $Q$ the quadrilateral with vertices $b,\tau(b),c,d$ and $\alpha$ the plane containing $Q$. Then also $S \cap (\alpha \setminus Q)$ is empty. Thus, as in the proof of \Cref{lem: sawtooth}, we can find an admissible sawtooth move on $e$ very close to $S$ so that every triangle of the move avoids $[a,b,d] \cup \tau([a,b,d]) $ and $\alpha \setminus Q $. Hence $L \xrightarrow{\Stil} L_1 \xrightarrow{\Ftil} L_2 \xrightarrow{\Stil^{-1}} L'$ and $h(L_1)$, $h(L_2) <h(L)$.

\item For $\Ftil=\Wtil_{a,d}^{b,c} $, one can suppose by \Cref{prop: really empty}, \Cref{rmk: concatenate} and the proof of the previous step, that it is a really empty move. Hence the projection of $e$ does not intersect the projection of this move. As a consequence, one can find an admissible sawtooth move $\Stil_{p_0,\ldots,p_n}^{q_1,\ldots,q_n}$ such that $p_0p_n=e$, and the triangles $[p_i,p_{i+1},q_{i+1}]$ are very close to being parallel to $\bx$. The construction is similar to the one in the proof of \Cref{lem: sawtooth}. Also in this case $L \xrightarrow{\Stil_{p_0,\ldots,p_n}^{q_1,\ldots,q_n}} L_1 \xrightarrow{\Wtil_{a,d}^{b,c}} L_2 \xrightarrow{(\Stil_{p_0,\ldots,p_n}^{q_1,\ldots,q_n})^{-1}} L'$ and $h(L_1),h(L_2)<h(L)$.

\item Let us deal now with the case $\F= \Etil^b_{a,c}$. Without loss of generality, $ac$ and $ab$ are negative while $bc$ is positive. We suppose for simplicity that $[a,c] \cap \ax$ is empty, the case where one of the endpoints lies in $\ax$ being analogous.
Suppose initially that $b$ and $c$ are close enough so that there are points $b_0,\ldots,b_n$ in $[a,b]$, such that $b_0=a$, $b_n=b$ and for $i=0,\ldots n-1$, $b_ib_{i+1}$ is negative,
$c_0,\ldots,c_n$ in $[a,c]$ such that $c_0=a$, $c_n=c$ and for $i=0,\ldots n-1$ $c_ic_{i+1}$ is negative, $d_1,\ldots,d_n \notin L$ with the property that $[b_i,b_{i+1},c_i,c_{i+1},d_{i+1}]\cap L=[c_i,c_{i+1}]$  and the edges $c_id_{i+1}$, $d_ic_i$, $b_id_{i+1}$, $d_ib_i$ are positive for $i=0,\ldots n-1$.

We also ask that, $$ \bigcup_{i=0}^{n-1}[b_i,b_{i+1},c_i,c_{i+1},d_{i+1}]\cap \bigcup_{i=0}^{n-1}[\tau(b_i),\tau(b_{i+1}),\tau(c_i),\tau(c_{i+1}),\tau(d_{i+1})]= \emptyset.$$ If $b$ is close enough to $c$ both conditions can be satisfied: in fact by \Cref{lem: sawtooth} there exists an admissible sawtooth move on $[a,c]$, that we call $\Stil_{c_0,\ldots,c_n}^{d_1,\ldots, d_n}$, and this implies that for $i, j=0,\ldots,n-1$ the set $[c_i,c_{i+1},d_{i+1}] \cap \tau([c_j,c_{j+1},d_{j+1}])$ is empty. Hence, if $b$ is close enough to $a$ we can have that all the conditions above are satisfied with this choice of $c_i$'s and $d_i$'s. 

In this case $$\Etil_{a,c}^b=(\Stil_{b_0,\ldots,b_n}^{d_1,\ldots,d_n})^{-1}\circ \mathcal{F}_n\circ \ldots \circ \mathcal{F}_{1} \circ \Stil_{c_0,\ldots,c_n}^{d_0,\ldots,d_n}$$
where the $\mathcal{F}_i$ are defined as:
\begin{itemize}
    \item when $c_{i}b_i$ is positive, for $0<i<n$ then $\mathcal{F}_i=(\Rtil_{d_i,b_{i}}^{c_i})^{-1}\circ \Rtil_{c_i,d_{i+1}}^{b_i}$ (\Cref{fig:Fm1});
    
    \item  when $c_{i}b_i$ is negative, for $0<i<n$, 
    $\mathcal{F}_i=(\Rtil_{b_i,d_{i+1}}^{c_i})^{-1}\circ \Rtil_{d_i,c_{i}}^{b_i}$ (\Cref{fig:Fm2});
    
\item when $i=n$, $\F_n=\Rtil_{d_n, c}^b$ (\Cref{fig:Fmn}).
 
\end{itemize}

\begin{figure}[H]
        \centering
        \includegraphics[width= 0.6 \textwidth]{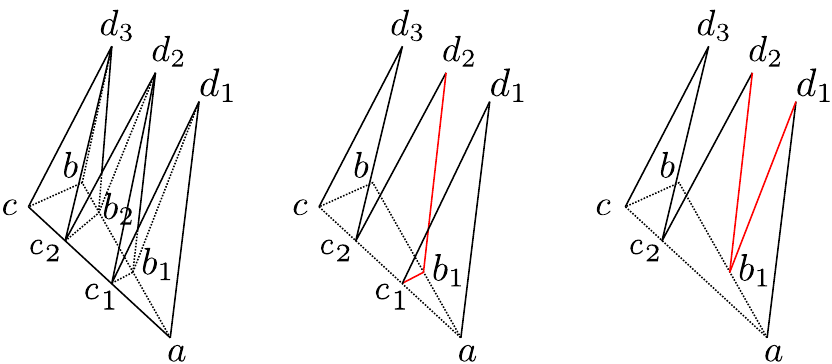}
        \caption{The move $\F_i$ when $c_{i}b_i$ is positive. To make the picture more clear, we just drew a schematic picture of the moves on $[a,c]$ and not on $[\tau(a),\tau(c)]$.}
        \label{fig:Fm1}
    \end{figure}

    \begin{figure}[H]
        \centering
        \includegraphics[width= 0.6 \textwidth]{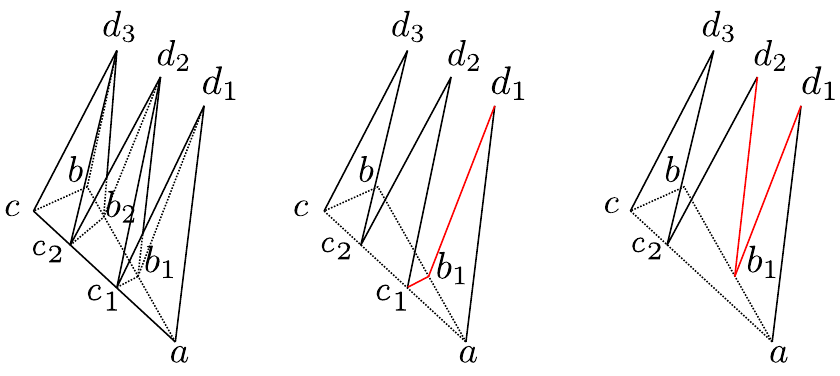}
        \caption{The move $\F_i$ when $c_{i}b_i$ is negative. As in \Cref{fig:Fm1}, to make the picture more clear, we just drew a schematic picture of the moves on $[a,c]$ and not on $[\tau(a),\tau(c)]$.}
        \label{fig:Fm2}
    \end{figure}

    \begin{figure}[H]
        \centering
        \includegraphics[width= 0.4 \textwidth]{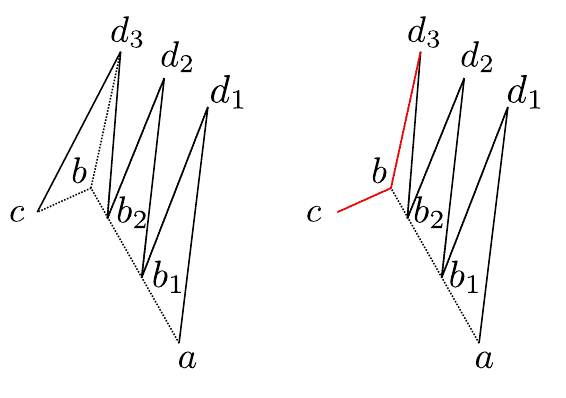}
        \caption{The move $\F_n$. As in \Cref{fig:Fm1} and \Cref{fig:Fm2}, to make the picture more clear, we just drew a schematic picture of the moves on $[a,c]$ and not on $[\tau(a),\tau(c)]$}.
        \label{fig:Fmn}
    \end{figure}

Notice that in this case 
$$ 
\begin{tikzcd}
L \arrow[r,"\Stil_{a_0,\ldots,a_n}^{d_0,\ldots,d_n}"] & 
L_1 \arrow[rr, bend left, dashed, "\F_{1}"] \arrow[r] & L_2 \arrow[r] & L_3 \arrow[r] & \cdots \arrow[r]&  L_{2n-1}  \arrow[r, " \F_n"]  & L_{2n} \arrow[r,"(\Stil_{b_0,\ldots,b_n}^{d_1,\ldots,d_n})^{-1}"] & L' \\
& 
\end{tikzcd} $$
and $h(L_i) < h(L)$ for $i=1,\ldots ,2n$.

Now, in the case $b$ and $c$ are not close enough, we can find a set of points $x_0,\ldots,x_k\in [b,c]$, so that $x_0=b$, $x_k=c$, $x_ix_{i+1}$ is positive and $x_i$ and $x_{i+1}$ are close enough in the above sense for $i=0,\ldots,k-1$. This can be done since $[b,c]$ is compact. 

Now notice that $\Etil_{a,c}^b= \Etil_{a,x_{k-1}}^{x_k} \circ \ldots \circ \Etil_{a,x_0}^{x_1}$, that $ax_i$ is negative for every $i$, that $x_ix_{i+1}$ is positive for every $i$ and that on every move on the right-hand side we can apply the procedure above. This, together with \Cref{rmk: concatenate}, lets us conclude the proof of this case.

\item We are left with the case of a $\D_{a,c}^b$, where $ac$ and $ab$ are negative. If $\D_{a,c}^b$ goes through $\infty$, then let $d \in \ax \setminus \{\infty\}$ and $d \notin [b,c]$ such that $ad$ is positive. Then 
$ L \xrightarrow{\D_{a,c}^d} L_1 \xrightarrow{\D_{a,d}^b} L'$ and $h(L_1) < h(L).$

If $\D_{a,c}^b $ does not go through $\infty$, we treat this case similarly to the case of a $\Etil$ move. More precisely: suppose that $b$ and $c$ are close enough so that there are points $b_0,\ldots,b_n \in [a,b]$ such that $b_0=a$, $b_n=b$ and, for $i=1,\ldots,n-1$, $b_ib_{i+1}$ is negative, $c_0,\ldots,c_n \in [a,c]$ such that $c_0=a$, $c_n=c$ and $c_ic_{i+1}$ is negative for $i=1,\ldots,n-1$, $d_0,\ldots, d_n\notin L$ with the property that $[b_i,b_{i+1},c_i,c_{i+1},d_{i+1}] \cap L =[c_i,c_{i+1}]$ and $ c_id_{i+1},d_ic_i,b_id_{i+1},d_ib_i$ are positive. We also ask that 
$$ \bigcup_{i=0}^{n-1}[b_i,b_{i+1},c_i,c_{i+1},d_{i+1}]\cap \bigcup_{i=0}^{n-1}[\tau(b_i),\tau(b_{i+1}),\tau(c_i),\tau(c_{i+1}),\tau(d_{i+1})]= [b,c] \subset \ax.$$
As before, if $b$ is close enough to $c$ then all the above conditions are satisfied. In this case 
$$ \D_{a,c}^b= (\Stil_{b_0,\ldots,b_n}^{d_1,\ldots,d_n})^{-1}\circ \F_{n}\circ\ldots \circ \F_{1} \circ \Stil_{c_0,\ldots,c_n}^{d_1,\ldots,d_n},$$
where the $\F_i$ are defined as in the preovious case for $0<i<n$ and $\F_n= \D_{d_n, c}^b $.

As above, find $x_0, \ldots, x_k \in [b,c]$, so that $x_0=b$, $x_k=c$ and $x_i$, $x_{i+1}$ are close enough in the above sense for $i=0,\ldots,k-1$. Notice that $\D_{a,c}^b= \D_{a,x_{k-1}}^{x_k} \circ \ldots\circ \D_{a,x_0}^{x_1}.$ Thus we can apply the above procedure to each move on the right-hand side, and we conclude the proof of this last case by \Cref{rmk: concatenate}.
\end{itemize}
\vspace{-1cm}
\end{proof}

\subsubsection{Proof of part (ii)}

We now proceed with the proof of Part (ii). Recall that we have to prove that if $ L_{1}\to L_2 \to L_{3}$ and $h(L_{2})>h(L_{1}),h(L_{3}$, then we can modify we find $L_1',\ldots, L_n'$ with $h(L'_j)< h(L_2)$ such that $ L_{1}\to L_1' \to \cdots\to L_n' \to L_{3}$. This will be done in \Cref{prop: reducemin}.

For the proof of \Cref{prop: reducemin} we will find useful the following:

\begin{lem} \label{lem: sawteeth}
    Let $\Stil_{a_0,\ldots,a_n}^{b_1,\ldots,b_n}$ be an admissible sawtooth move and let $e= a_0a_n$. Then if $e' \neq e, \tau(e)$ there always exists $\Stil_{c_0,\ldots,c_m}^{d_1,\ldots,d_m}$, with $e'=c_0c_m$ such that for $i=0,\ldots, n-1$ and $j=0,\ldots,m-1$ the intersections $[c_j,c_{j+1},d_{j+1}]\cap [a_i,a_{i+1},b_{i+1}]$ and $[c_j,c_{j+1},d_{j+1}]\cap  \tau([a_i,a_{i+1},b_{i+1}])$ are either empty or equal to $ e \cap (e' \cup \tau(e))$ (this can happen when $e$ and $e'$ or $e $ and $\tau(e')$ are consecutive edges on the link).
\end{lem}

In order to prove \Cref{lem: sawteeth}, we will need the following technical lemma:

\begin{lem} \label{lem: auxiliary}
    Let $x \in \R^3 \setminus \bx$ and let $T_1,\ldots, T_n$ be a set of triangles $T_i=[a_i,b_i,c_i]$, such that for every $i \neq j$ the set $(T_i \setminus \{a_i,b_i,c_i\}) \cap T_j $ is empty. Suppose as well that for $i=1,\ldots, n$ the set $T_i \cap \bx$ is equal to a point $t_i$ contained in the interior of $T_i$, that $\operatorname{Int} (T_i) \cap \ax = \emptyset$ and that $x \notin \operatorname{Int}(T_i)$. Then one can find $e_x \in \bx \setminus \{\infty\}$ such that $[x,e_x] \cap T_i\subset \{x\}$ for for $i=1,\ldots,n$. 
\end{lem}

\begin{proof}
    Let $\alpha_x$ be the plane containing $\bx$ and $x$. One has that $\ax$ is either contained in $\alpha_x$ or $\ax \cap \alpha_x = \{0\}$.

    Consider the intersection of the triangles $T_i$ with the plane $\alpha_x$. 
    \begin{figure}
        \centering
        \includegraphics[width=0.3 \textwidth]{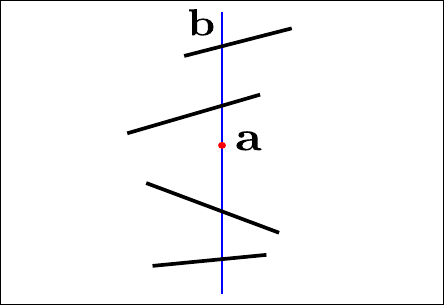}
        \caption{The plane $\alpha_x$ in the proof of \Cref{lem: auxiliary}. The blue segments represent the intersections of the triangles $T_i$ with $\alpha_x$. }
        \label{fig:triangolini}
    \end{figure}
    These must be segments intersecting the line $\bx$ (see \Cref{fig:triangolini}).

    %If $\ax \subset \alpha_x$ consider just the $T_i$ that lie in the same half-plane as $x$ with respect to $\ax$.
    
    Notice that if $x$ lies on the same plane as one of the $T_i$'s, and the segment $[x,t_i]$ does not intersect any other $T_j$, then we can find the desired segment between $x$ and $\bx$ just perturbing slightly $[x,t_i]$, since this will be disjoint from $T_i$ and and still disjoint from all other $T_j$ (see \Cref{fig:allineato}). 
    \begin{figure}
        \centering
        \includegraphics[width=0.3 \textwidth]{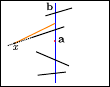}
        \caption{If the point $x$ lies in the same line as one of the segments $T_i \cap \alpha_x$ then the orange line in the picture connects $x$ to $\bx$ without intersecting the $T_j$'s.}
        \label{fig:allineato}
    \end{figure}
    We will therefore suppose that $[x,t_i] $ is transverse to every $T_j$.
    
     Consider a half-line $hr_x$ parallel to $\bx$ with endpoint $x$. 
     If $\ax \subset \alpha_x$ consider the half-line avoiding $\ax$.
     
     If this half-line does not intersect any $T_i$, then by slightly rotating the half-line inside $\alpha_x$ we find a segment connecting $x$ and $\bx$, disjoint from the triangles. 
     
     In the other case, when the half-line intersects at least one triangle, let $T_{i_1}$ be the first triangle one meets travelling from $x$ along the half-line, and let $s_{1} \in T_{i_1} $ be the intersection point (notice that $T_{i_1}$ might not be unique, in that case we pick one of the possible triangles).

    Now suppose we have defined $i_{j-1}$ and $s_{j-1}$. Consider the segment $[x,t_{i_{j-1}}]$. 
    If this does not intersect any other triangle than $T_{i_{j-1}}$, then very close to $t_{i_{j-1}}$ on $\bx$ one can find a point $e_x$ so that $[x,c_x]$ is disjoint from every $T_k$, from $\ax$, and from $L$, as desired. 
    In the other case, if $[x,t_{i_{j-1}}]$ actually intersects some triangles,
let $T_{i_{j}}$ be the first triangle (or one of them if the choice is not unique) one meets travelling from $x$ along $[x,t_{i_{j-1}}]$, and let $s_j \in T_{i_{j}}$ be the intersection point (see \Cref{fig:iterazione}). 
\begin{figure}
         \centering
    \includegraphics[width=0.3 \textwidth]{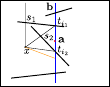}
         \caption{The procedure to define the $t_{i_j}$ and the $s_j$ in the proof of \Cref{lem: auxiliary}.}
         \label{fig:iterazione}
     \end{figure}

We want to say that there exists $j$ so that $[x,t_{i_j}]$ does not intersect any other triangle than $T_{i_{j}}$. Since the $T_i$'s are finite, if that is not the case then the only option is that the sequence of $i_j$ is definitively periodic.

We exclude this case by noticing that the $t_{i_j}$ are ordered on $\bx\setminus \{\infty\}$. 

In fact, without loss of generality one can suppose that for $i>2$, $t_{i_{j-1}}$ and $t_{i_j}$ lie on the $\bx$ axis as in \Cref{fig:casidaevitare}.
\begin{figure}
    \centering
    \includegraphics[width=0.3 \textwidth]{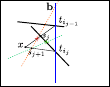}
    \caption{The $t_{i_j}$ must be ordered on the $\bx$ axis. If $t_{i_{j+1}}$ lies on the same side as $t_{i_{j-1}}$ with respect to $t_{i_j}$ then $[s_{j+1}, t_{i_{j+1}}] \subset T_{i_{j+1}}$ would intersect in its interior either $[s_j,t_{i_j}]\subset T_{i_{j}}$ (in the red cross) or $[x,s_j]$ (in the green cross).}
    \label{fig:casidaevitare}
\end{figure}

The point $t_{i_j}$ divides $\bx\setminus\{\infty\}$ in two half-lines, let $h$ be the one containing $t_{i_{j-1}}.$
Now, $t_{i_{j+1}}$ cannot belong to $h$, otherwise $[s_{j+1},t_{i_{j+1}}]$, which is contained in $T_{i_{j+1}}$, would intersect either $[s_{j},t_{i_j}]$,
or $[x,s_{j}]$.
The first case is impossible since $[s_{j},t_{i_j}]$ is contained in $T_{i_j}$, but $T_{i_j}$ and $T_{i_{j+1}}$ cannot intersect in their interior.
The second case is also impossible since $s_j$ would not be the first point of intersection of $[x,t_{i_j}]$ with one of the triangles $T_k$ (see \Cref{fig:casidaevitare}).
   
    Hence the $t_{i_j}$ are ordered on $\bx$ and, as a consequence there exists a $t_{i_j} $ so that $[x,t_{i_j}]$ does not intersect any of the $T_k$ for $k \neq i_j$. We can then conclude that near $[x,t_{i_j}]$ we can find a segment $\gamma_x=[x,e_x]$, disjoint from all the $T_k$'s, and $\ax$.
\end{proof}

We are finally ready to prove \Cref{lem: sawteeth}.

\begin{proof}[Proof of \Cref{lem: sawteeth}]
%Roughly speaking, we want to show that for every $x \in e'=[c,d]$, there exists a segment connecting $x$ to $\bx$ that avoids the link, the $\ax$ axis and all the triangles $[a_i,a_{i+1},b_{i+1}]$ and $ [\tau(a_i),\tau(a_{i+1}),\tau(b_{i+1})]$. 
    %In fact, if such a segment exists for every $x$, then one can find in a small neighbourhood of this segment points $p,r\in e'$ such that $x \in [p,r]\subset [c,d]$ and $pr$ is a negative edge and an open set $U_x$ inside this neighbourhood such that for every $q \in U_x$, the segments $pq$, $qr$ are positive and moreover $[p,q,r]$ does not intersect the link $L$, the axis $\ax$, $[a_i,a_{i+1},b_{i+1}]$ and $ [\tau(a_i),\tau(a_{i+1}),\tau(b_{i+1})]$. Then, since $e'$ is compact, one can find $c_0,\ldots,c_{m} \in [c,d]$, $d_{1},\ldots,d_m$ such that $\Stil_{c_0,\ldots,c_{m}}^{d_{1},\ldots,d_m}$ is a sawtooth move. In making all these choices we must be careful enough so that the resulting sawtooth move is admissible.
Notice that the triangles $T_i=[a_i,a_{i+1},b_{i+1}]$ and $T_{i+n}=\tau(T_i)$ satisfy the hypotheses of \Cref{lem: auxiliary}.
    For every $x \in e'$, let $e_x$ be the point given by \Cref{lem: auxiliary} so that the segment $\gamma_x=[x,e_x]$ does not intersect in its interior any of the triangles $T_i$ for $i=0,\ldots, 2n-1 $.
    %We now proceed to find such segment for every $x$.
    %We will suppose for simplicity from now on that $e$ and $e' \cup \tau(e')$, do not share any endpoints, since the proof below works with minor adjustments.
    
    %Let $\alpha_x$ be the plane containing $\bx$ and $x$. One has that $\ax$ is either contained in $\alpha_x$ or $\ax \cap \alpha_x = \{0\}$.

     %As a consequence, for every $x \in e'$ we have that $\ax \cap \alpha_x=\{0\}$.
    
    Notice that $\alpha_x \cap L$ is a finite set of points: hence a generic segment connecting $x$ and $\bx$ avoids these points. Up to slightly perturbing $e_x$ inside $\bx$ we can suppose that $[x,e_x]\cap L = \{x\}$.
Suppose first that for every $x \in e'$, $\alpha_x \cap \tau(e')$ is empty.

Hence, for every $x, y \in e'$ the segments $\gamma_x$ and $ \tau(\gamma_y)$ are either disjoint or $\gamma_x\cap \tau(\gamma_y) = \{e_x\}$. In fact they lie in different planes $\alpha_x$ and $\alpha_{\tau(y)}$ that intersect in $\bx\setminus \{\infty\}$.

Hence, for every $x \in e'$, we can find $p_x,r_x \in e'$ close to $x$, such that $ I_x=[p_x,r_x] $ is a neighbourhood of $x$, $ p_xr_x$ is negative, and an open set $U_x$ close to $e_x$ such that for every $ q \in U_x$ , the segments $p_zq, qr_z$ are positive and the triangle $[p_x,q,r_x]$ does not intersect the $L \setminus [p_x,r_x]$, the axis $\ax$ and all the triangles $T_i\setminus [p_x,r_x]$.
    By compactness, there exist $x_1,\ldots,x_m \in e'$, such that $x_ix_{i+1}$ is negative and $\bigcup_{i=1}^m I_{x_i}=e'$. 
    Up to slightly moving some of the $e_{x_i}$'s, we can suppose as well that $ \gamma_{x_i}$ and $\tau(\gamma_{x_j})$ are disjoint for every $i\neq j$.
    
    Let $c_0$ and $c_m$ be the extremes of $e'$ and for $i=1,\ldots m-1$ let $c_i \in I_{x_i}\cap I_{x_{i+1}}$ such that $ c_ic_{i+1}$ is negative. Pick $ d_i \in U_{x_i}$. Notice that for a generic choice of $d_i$ close to $e_{x_i}$, $[c_i,c_{i+1},d_{i+1}] \cap \tau([c_j,c_{j+1},d_{j+1}])\subset e' \cap \tau(e') $ for every $i,j$. Hence $\Stil_{c_0,\ldots,c_m}^{d_1,\ldots,d_m}$ is a well-defined and admissible sawtooth move on $L$ that has all the required properties.
    \vspace{1mm}
    
    Now suppose that the set $X$ of points $x\in e'$ such that $ x, \tau(x) \in \alpha_x$ is non-empty. Notice that this can happen if and only if $\alpha_x$ is orthogonal to $\ax$ or if $ \ax \subset \alpha_x$.
    
    Hence $X=\{t_1,\ldots, t_h\}$ is a finite set. Let $\gamma_{t_1}=[t,e_{t_1}]$ be the segment given by \Cref{lem: auxiliary} that avoids all the $T_i$'s. Notice that $\gamma_{t_1} \cap \tau(\gamma_{t_1}) \subset \{x\}$. 

    Let $p_{t_1},q_{t_1},r_{t_1}$ be very close to $\gamma_{t_1}$ so that $ [p_{t_1},r_{t_1}]$ is a neighbourhood of ${t_1}$ in $e'$, $p_{t_1}r_{t_1}$ is negative and $p_{t_1}q_{t_1}$, $p_{t_1}r_{t_1}$ are positive and $[p_{t_1},q_{t_1},r_{t_1}] \setminus [p_{t_1},r_{t_1}]$ does not intersect $L $, $\ax$, $[a_i,a_{i+1},b_{i+1}]$ and $ \tau([a_i,a_{i+1},b_{i+1}])$. 

    Notice that the set of triangles $\bigcup_{i=0}^{2n-1}\{T_i\} \cup \{[p_{t_1},q_{t_1},r_{t_1}], \tau([p_{t_1},q_{t_1},r_{t_1}]) \} $ still satisfies the hypothesis of \Cref{lem: auxiliary}. 
    
    Hence we can apply this procedure iteratively on every $t_j \in X$, and at every step we find a triangle $ [p_{t_j},q_{t_j},r_{t_j}]$ that avoids all the triangles $T_i$ and also $ [p_{t_l},q_{t_l},r_{t_l}]$ and $\tau([p_{t_l},q_{t_l},r_{t_l}])$ for every $l < j$. 
    Now notice that the closure of $e' \setminus \bigcup_{j=1}^h[p_{t_j},r_{t_j}]$ is a union of disjoint segments $I_1,\ldots, I_k$. 
    Notice that for every $x \in I_i$, then $\alpha_x \cap \tau(I_i)= \emptyset$. 
    
    In fact, if $\tau(y) \in \alpha_x$, and $y \neq x$ then we can see that there must be $z \in [x,y]$ such that $\tau(z) \in \alpha_z$, which is not possible. 

    To see that, let $n_x$ be a unit vector orthogonal to $\alpha_x$ such that $ \langle y, n_x\rangle $ is positive, and let $n_y = \tau(n_x)$, which is orthogonal to $\alpha_y$ since $\tau$ is an isometry of $\R^3$ with the standard scalar product. Suppose that $ \langle x, n_y\rangle$ is negative, the other case being analogous. 
    Let $z= (1-t)x +t y$, with $ t= \frac{\langle x, n_y \rangle }{\langle x, n_y \rangle- \langle y, n_x \rangle}.$
    Let $v_z\coloneqq n_x+n_y$. We show that $v_z^\perp =\alpha_z$. Of course $v_z$ is orthogonal to $\bx$, hence we want to see that $ v_z$ is orthogonal to $z$ as well. One can easily compute 
    $ \langle z,  v_z\rangle = (1-t)\langle x, n_y \rangle + t \langle y, n_x \rangle =0$.
    Moreover $$ \langle \tau(z) , v_z \rangle = (1-t)\langle \tau(x), n_x \rangle + t \langle \tau(y), n_y \rangle = (1-t)\langle x, n_y \rangle + t \langle y, n_x \rangle = \langle z , v_z\rangle=0, $$
    where the second equality follows from the fact that $\tau$ is an isometry. Hence $\tau (z) \in v_z^\perp = \alpha_z$.

    Then we apply iteratively the procedure of the first part of the proof to $I_i$ for $i=1,\ldots,k$, making sure to add the triangles that we find at every step to the set of the triangles to avoid, 
    i.e. we find $p^{i}_{0}, \ldots ,p^{i}_{k_{i}} \in I_{i}$ and 
    $q^{i}_{1},\ldots,q^{i}_{k_i}$, such that $p^{i}_{0} p^{i}_{k_i}=I_i$, 
    $ p^i_j p^i_{j+1}$ is negative, $ p^i_j q^i_{j+1}$ and $q^i_{j+1} p^i_{j+1}$ are positive. 
    Moreover $[p^i_j, p^i_{j+1},q^i_{j+1}] $ avoids 
    $[p_{t_j},q_{t_j},r_{t_j}]$ for every $t_j \in X$ and 
    $ [p^h_l, p^h_{l+1},q^h_{l+1}]$ for every $ j,l$ and for every $h<i$, and when $h=i$ and $j \neq l$.

    Hence these triangles define an admissible sawtooth move on $e'$, concluding the proof.
\end{proof}

We are now ready to prove the main goal of part (ii) of this third step.
\begin{prop}\label{prop: reducemin}
If two links $L$ and $L'$ are joined by a chain of deformation moves
$$L \to L'' \to L' $$
such that $h(L'')>h(L)$ and $h(L'')>h(L')$, then there is a chain of links
$$L=L_0 \to L_1\to \cdots \to L_n=L' $$
such that $h(L_i) < h(L'')$ for each $i$.
\end{prop}

\begin{proof}
The moves that increase $h$ are:
\begin{itemize}[wide, labelwidth=0pt, labelindent=0pt] 
    \item moves of type $\Stil^{-1}$;
    \item $\D$ moves that substitute two positive edges with two negative ones;
    \item $\Etil$ moves that substitute two edges (positive or negative) with four negative ones or that substitute two positive edges with two negative and two positive ones;
    \item $\Etil^{-1}$ moves that substitute four positive edges with two negative ones. This is a special case of $\Stil^{-1}$ and therefore we will not consider it separately.
\end{itemize}
As a consequence the possible chains that we have to consider are the following:
\begin{enumerate}
    \item first we apply a type $\Stil^{-1}$ move and then $\Stil$; \label{casouno}
    \item first we apply a type $\Etil$ move and then $\Stil$ (and the inverse);\label{casodue}
    \item first we apply a type $\D$ move and then $\Stil$ (and the inverse);\label{casotre}
    \item first we apply a type $\Etil$ move and then $\Etil^{-1}$;\label{casoquattro}
    \item first we apply a type $\D$ move and then $\Etil^{-1}$ (and the inverse);\label{casocinque}
    \item first we apply a type $\D$ move and then $\D^{-1}$;\label{casosei}
\end{enumerate}

By \Cref{lem: sawtooth} one can find an admissible sawtooth move $\Stil_{p_0,\ldots, p_n}^{q_1,\ldots,q_n}$ on $L''$. Hence
$$\begin{tikzcd}[column sep = large]
    L \arrow[r] & L'' \arrow[r, "\Stil_{p_0,\ldots, p_n}^{q_1,\ldots,q_n}" ] &\widetilde{L}'' \arrow[r, "(\Stil_{p_0,\ldots, p_n}^{q_1,\ldots,q_n})^{-1}" ]& L'' \arrow[r] & L' 
\end{tikzcd} $$
and $h(\widetilde{L}'') < h(L'').$
Thus, cases \ref{casoquattro}, \ref{casocinque}, \ref{casosei} can be reduced to a composition of cases \ref{casouno}, \ref{casodue}, \ref{casotre} just by composing in-between the two moves with $(\Stil_{p_0,\ldots, p_n}^{q_1,\ldots,q_n})^{-1} \circ \Stil_{p_0,\ldots, p_n}^{q_1,\ldots,q_n}$.

Hence, let us check the first three cases.

\begin{enumerate}[wide, labelwidth=0pt, labelindent=0pt]
    \item Suppose that the moves are $(\Stil_{p_0,\ldots,p_m}^{q_1,\ldots,q_m})^{-1} $ and $\Stil_{r_0,\ldots,r_n}^{s_1,\ldots,s_n} $. If the two sawtooth moves are applied on the same edges, we already showed in \Cref{prop: sawtooth} that we can go from $L$ to $L'$ only with moves of type $\Rtil, \Wtil$. 
    
    Otherwise, by \Cref{lem: sawteeth} we can find an admissible sawtooth move $\Stil_{t_0,\ldots t_k}^{u_1,\ldots,u_k}$ with $t_0t_k=r_0r_n$ the intersections $[t_j,t_{j+1},u_{j+1}] \cap [p_i,p_{i+1},q_{i+1}]$ and $[t_j,t_{j+1},u_{j+1}] \cap \tau([p_i,p_{i+1},q_{i+1}])$ are either empty or contained in $L$. In particular $(\Stil_{p_0,\ldots,p_m}^{q_1,\ldots,q_m})^{-1} $ and $ \Stil_{t_0,\ldots t_k}^{u_1,\ldots,u_k}) $ commute and $$ \Stil_{r_0,\ldots,r_n}^{s_1,\ldots,s_n} \circ (\Stil_{p_0,\ldots,p_m}^{q_1,\ldots,q_m})^{-1}= (\Stil_{r_0,\ldots,r_n}^{s_1,\ldots,s_n}\circ(\Stil_{t_0,\ldots t_k}^{u_1,\ldots,u_k})^{-1}) \circ ((\Stil_{p_0,\ldots,p_m}^{q_1,\ldots,q_m})^{-1} \circ \Stil_{t_0,\ldots t_k}^{u_1,\ldots,u_k}).$$  
    Notice that every move on the right-hand side is admissible.
    Applying first $\Stil_{t_0,\ldots t_k}^{u_1,\ldots,u_k}$ we first decrease $h$ by $2$, then if we apply $(\Stil_{p_0,\ldots,p_m}^{q_1,\ldots,q_m})^{-1}$ we increase it again by $2$. In the end, by \Cref{prop: sawtooth} we can factorize $\Stil_{r_0,\ldots,r_n}^{s_1,\ldots,s_n}\circ (\Stil_{t_0,\ldots t_k}^{u_1,\ldots,u_k})^{-1}$ into $\Wtil$ and $\Rtil$ moves, that preserve $h$. Hence at every step $h$ is less than or equal to $h(L)$ and therefore it is strictly less than $h(L'')$.
    
    \item Suppose that the two moves are $\Etil_{a,c}^b$ and $\Stil_{p_0,\ldots,p_m}^{q_1,\ldots,q_m}$ and suppose that $bc$ is negative (the case where $bc$ is positive and $ab$ is negative is analogous). 

   By \Cref{lem: sawtooth} we can find an admissible sawtooth move $\Stil_{r_0,\ldots,r_n}^{s_1,\ldots,s_n} $ on $\Etil_{a,c}^b(L)$, with $r_0r_n=bc$ such that 
    $$\Stil_{p_0,\ldots,p_m}^{q_1,\ldots,q_m} \circ  \Etil_{a,c}^b = (\Stil_{p_0,\ldots,p_m}^{q_1,\ldots,q_m} \circ (\Stil_{r_0,\ldots,r_n}^{s_1,\ldots,s_n})^{-1}) \circ (\Stil_{r_0,\ldots,r_n}^{s_1,\ldots,s_n}\circ \Etil_{a,c}^b) $$
    and $\Stil_{p_0,\ldots,p_m}^{q_1,\ldots,q_m} \circ (\Stil_{r_0,\ldots,r_n}^{s_1,\ldots,s_n})^{-1}$ falls back into Case $(1)$. Hence we will now proceed to factorize $\Stil_{r_0,\ldots,r_n}^{s_1,\ldots,s_n}\circ \Etil_{a,c}^b $.

    For every $x \in \R^3$, let $r_x$ be the line parallel to $\bx$ containing $x$.
    We claim that we can assume the triangles $[r_i,r_{i+1},s_{i+1}]$ to intersect $[a,b,c]$ and $\tau([a,b,c])$ only inside $[b,c] \cup \tau([b,c])$. In fact, as in the proof of \Cref{lem: sawtooth} we can pick $r_0,\ldots,r_n\in [b,c]$ and $s_1,\ldots, s_n$ such that $\Stil_{r_0,\ldots,r_n}^{s_1,\ldots, s_n}$ is admissible and the triangles $[r_i,r_{i+1},s_{i+1}]$ do not intersect the convex set $ \tau([a,b,c])$ outside of $\tau([b,c]) \cap [b,c] $. These triangles will be generically disjoint from $[a,b,c] \setminus [b,c]$.

    Let $c'$ be close to $c$ so that $c'c$ is positive and notice that we can pick $c'$ so that $[a,b,c']\cap [r_i,r_{i+1},s_{i+1}]$ is either empty or equal to $\{b\}$ when $i=m-1$. One has to pick $c'$ close enough to $c$ and such that its projection on $\R^2$ lies in the same half-plane as $a$ with respect to the line containing $b$ and $c$. Then (see \Cref{fig:caso2})
    $$ \Stil_{r_0,\ldots,r_n}^{s_1,\ldots,s_n} \circ \Etil_{a,c}^b= (\Etil_{a,b}^{c'})^{-1}\circ\Wtil_{c',r_1}^{a,s_1}\circ \Wtil_{c',r_2}^{r_{1},s_2}\circ\cdots \circ \Wtil_{c',r_{n-1}}^{r_{n-2},s_{n-1}}\circ \Wtil_{c',c}^{r_{n-1},s_n} \circ \Etil_{a,c}^{c'}.$$
    \begin{figure}
        \centering
        \includegraphics[width= 0.65 \textwidth]{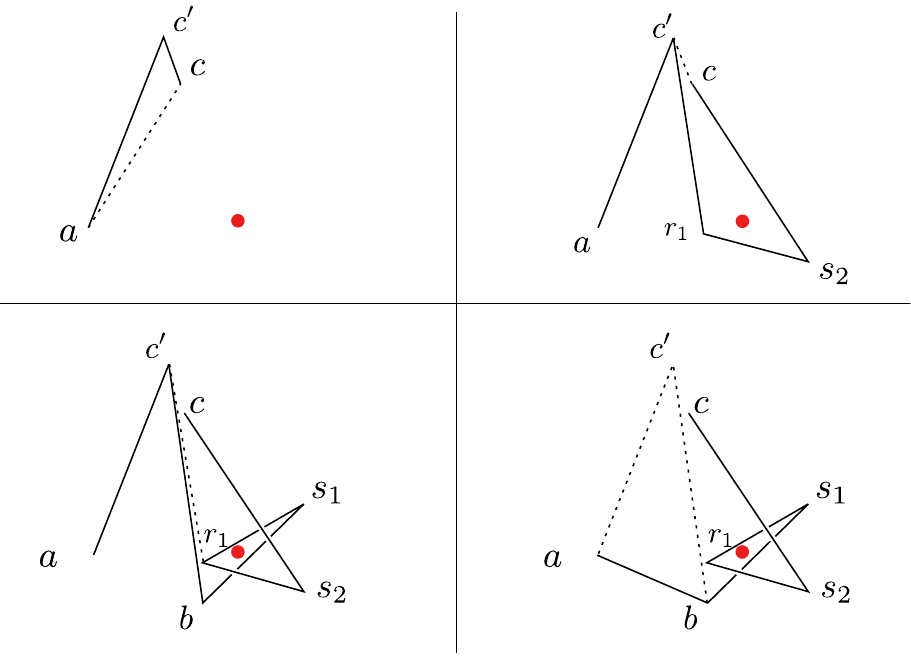}
        \caption{The factorization in Case $2$. In order to simplify the picture we only drew the effect of the sequence of moves on $[a,c]$ and not on $\tau([a,c])$.}
        \label{fig:caso2}
    \end{figure}
    Notice that if $c'$ is close enough to $c$ all the moves above are well-defined and admissible. Moreover, the move $\Etil_{a,c}^{c'} $ and all the $\Wtil$ moves do not increase $h$. 
    
    \item  
    Suppose that the two moves are $\D_{a,b}^c$ and $\Stil_{p_0,\ldots,p_m}^{q_1,\ldots,q_m}$. 
    Without loss of generality we can suppose that $[p_0,p_m]=[a,c]$. In fact, if that is not the case, by \Cref{lem: sawtooth} we can find an admissible sawtooth move $ \Stil_{t_0,\ldots,t_k}^{u_1,\ldots,u_k}$ on $[a,c]$ and $\tau([a,c])$ so that $$\Stil_{p_0,\ldots,p_m}^{q_1,\ldots,q_m} \circ  \D_{a,b}^c = (\Stil_{p_0,\ldots,p_m}^{q_1,\ldots,q_m} \circ (\Stil_{t_0,\ldots,t_k}^{u_1,\ldots,u_k})^{-1}) \circ (\Stil_{t_0,\ldots,t_k}^{u_1,\ldots,u_k}\circ \D_{a,b}^c) $$
    and $\Stil_{p_0,\ldots,p_m}^{q_1,\ldots,q_m} \circ (\Stil_{t_0,\ldots,t_k}^{u_1,\ldots,u_k})^{-1}$ falls back into Case $1$.
    
    Call $Q$ the quadrilateral with vertices $a,b,c,\tau(a)$.
    Notice that $Q\cap [p_i,p_{i+1},q_{i+1}] =[p_i,p_{i+1}]. $
    \begin{itemize}[wide, labelwidth=0pt, labelindent=0pt]
        \item Let us first suppose that $\D_{a,b}^c$ does not go through $\infty$.
    Let $a'$ and $c'$ be points close to $a$ and $c$ respectively, such that their projections in $\R^2$ belong to the image of $Q$ in $\R^2$, and such that $aa'$ and $c'c$ are positive.
    
    Then (see \Cref{fig:caso3})
    $$ \Stil_{p_0,\ldots,p_n}^{q_1,\ldots,q_n} \circ \D_{a,b}^c= (\Rtil_{q_m,c}^{c'})^{-1}\circ (\W_{c',c}^{a',b})^{-1} \circ \Wtil_{a',p_{m-1}}^{p_{m},q_{m}}\circ \Wtil_{a',p_{m-2}}^{p_{m-1},q_{m-1}}\circ \cdots \circ \Wtil_{a',p_1}^{p_2,q_2} \circ \Wtil_{a,a'}^{p_1,q_1}\circ \Rtil_{a,b}^{a'}.$$

    \begin{figure}
        \centering
        \includegraphics[width= 0.63 \textwidth]{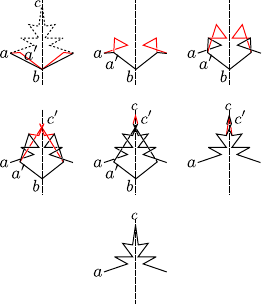}
        \caption{The factorization in Case $3$. Observe that to avoid overlaps in the picture and therefore make it more clear, not all edges have the correct sign according to the picture.}
        \label{fig:caso3}
    \end{figure}
    It is not too hard to see that $(\W_{c',c}^{a',b})^{-1}$ is well-defined and admissible since we picked $a'$ and $c'$ such that their projections in $\R^2$ belong to the image of $Q$ in $\R^2$. 
    
    Notice as well that all the other moves on the right-hand side are also well-defined and admissible for a generic choice of $a',c'$. Since $\Wtil, \W,\Rtil$ moves preserve the function $h$ this concludes the proof in this case.
 
 \item If the move $\D_{a,b}^c$ goes through $\infty$, let $\alpha$ be the plane containing $Q$. One has to notice that one can find $   \widetilde{a},\widetilde{c}$, inside $\alpha \setminus Q$  close to $a$ and $c$ respectively such that $a\widetilde{a}$ and $\widetilde{c}c$ are positive and such that $[\widetilde{a},b,c]\cup [\tau(\widetilde{a}),b,c]\supset Q$ and that $ [c,\widetilde{a},\widetilde{c}] \cap Q= \{c\}$.
 Then for a generic choice of $ a'$ and $c'$ close to $\widetilde{a}$ and $\widetilde{c}$ respectively we have that (see \Cref{fig:caso3infinito}),
 $$ \Stil_{\underline{p}}^{\underline{q}} \circ \D_{a,b}^c= (\Rtil_{q_m,c}^{c'})^{-1}\circ (\W_{c',c}^{a',b})^{-1} \circ \Wtil_{a',p_{m-1}}^{c',q_m}\circ \Wtil_{a',p_{m-2}}^{p_{m-1},q_{m-1}}\circ \cdots \circ \Wtil_{a',p_1}^{p_2,q_2} \circ \Wtil_{a,a'}^{p_1,q_1}\circ \Rtil_{a,b}^{a'},$$
 where $\underline{p}=p_0,\ldots,p_m$ and $ \underline{q}= q_1,\ldots,q_m$.
 \begin{figure}
     \centering
     \includegraphics[width= 0.65 \textwidth]{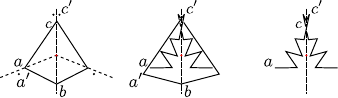}
     \caption{The factorization in Case $3$ when $\D$ goes through $\infty$. Observe that to avoid overlaps in the picture and therefore make it more clear, not all edges have the correct sign according to the picture.}
     \label{fig:caso3infinito}
 \end{figure}
 Notice that our choice of $a'$ and $c'$ implies that the moves on the right-hand side are well-defined and admissible.
 \end{itemize}
 \end{enumerate}
 This concludes the proof.
\end{proof}

We are now ready to prove the first version of Markov theorem for strongly involutive links.

\begin{figure}
    \centering
    \includegraphics[width=0.4 \textwidth]{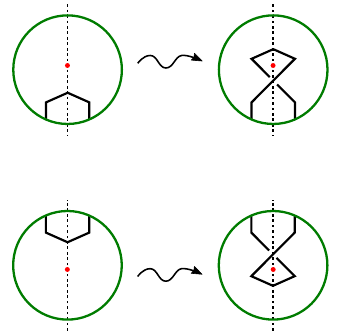}
    \caption{Stabilization on a fixed point. Also the mirror images have to be considered.}
    \label{fig:stabfix}
\end{figure}

\begin{figure}
    \centering
    \includegraphics[width=0.4 \textwidth]{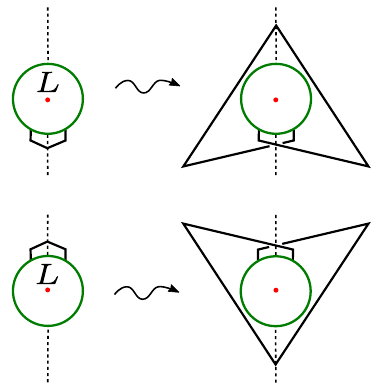}
    \caption{Stabilization on a fixed point through $\infty$. Also the mirror images have to be considered.}
    \label{fig:stabfixinfty}
\end{figure}

\begin{figure}
    \centering
    \includegraphics[width=0.4 \textwidth]{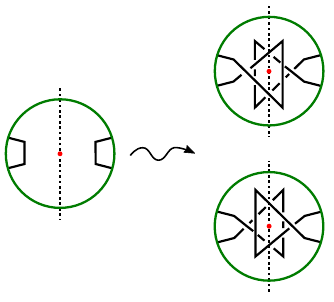}
    \caption{Double stabilization. Also the mirror images have to be considered.}
    \label{fig:stabdouble}
\end{figure}

\begin{thm} \label{thm: markovlink}
    Let $L$ and $L'$ be two strongly involutive links so that $h(L)=h(L')=0$ with respect to the braid axis $\bx$. Then $L$ and $L'$ are equivariantly isotopic if and only if they are related by a finite sequence of the following:
    \begin{itemize}
        \item equivariant isotopies in the exterior of $\bx$ that preserve $h$;
        \item stabilization on a fixed point (as described in \Cref{fig:stabfix});
        \item stabilization on a fixed point through $\infty$ (as described in \Cref{fig:stabfixinfty});
        \item double stabilization (as described in \Cref{fig:stabdouble}).
    \end{itemize}
\end{thm}

\begin{proof}
    As a consequence of \Cref{prop: isolatemax} and \Cref{prop: reducemin} $L$ and $L'$ are related by a sequence of $\Rcal,\Rtil,\W,\Wtil$ moves, that is to say the only moves in our set that preserve $h(L)=0$. Due to \Cref{prop: empty} and \Cref{prop: really empty} we can suppose that the $\W$ and $\Wtil$ are really empty. Moreover we can suppose the $\W$ moves that do not go through $\infty$ to be of type A, the $\W$ moves that do go through $\infty$ to be of type B and the $\Wtil$ moves to be of crossing type.

    One has to notice that $\Rcal$ and $\Rtil$ moves are equivariant isotopies performed in the exterior of $\bx$, and they preserve $h$. Moreover the really empty $\Wtil$ moves of crossing type coincide with double stabilizations as in \Cref{fig:stabdouble}, really empty $\Wtil$ moves that do not go through $\infty$ of type A are stabilizations on a fixed points as in \Cref{fig:stabfix}, while $\W$ moves through $\infty$ are stabilization on a fixed point through $\infty$ as in \Cref{fig:stabfixinfty}.
\end{proof}

\begin{rmk}
    Let $r:S^3\to S^3 $ be the diffeomorphism that, when restricted to $\R^3$ is the rotation of $\pi$ around $\bx$. Notice that $ r$ and $\tau$ commute and that $ r$ reverses the orientation on $\ax$.
    As a consequence of the proof of Proposition 2.4 in \cite{LobbWatson}, two strongly involutive links $(L,\tau)$ and $ (L',\tau)$ are Sakuma equivalent if and only if  $(L,\tau)$ and either $ (L',\tau)$ or $(r(L'),\tau )$ are equivariantly isotopic. Hence, in the hypotheses of \Cref{thm: markovlink}, $L$ and $L'$ are Sakuma equivalent if and only if $L$ and either $L'$ or $r(L')$ are related by a finite sequence of the operations described in \Cref{thm: markovlink}.
\end{rmk}

\subsection{Step 4: proof of the Equivariant Markov Theorem} \label{subsec: step4}
We now want to give an algebraic version of \Cref{thm: markovlink}. We will do this in \Cref{thm: eqmarkov}, but first we have to show how $\Rcal$ and $\Rtil$ moves act on braids.

\begin{dfn}
    Given two palindromic $n$-braids $\alpha$ and $\beta$, and an $n$-braid $\gamma$, then we say that the pairs $(\alpha, \beta)$ and $(\gamma \alpha \overline{\gamma},\overline{\gamma}^{-1}\beta \gamma^{-1} ) $ are equivariantly conjugated.
    Notice that the second is still a pair of palindromic braids.
    
\end{dfn}

\begin{rmk}
    \label{rmk: nonuniquebraids}
    Recall that by \Cref{rmk: braidclosure}, if $L$ is such that $h(L)=0$ up to equivariant isotopy that modifies the link diagram only up to planar isotopy, we can suppose that $L$
 is the equivariant closure of two palindromic braids. These braids are not unique since one can, up to equivariant planar isotopy, move some of the crossings from $\alpha$ to $\beta$.
 However all possible choices of such a pairs are easily seen to be equivariantly conjugated.
 \end{rmk}

\begin{prop} \label{prop: Rconj}
Let $L$ and $L'$ be two strongly involutive links with $h(L)=h(L')=0$. Suppose that there is a finite sequence of admissible $\mathcal{R}$ and $\widetilde{\mathcal{R}}$ move that transform $L$ into $L'$. If $L$ is equivariantly isotopic in the sense of \Cref{rmk: braidclosure} to the equivariant closure of $(\alpha, \beta)$ and $L'$ to the equivariant closure of $(\alpha' , \beta')$, then the two pairs are equivariantly conjugated.
\end{prop}

\begin{proof}
Let $\pi:\R^3\to \R^2$ be the orthogonal projection on the $(x,y)$-plane. 
One can suppose up to a small equivariant planar isotopy that no vertex of $L$ projects to the $x$-axis, and that there are no crossings on the $x$-axis. From this projection there is a canonical choice of palindromic braids $\alpha$ and $\beta$ such that $L = \widehat{\alpha\beta}$. One should think of $\alpha$ as the palindromic braid made of the strands that project via $\pi$ to the half-plane $\{y\le 0\}$, and of $\beta$ as the one made of the strands that project via $\pi$ to the half-plane $\{y\ge 0\}.$
\begin{itemize}[wide, labelwidth=0pt, labelindent=0pt]
    \item Let us deal with the case of the move $\Rtil_{a,c}^b$. 

    By \Cref{lem: reduceRtil}, up to equivariant planar isotopy we can suppose that $\pi([a,b,c])$ is completely contained in one of the two half-planes, let us suppose $\{y \le 0\}$, the other case being symmetrical.
    By \Cref{rmk: nonuniquebraids}, we know that this planar isotopy can change the braids $(\alpha, \beta)$ by equivariant conjugation.
        Now $\Rtil_{a,c}^b$ acts only on $\alpha$, changing its representation as a palindromic braid.
     For an example, see \Cref{fig:braidrelations}. This concludes the proof in this case.

     \begin{figure}
         \centering
         \includegraphics[width=0.75 \textwidth]{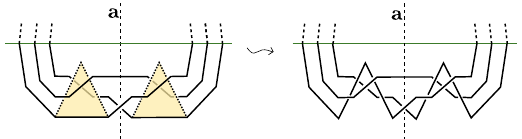}
         \caption{A move $\Rtil$ acting on $\alpha$, where the green horizontal line represents the $x$-axis. Of course the braid on the left and the one on the right are equal, since they are isotopic relative to their boundary.}
         \label{fig:braidrelations}
     \end{figure}

\item Consider now the move $\Rcal_{a,c}^b$. By \Cref{lem: reduceR} and the previous case, we can suppose that the quadrilateral $Q$ with vertices $a,b,c, \tau(b)$ is small. Since $Q \cap \bx = \emptyset$, we can suppose that $\pi(Q)$ does not intersect the $x$-axis, and is therefore contained in one of the two half-spaces, suppose $ \{y \le 0\}$. 
Now $\Rcal_{a,c}^b$ acts only on $\alpha$, changing its representation as a palindromic braid.
\begin{figure}
    \centering
    \includegraphics[width=0.4 \textwidth]{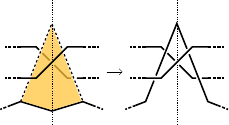}
    \caption{A move $\Rcal$ acting on $\alpha$. Of course the braid on the left and the one on the right are equal, since they are isotopic relative to their boundary.}
    \label{fig:Rcross}
\end{figure}
\end{itemize}
This concludes the proof.
\end{proof}

Now that we identified $\Rcal$ and $\Rtil$ moves with equivariant conjugation on pairs of palindromic braids, we are left with determining how the various types of stabilizations act on the braids.

Let $B_n$ denote the group of $n$-braids. Consider the two different inclusions of $B_n$ in $B_{n+1}$:
\begin{enumerate}
    \item $E_n: B_n \ni \sigma_i \mapsto \sigma_i \in B_{n+1};$
    \item $S_n: B_n \ni \sigma_i \mapsto \sigma_{i+1}$;
\end{enumerate}
where the $\sigma_i$'s are the standard generators of the Artin braid group. To simplify notation, let us call $S_n$ just $S$ and $E_n$ just $E$ for every $n$.

\begin{dfn}\label{dfn: Wstab}
Let $(\alpha,\beta)$ be a pair of palindromic braids, such that $\alpha = x \Delta \overline{x}$ and $\beta= y \Delta' \overline{y}$ as observed in \Cref{rmk: palindromicity}. Suppose that $ \Delta =\sigma_{i_1}^{\varepsilon_1} \cdots \sigma_{i_n}^{\varepsilon_n}$ and $ \Delta' =\sigma_{j_1}^{\eta_1} \cdots \sigma_{j_k}^{\eta_k}$.
    \begin{itemize}[wide, labelwidth=0pt, labelindent=0pt]
        \item If $ \sigma_1 \notin \{ \sigma_{i_1},\cdots, \sigma_{i_n}\}$ (respectively $ \sigma_1 \notin \{ \sigma_{j_1},\cdots, \sigma_{j_k}\})$, we call \emph{stabilization of type S on a fixed point} the operation that substitutes $(\alpha,\beta) $ with $(S(x)S(\Delta)\sigma_1^{\pm 1} S(\overline{x}), S(\beta))$ (respectively $(S(\alpha), S(y)S(\Delta') \sigma_{1}^{\pm 1} S(\overline{y}))$). 

        \item If $ \sigma_{n} \notin \{ \sigma_{i_1},\cdots, \sigma_{i_n}\}$ (respectively $ \sigma_n \notin \{ \sigma_{j_1},\cdots, \sigma_{j_k}\})$, we call \emph{stabilization of type E on a fixed point} the operation that substitutes $(\alpha,\beta) $ with $(E(x)E(\Delta)\sigma_{n+1}^{\pm 1} E(\overline{x}), E(\beta))$ (respectively $(E(\alpha), E(y)E(\Delta') \sigma_{n+1}^{\pm 1} E(\overline{y}))$). 

        \item We call \emph{double stabilization} the operation that substitutes the pair $(\alpha,\beta) $ with the pair $(S^2(x) S^2(\Delta) \sigma_1^{\varepsilon} S^2(\overline{x}), \sigma_2^\eta S^2(y) S^2(\Delta) \sigma_1^{-\varepsilon} S^2(\overline{x})\sigma_2^\eta) $
        where $\varepsilon, \eta = \pm 1$ (see \Cref{fig:doubledest}).
    \end{itemize}
\end{dfn}

\begin{figure}
    \centering
    \includegraphics[width=0.55 \textwidth]{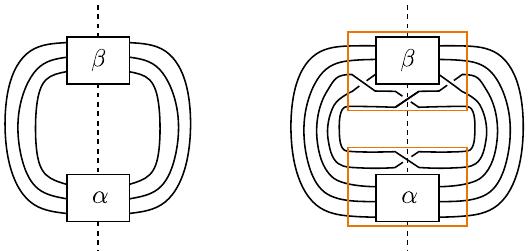}
    \caption{Double stabilization on two palindromic braids. Notice that we drew a smooth picture instead of a PL one only to make it more clear.}
    \label{fig:doubledest}
\end{figure}

Notice that all these operations preserve the palindromicity of the two braids of the pair.

\begin{rmk}\label{rmk:stabilizations}
  Notice that a stabilization on a fixed point as in \Cref{fig:stabfix}, acts on $(\alpha,\beta)$ exactly as a stabilization of type S on a fixed point. 
  
  A stabilization on a fixed point through $\infty$ as in \Cref{fig:stabfixinfty}, acts on $(\alpha,\beta)$ as a stabilization of type E on a fixed point.

  Furthermore, a double stabilization as in \Cref{fig:stabdouble} acts on $(\alpha,\beta)$ as a double stabilization as defined in \Cref{dfn: Wstab} (see \Cref{fig:equaldoublestab}). 
\end{rmk}

\begin{figure}
    \centering
    \includegraphics[width= 0.4 \textwidth]{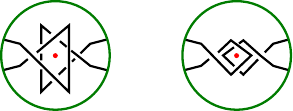}
    \caption{Notice that the double stabilization shown in \Cref{fig:stabdouble} (left) and the stabilization defined in \Cref{dfn: Wstab} (right) act on the same way on the braids.}
    \label{fig:equaldoublestab}
\end{figure}

As a consequence of \Cref{prop: Rconj} and \Cref{rmk:stabilizations} we can give an algebraic version of \Cref{thm: markovlink}.

\begin{thm}[Equivariant Markov Theorem]\label{thm: eqmarkov}
    Suppose that $\alpha,\beta$ are palindromic $n$-braids and $\gamma,\delta$ are palindromic $m$-braids. Suppose that  $(\widehat{\alpha\beta},\tau)$ and $(\widehat{\gamma\delta},\tau) $ are equivariantly isotopic. Then the pair $(\gamma,\delta)$ can be obtained by the pair $(\alpha,\beta)$ by a finite sequence of the following operations:
    \begin{itemize}
        \item equivariant conjugation;
        \item (de-)stabilization of type $S$ on a fixed point;
        \item (de-)stabilization of type $E$ on a fixed point;
        \item double stabilization. 
    \end{itemize}
\end{thm}

\begin{rmk}
It is a straightforward consequence of Proposition 2.4 in \cite{LobbWatson} and \Cref{thm: eqmarkov} that if $(\widehat{\alpha\beta},\tau) $ and $(\widehat{\alpha'\beta'},\tau) $ are Sakuma equivalent, then either $(\alpha, \beta)$ and $(\alpha',\beta')$ or $(\beta, \alpha)$ and $(\alpha',\beta')$ are related by the operations in \Cref{thm: eqmarkov}.
\end{rmk}
\newpage

\appendix 

\section{Equivalence of smooth and piecewise linear strongly involutive links} \label{app: equiv}
In this appendix we want to show that smooth strongly involutive links up to \emph{smooth Sakuma equivalence} are actually equivalent to PL strongly invertible links up to \emph{PL Sakuma equivalence}. In order to do that we will start by showing that smooth strongly involutive links up to \emph{smooth equivariant isotopy} are actually equivalent to PL strongly invertible links up to \emph{PL equivariant isotopy}.

As a consequence of the results in this appendix and \cite[Proposition 2.4]{LobbWatson}, we have that two Sakuma equivalent strongly involutive PL links are actually equivariantly isotopic, up to a global move that reverses the orientation on the axis.

Moreover, these results will imply that all the results in this paper apply to smooth strongly involutive links as well.

Suppose that $\mathcal{L}$ is a smooth strongly involutive link, with involution $\tau$ the $\pi$-rotation around a fixed axis in $\R^3$ that fixes $\infty$. Take a sufficiently small $\varepsilon>0$, so that $V_\varepsilon(\mathcal{L})$, the $\varepsilon$-neighbourhood of $\mathcal{L}$, is a regular neighbourhood of $\mathcal{L}$. It is not hard to see that $V_\varepsilon(\mathcal{L})$ is $\tau$-invariant. We can approximate $\mathcal{L}$ via a PL link $L$ inside $V_\varepsilon(\mathcal{L})$, so that the vertices of $L$ lie on $\mathcal{L}$ and the projection of $V_\varepsilon(\mathcal{L})$ onto $\mathcal{L}$ sends $L$ homeomorphically to $\mathcal{L}$. It is not hard to show that we can take this approximation to be $\tau$-equivariant, and that any such approximation is equivariantly isotopic. We call $\operatorname{PL}(\mathcal{L})$ the class of $L$ up to equivariant isotopy. 

\begin{prop}[Equivalence of smooth and PL strongly involutive links up to equivariant isotopy] \label{prop: eq iso smooth pl}
The operator $\operatorname{PL}$, that assigns to every smooth strongly involutive link the equivariant isotopy class of a PL approximation induces a well-defined isomorphism between the set of smooth strongly involutive links up to smooth equivariant isotopy to the set of PL strongly involutive links up to PL equivariant isotopy.
\end{prop}

\begin{proof}
Let $\mathcal{L}$ and $\mathcal{L'}$ be two smooth strongly involutive links that are equivariantly isotopic via the isotopy $H: S^3 \times [0,1] \to S^3$. Up to subdividing the isotopy in a concatentation of isotopies, we can always suppose that there exists a point $p$ in the fixed point set of $\tau$ that does not belong to $H(\mathcal{L}, [0,1])$. Let us suppose without loss of generality that that point is $\infty$. 
By compactness, there exists $\varepsilon>0$ such that the $\varepsilon$-neighbourhood of $\mathcal{L}_t\coloneqq H(\mathcal{L}, t)$ is a regular neighbourhood of $\mathcal{L}_t$ for every $t \in [0,1]$. 

Let $0=t_0 < t_1<\ldots<t_n =1$ such that for $i=0,\ldots,n-1$ and $t\in [t_i, t_{i+1}]$, $\mathcal{L}_t \subset V_\varepsilon(\mathcal{L}_{t_{i}})$ and $\mathcal{L}_t$ is transverse to the fibers of the projection of $ V_\varepsilon(\mathcal{L}_{t_{i}})$ to $\mathcal{L}_{t_{i}} $. Then one can take good enough PL approximations $L_i$ and $L_{i+1}$ of $ \mathcal{L}_{t_{i}}$ an $\mathcal{L}_{t_{i+1}}$ respectively so that $L_i$ and $L_{i+1}$ are contained in $V_\varepsilon(\mathcal{L}_{t_{i}})$ and project homeomorphically to $\mathcal{L}_{t_{i}}$. 

In this setting, it is easy to see that $L_i$ and $L_{i+1}$ are PL equivariantly isotopic. Hence, since $\mathcal{L}= \mathcal{L}_{t_0} $ and $ \mathcal{L'}= \mathcal{L}_{t_0}$, one shows that $\operatorname{PL}(\mathcal{L})$ equals $\operatorname{PL}(\mathcal{L'})$, hence $\operatorname{PL}$ induces a well-defined map between the set of smooth strongly involutive links up to smooth equivariant isotopy to the set of PL strongly involutive links up to PL equivariant isotopy.

This map is clearly surjective, while to discuss injectivity one just needs to realise that if $L$ and $L'$, PL approximations of $\mathcal{L}$ and $\mathcal{L'}$ respectively, are related by admissible equivariant triangle moves, then $\mathcal{L}$ and $\mathcal{L'}$ are clearly smoothly equivariant isotopic, since we can use the triangles of the moves to perform an isotopy. This, together with \Cref{lem: equivalence}, concludes the proof.
\end{proof}

As a consequence of \cite[Proposition 2.4]{LobbWatson}
and \Cref{prop: eq iso smooth pl}, the operator PL induces a well-defined surjective map from the set of smooth strongly involutive links up to smooth Sakuma equivalence and the set of PL strongly involutive links up to PL Sakuma equivalence. The following lemma will allow us to show the injectivity of this induced map.

\begin{lem} \label{lem: smooth approx}
Let $f:S^3 \to S^3$ be a PL homeomorphism that commutes with $\tau$. Then, if we fix a Riemannian metric on $S^3$ so that $\tau$ is an isometry, for every $\varepsilon>0$, there exists a smooth diffeomorphism $\widetilde{f}:S^3 \to S^3$ that commutes with $\tau$ and such that for every $x \in S^3$, $d(f(x),\widetilde{f}(x))< \varepsilon$, where $d$ is the distance induced by the Riemannian metric.
\end{lem}

\begin{proof}
    By \cite{Munkres} (see Theorem 6.3), for every $\delta>0$, $f$ can be $\delta$-approximated by a smooth diffeomorphism $h:S^3 \to S^3$. Let $\tau'$ be the involution $ h \circ \tau \circ h^{-1}$. Notice that for all $x \in S^3$, $d(\tau(x), \tau'(x))< 2 \delta$. In fact, by the triangular inequality, $$d(\tau(x), \tau'(x))\le d(\tau(x), f \circ \tau \circ h^{-1}(x))+ d(f \circ \tau \circ h^{-1}(x), \tau'(x)) $$ and, since $h$ approximates $f$, and $\tau$ commutes with $f$ and is an isometry, one sees that $ d(\tau(x), f \circ \tau \circ h^{-1}(x))$ is equal to $d(h \circ h^{-1}(x), f \circ  h^{-1}(x))$ that is less than $\delta$ and that $d(f \circ \tau \circ h^{-1}(x), \tau'(x))$ is also less than $ \delta$. 
    
    Let $U$ denote the fixed point set of $\tau$, and $D$ a disk in $S^3$ such that $\tau(D) \cap D= \partial D= U$. The unknot $h(U)$ is the fixed point set of $\tau'$, and, since $f$ fixes $U$ setwise, $h(U)$ is contained in a $ \delta$-neighbourhood of $U$. 
    
    Notice that $h \circ f^{-1}$ is a $\delta$-approximation of the identity map, i.e. for every $x \in S^3$,  $ d(x, h \circ f^{-1}(x))< \delta$. The boundary of $h\circ f^{-1}(D)$ is $h \circ f^{-1}(U)=h(U)$, since $f^{-1}$ commutes with $\tau$. Let $D'$ be a smooth disk $\delta$-approximating the disk $ h\circ f^{-1}(D)$, relative to the boundary, i.e. there exists a homeomorphism $g':D' \to h\circ f^{-1}(D) $ such that $d(x,g'(x)) < \delta$. Then there exists a diffeomorphism $g:D' \to D$, approximating the homeomorphism $ f \circ h^{-1}\circ g':D' \to D$, such that $d(g(x),x) < 3 \delta$. One can see that if $D'$ is picked close enough to $ h\circ f^{-1}(D)$, then $D' \cap \tau'(D')= \partial D'$.
    
    Hence, the smooth diffeomorphism $g$ can be extended uniquely to a smooth diffeomorphism $D' \cup \tau' (D') \to D \cup \tau (D)$ that we still denote by $g$, such that $g(\tau'(x))=\tau(g(x))$.  Now $D' \cup \tau' (D') $ and $D \cup \tau(D)$ are $2$-spheres that separate $S^3$ in two balls that are exchanged by $\tau'$ and $\tau$ respectively. 
    
    One can extend $g$ on one of the two balls smoothly, so that $ g(x)$ is distant from $x$ at most $3 \delta$, and then uniquely extend it to the whole of $S^3$ so that $ g\circ\tau=\tau\circ g$. Now $g$ is a $3\delta$-approximation of the identity map on $S^3$, and $ \tau= g \circ \tau' \circ g^{-1}$. 
    
    To conclude one just has to notice that $\widetilde{f} \coloneqq g \circ h$ commutes with $\tau$ and is a $4 \delta$-approximation of $f$. For $\delta$ sufficiently small, such that $ 4\delta<\varepsilon$, this concludes the proof. 
\end{proof}

\begin{cor}[Equivalence of smooth and PL strongly involutive links up to Sakuma equivalence] \label{prop: Sak smooth pl}
The operator $\operatorname{PL}$ induces a well-defined isomorphism between the set of smooth strongly involutive links up to smooth Sakuma equivalence to the set of PL strongly involutive links up to PL Sakuma equivalence.
\end{cor}

\begin{proof}
    The map PL is clearly surjective. To show injectivity, one has to notice that if $L$ and $L'$ are PL-approximations of $\mathcal{L}$ and $\mathcal{L}'$ respectively, and there exists a $\tau$-equivariant PL homeomorphism $f:S^3 \to S^3$ mapping $L$ to $L'$, then for $\varepsilon>0$ sufficiently small, the map $\widetilde{f}$ from \Cref{lem: smooth approx} sends $\mathcal{L}$ in an equivariant tubular neighbourhood of $ \mathcal{L}'$, and the two links are equivariantly isotopic.
\end{proof}

\section{Table of some strongly invertible knots in equivariant braid closure form} \label{app: examples}
We provide a short list of examples of low crossing number strongly invertible knots in equivariant braid closure form. The strong involutions are given by rotation around
a vertical axis. Mirrors are not included in the tables, since they can be obtained by just inverting every crossing. Notice that when a knot has more than one non-equivalent strong inversion we distinguish the different representatives by adding a letter to the name of the knot (e.g. $5_2a$ and $5_2b$).

\vspace{0.5cm}
\begin{center}
\begin{tabular}{|c|c|c|}
\hline
Knot    & Diagram & Pair of palindromic braids \\
 \hline
 $3_1$ & \begin{minipage}{0.21\textwidth} \centering
 \vspace{0.7mm}
     \includegraphics[width=0.97 \textwidth]{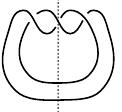}
 \end{minipage}    &  $(\id, \sigma_1^3)$\\
 \hline 
 $4_1$ & \begin{minipage}{0.21\textwidth} \centering \vspace{0.7mm}
     \includegraphics[width=0.97 \textwidth]{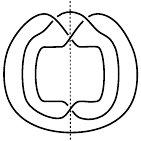} 
     \end{minipage}&  $(\sigma_1^{-1}, \sigma_2\sigma_1^{-1}\sigma_2)$ \\
 \hline
 $5_1$ &\begin{minipage}{0.21\textwidth} \centering \vspace{0.7mm}
     \includegraphics[width=0.9 \textwidth]{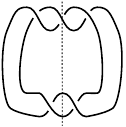} 
     \end{minipage} & $(\sigma_1^2,\sigma_1^3)$\\
 \hline

 $5_2a$ &\begin{minipage}{0.21\textwidth} \centering \vspace{0.7mm}
     \includegraphics[width=0.97 \textwidth]{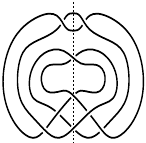} 
     \end{minipage} & $(\sigma_3\sigma_2^{-1}\sigma_1\sigma_3^{-1}\sigma_2^{-1}\sigma_3, \sigma_3\sigma_1\sigma_3 )$\\
 \hline
  $5_2b$ & \begin{minipage}{0.21\textwidth} \centering \vspace{0.7mm}
     \includegraphics[width=0.97 \textwidth]{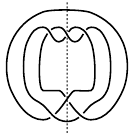} 
     \end{minipage}& $(\sigma_2\sigma_1^{-1}\sigma_2, \sigma_1^{3})$ \\
     \hline
      \end{tabular} 
      \newpage
\begin{tabular}{|c|c|c|}
\hline
Knot    & Diagram & Pair of palindromic braids \\
 \hline
     $6_1a$ & \begin{minipage}{0.27\textwidth} \centering \vspace{0.7mm}
     \includegraphics[width=0.97 \textwidth]{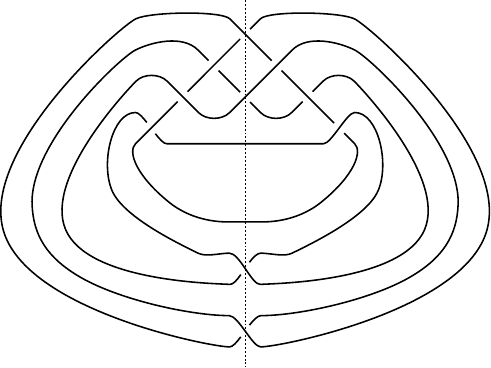} 
     \end{minipage}& $(\sigma_2\sigma_4,\sigma_1^{-1}\sigma_2\sigma_3^{-1}\sigma_2^{-1}\sigma_4\sigma_2^{-1}\sigma_3^{-1}\sigma_2\sigma_1^{-1})$\\
     \hline
     $6_1b$ & \begin{minipage}{0.25\textwidth} \centering \vspace{0.7mm}
     \includegraphics[width=0.97 \textwidth]{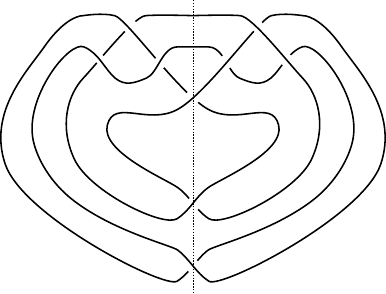} 
     \end{minipage}& $(\sigma_1^{-1}\sigma_3,\sigma_2\sigma_3\sigma_2^{-1}\sigma_1^{-1}\sigma_2^{-1}\sigma_3\sigma_2)$ \\
     \hline
     $6_2a$ & \begin{minipage}{0.25\textwidth} \centering \vspace{0.7mm}
     \includegraphics[width=0.97 \textwidth]{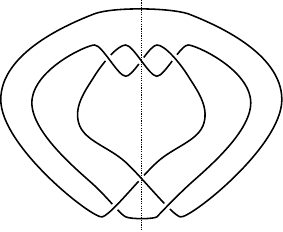} 
     \end{minipage}& $(\sigma_2^{-1}\sigma_1\sigma_2^{-1}, \sigma_1^3)$  \\
     \hline
     $6_2b$ & \begin{minipage}{0.25\textwidth} \centering \vspace{0.7mm}
     \includegraphics[width=0.97 \textwidth]{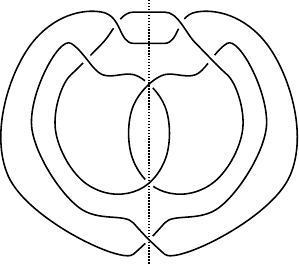} 
     \end{minipage}& $(\sigma_1^{-1}\sigma_3,\sigma_2\sigma_3\sigma_1\sigma_3\sigma_2)$ \\
     \hline
     $6_3$ & \begin{minipage}{0.25\textwidth} \centering \vspace{0.7mm}
     \includegraphics[width=0.97 \textwidth]{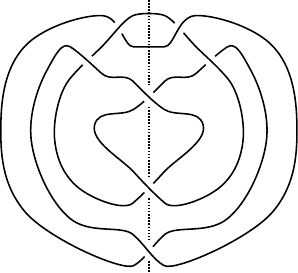} 
     \end{minipage}& $(\sigma_1\sigma_3^{-1}, \sigma_2\sigma_3^{-1}\sigma_1\sigma_3^{-1}\sigma_2)$ \\
     \hline
\end{tabular}
\end{center}

\bibliographystyle{alpha}
\bibliography{bibliografia}

\end{document}